\documentclass[reqno, 10pt]{amsart}
\usepackage{amssymb,amsmath,amsthm,}
\usepackage{a4wide}
\usepackage{amscd}
\usepackage{amsfonts}
\usepackage{amssymb}
\usepackage{latexsym}
\usepackage{color}
\usepackage{esint}
\usepackage{graphicx}
\usepackage{caption}
\usepackage{subcaption}
\usepackage{amsthm}

\usepackage[left=.7 in, right=.7 in,top=.7 in, bottom=.7 in]{geometry}

\newtheorem{theorem}{Theorem}

\newtheorem{corollary}{Corollary}
\newtheorem{definition}{Definition}
\newtheorem{lemma}{Lemma}
\newtheorem{assumption}{Assumption}
\newtheorem{proposition}{Proposition}
\newtheorem{remark}{Remark}

\let\e=\varepsilon

\let\p=\partial

\let\O=\Omega

\let\o=\omega

\let\b=\beta

\newcommand{\R}{\mathbb{R}}

\newcommand{\be}{\begin{equation}}
\newcommand{\bm}{\begin{multline}}
\newcommand{\ee}{\end{equation}}
\newcommand{\dd}{\mathrm{d}}

\newcommand{\xb}{x_{\mathbf{b}}}
\newcommand{\tb}{t_{\mathbf{b}}}
\newcommand{\vb}{v_{\mathbf{b}}}

\newcommand{\xf}{x_{\mathbf{f}}}
\newcommand{\tf}{t_{\mathbf{f}}}
\newcommand{\vf}{v_{\mathbf{f}}}
\newcommand{\VN}{\mathbb{V}^{N}}
\newcommand{\X}{\mathbf{x}}
\newcommand{\V}{\mathbf{v}}
\newcommand{\dist}{\mathrm{dist}}

\newcommand{\Bes}{\begin{eqnarray*}}
	\newcommand{\Ees}{\end{eqnarray*}}
\newcommand{\Be}{\begin{equation}}
\newcommand{\Ee}{\end{equation}}

\numberwithin{equation}{section}

\def\p{\partial}

\def\O{\Omega}
\def\R{\mathbb{R}}

\def\B{\begin{equation}}
\def\E{\end{equation}}
\def\BN{\begin{eqnarray*}}
\def\EN{\end{eqnarray*}}

\begin{document}
	\date{ \today
	}
	
	\title{DECAY OF THE BOLTZMANN EQUATION WITH THE SPECULAR BOUNDARY CONDITION IN NON-CONVEX CYLINDRICAL DOMAINS }

	\author{Chanwoo Kim \and Donghyun Lee}

	\begin{abstract}
		A basic question about the existence and stability of the Boltzmann equation in general non-convex domain with the specular reflection boundary condition has been widely open. In this paper, we consider cylindrical domains whose cross sections are general non-convex analytic planar domain. We establish the global-wellposedness and asymptotic stability of the Boltzmann equation with the specular reflection boundary condition in such domains. Our method consists of sharp classification of billiard trajectories which bounce infinitely many times or hit the boundary tangentially at some moment, and a delicate construction of an $\e$-tubular neighborhood of such trajectories. Analyticity of the boundary is crucially used. Away from such $\e$-tubular neighborhood, we control the number of bounces of trajectories and its' distance from singular sets in a uniform fashion. The worst case, sticky grazing set, can be excluded by cutting off small portion of the temporal integration. Finally we apply a method of \cite{KimLee} by the authors and achieve a pointwise estimate of the Boltzmann solutions.  
	\end{abstract}
		
	\maketitle
	
	\tableofcontents
			
	\section{Introduction}
	
	\textit{The Boltzmann equation} is a mathematical model for dilute gas which describes a probability density function of particles. In addition to free transport of a particle, a collision effect is also considered. If there is no external force or self-generating force, 
	 probability density function $F(t,x,v)$ is governed by
	\begin{equation} \label{Boltzmann}
		\p_{t}F + v\cdot\nabla_{x} F = Q(F,F),\quad    F(0,x,v)= F_{0}(x,v),
	\end{equation}
	where 
	the position $x\in U \subset \mathbb{R}^{3}$ and velocity $v\in\mathbb{R}^{3}$ at time $t \geq 0$. The collision operator $Q(F_1,F_2)$ takes the form of
	\begin{equation*} \label{collision Q}
		Q(F_{1}, F_{2}) = \int_{\mathbb{R}^{3}} \int_{\mathbb{S}^{2}} B(v-u,\o) \big[ F_{1}(u^{\prime})F_{2}(v^{\prime}) - F_{1}(u)F_{2}(v) \big] d\o du,\quad 
	\end{equation*} 
	where 
	$u^{\prime} = u + ((v-u)\cdot\o)\o$, $v^{\prime} = v - ((v-u)\cdot\o)\o$. For collision kernel, we choose so-called the hard sphere model $B(v-u,\o) = |(v-u)\cdot\o|$. We study (\ref{Boltzmann}) when $F$ is near the \textit{Maxwellian} $\mu = e^{-\frac{|v|^{2}}{2}}$. 
	
	When the gas contacts with the boundary, we need to impose boundary condition for $F$ on $\p U$, the boundary of the domain $U$. 
	In this paper, we impose \textit{the specular reflection boundary condition}, which is one of the most basic conditions
	\begin{equation} \label{specular}
		F(t,x,v) = F(t,x,R_{x}v),\quad x\in\p U,
	\end{equation}
	where $R_{x} := I - 2\mathbf{n}(x)\otimes \mathbf{n}(x)$ and $\mathbf{n}(x)$ is the outward unit normal vector at $x\in\p U$. Note that the Maxwellian is an equilibrium state (or a steady solution) of (\ref{Boltzmann}) with (\ref{specular}).
	
	\vspace{2pt}

	Despite extensive developments in the study of the Boltzmann theory, many basic boundary problems, especially regarding the specular reflection BC with general domains, have remained open. In 1977, in \cite{SA}, Shizuta and Asano  announced the global existence of the Boltzmann equation with the specular boundary condition in smooth convex domain without a complete proof. The first mathematical proof of such problem was given by Guo in \cite{Guo10}, but with a strong extra assumption that the boundary should be a level set of a real analytic function. Very recently the authors proved the unique existence and asymptotic stability of the specular boundary problem for general smooth convex domains (with or without external potential) in \cite{KimLee}, using triple iteration method and geometric decomposition of particle
	trajectories. This marks a complete resolution of a 40-years open question after \cite{SA}. 
	
  Meanwhile, there were even fewer results for general non-convex domains with the specular boundary condition. An asymptotic stability of the global Maxwellian is established in \cite{DV}, provided certain \textit{a-priori} strong Sobolev estimates can be verified. However, such strong estimates seem to fail especially when the domain is non-convex (\cite{GKTT1, GKTT2, Kim11}). Actually we believe that the solution cannot be in $C^1$ (but in $C^{0,\alpha}$) when the domain is non-convex. To the best of our knowledge, our work is the first result on the global well-posedness and decay toward Maxwellian results for any kind of non-convex domains with the specular boundary condition! One of the intrinsic difficulties of the non-convex domain problem is the (billiard) trajectory is very complicated to control (e.g. infinite bouncing, grazing).  
  The problems of general smooth non-convex domains or three-dimensional non-convex domains are still open.

	\vspace{2pt}
	
	In the case of the specular reflection boundary condition, we have the total mass and energy conservations as
	\begin{equation} \label{conserv_F_mass}
	\iint_{U \times\R^{3}} F(t )   =  \iint_{U \times\R^{3}} F_{0}   ,    
	 \ \ \  \ \
	\iint_{U \times\R^{3}} \frac{|v|^{2}}{2} F(t) 
	= 	 \iint_{U \times\R^{3}}  \frac{|v|^{2}}{2} F_{0} 
	.
	\end{equation}
	%
	%

	%
	
	\noindent  By normalization, we assume that 
	\begin{equation}\begin{split}\label{normalize_M}
	\iint_{ U\times \R^{3}} F_{0} (x,v)   =  \iint_{ U\times \R^{3}} \mu  ,   	 \ \ \  \ \
	\iint_{ U\times \R^{3}}  \frac{|v|^{2}}{2} F_{0} (x,v)   =  \iint_{ U\times \R^{3}}  \frac{|v|^{2}}{2} \mu   .   
	\end{split}
	\end{equation}
	
	In general, the total momentum is not conserved. However, in the case of axis-symmetric domains, we have an angular momentum conservation, i.e. if there exist a vector $x_{0}$ and an angular velocity $\varpi$ such that 
	\begin{equation}\label{axis-symmetric}
	\{  (x-x_{0}) \times \varpi\} \cdot \mathbf{n}(x) =0 \ \ \ \ \  \text{for all} \ x \in \p U,
	\end{equation}
then we have a conservation of the angular momentum as
	\begin{equation}\label{conserv_F_angular}
	\iint_{ U\times \R^{3}} \{  (x-x_{0}) \times \varpi\} \cdot v  F(t )
	=   \iint_{ U\times \R^{3}} \{  (x-x_{0}) \times \varpi\} \cdot v  F_{0}
	.
	\end{equation}
	In this case, we assume
	\begin{equation} \label{normalize_M_angular}
	\iint_{ U\times \R^{3}} \{  (x-x_{0}) \times \varpi\} \cdot v  F_{0}(x,v)
	= 0.  \\
	\end{equation}
	
	In this paper, we deal with periodic cylindrical domain with non-convex analytic cross section. A domain $U$ is given by 
	\begin{equation} \label{domain U}
	U = \O\times [0, H],  \ \ \
	(x_{1},x_{3}) \in \O \ \ \text{and} \ x_{2} \in [0, H] \ \ \ 
	 \text{for } (x_{1},x_{2},x_{3}) \in U
	\end{equation}
	where $\O \subset \R^{2}$ is the cross section. See Figure~\ref{fig1}. We assume that $F$ is periodic in $x_{2}$, i.e. $F(t, (x_{1}, 0, x_{3}),v) = F(t, (x_{1}, H, x_{3}), v )$. For the boundary of $U$, we denote
	$
		\p U :=  \p\O\times [0, H].
	$
We are interested in non-convex analytic cross section $\O \subset \mathbb{R}^{2}$:
	\begin{definition}\label{AND}
		Let $\Omega\subset\mathbb{R}^2$ be an open connected bounded domain and there exist simply connected subsets ${\Omega}_i \subset\mathbb{R}^2$, for $i=0,1,2,\cdots,M<\infty$ such that
		\begin{equation*} \label{domain}
		{\Omega} = {\Omega}_0 \backslash \{  {\Omega}_1 \cup  {\Omega}_2 \cup \cdots \cup  {\Omega}_M\},
		\end{equation*}
		where
		\begin{itemize}
			\item[1.] ${\Omega}_0 \supset\supset   {{\Omega}}_i $ for all $ \ i=1,2,\cdots,M$, and $ \p\O_{i} \cap \p\O_{j} =\emptyset$ for all $i\neq j$ and $i,j= 0,1,2,\cdots,M,$
			\item[2.] for each $\Omega_i$, there is a closed regular analytic curve $\alpha_i : [a_i,b_i]\rightarrow \mathbb{R}^2$ such that $\partial\Omega_i$ is an image of $\alpha_i$.
			\item[3.] $\p\O = \bigsqcup_{i=0}^{M} \p\O_{i}$,
		\end{itemize}
		where $\bigsqcup$ means disjoint union. 
	\end{definition}
	
	\begin{figure}[h] 
		\centering
		\begin{subfigure}[b]{0.55\textwidth}
			\includegraphics[width=\textwidth]{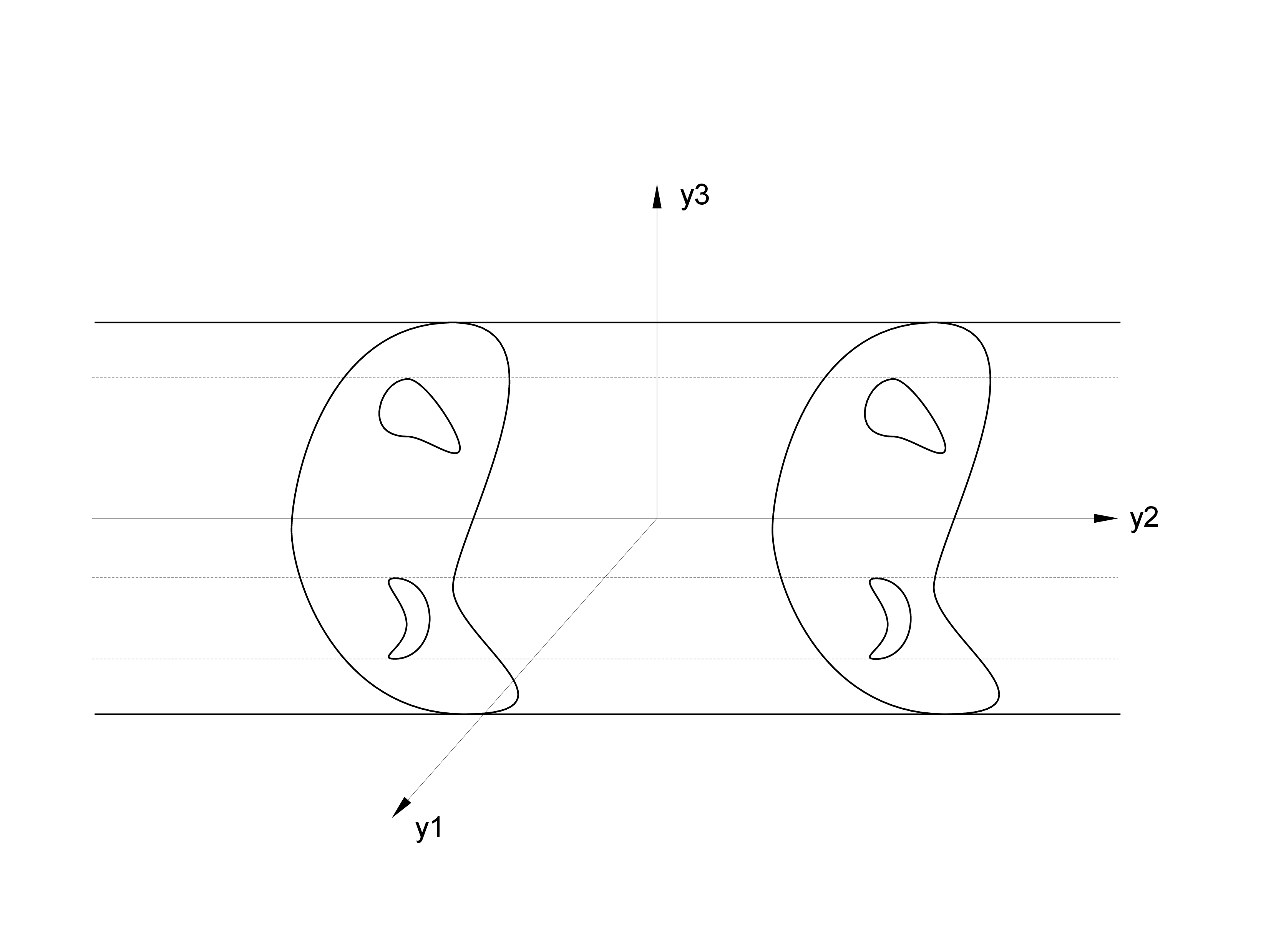}
			\caption*{Periodic cylindrical domain}
		\end{subfigure}
		~ 
		\begin{subfigure}[b]{0.4\textwidth}
			\includegraphics[width=\textwidth]{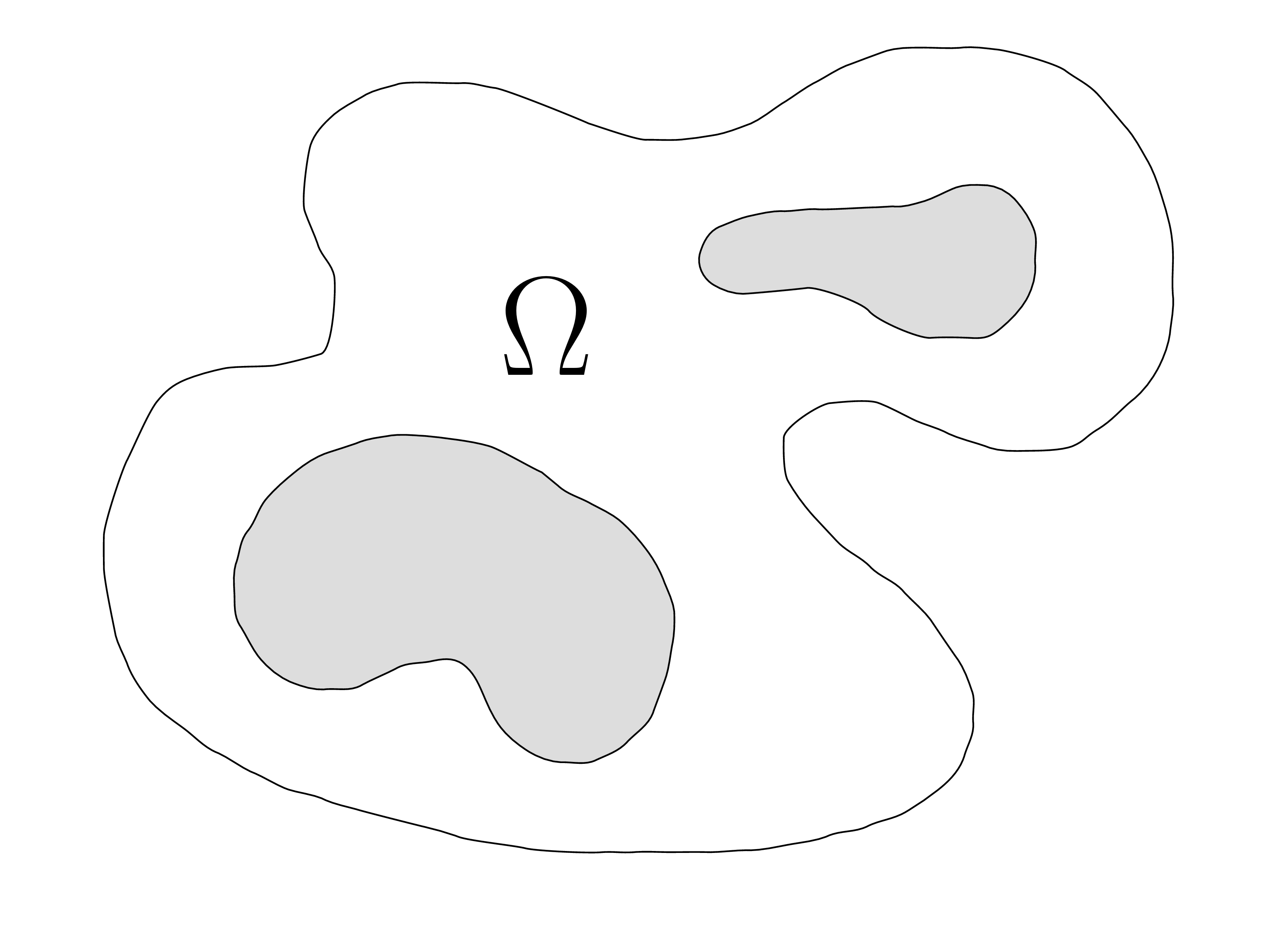}
			\caption*{Non-convex analytic cross section with punctures}
		\end{subfigure}
		\caption{Periodic cylindrical domain with non-convex analytic cross section}
		\label{fig1}
	\end{figure}	
	
	\begin{theorem}\label{theorem_decay}
		Let $w(v) = (1+|v|)^{\beta}$ with $\beta > \frac{5}{2}$. We assume periodic cylindrical domain $U$ defined in (\ref{domain U}), where analytic non-convex cross section with punctures $\O$ is defined in Definition \ref{AND}. We assume (\ref{normalize_M}) and also assume (\ref{normalize_M_angular}) if the cross section $\O$ is axis-symmetric (\ref{axis-symmetric}). Then, there exist $0 < \delta \ll 1$ such that if 
		\[
			F_{0} = \mu + \sqrt{\mu} f_{0} \geq 0\quad\text{and}\quad \|wf_{0}\|_{\infty} < \delta,
		\]
		then the Boltzmann equation (\ref{Boltzmann}) with the specular BC (\ref{specular}) has a unique global solution $F(t) = \mu + \sqrt{\mu}f(t) \geq 0$. Moreover, there exist $\lambda > 0$ such that
		\[	
			\sup_{t \geq 0} e^{\lambda t}\|wf(t)\|_{\infty} \lesssim \|wf_{0}\|_{\infty},
		\] 
		with conservations (\ref{conserv_F_mass}). In the case of axis-symmetric domain (\ref{axis-symmetric}), we have additional angular momentum conservation (\ref{conserv_F_angular}).   
	\end{theorem}

\vspace{4pt}

From (\ref{Boltzmann}), the perturbation $f$ satisfies 
	\begin{equation}\label{E_eqtn}
	\p_{t}f + v\cdot \nabla_{x} f + L f 
	 = \Gamma(f,f),  
	\end{equation} 
	and $f(t,x,v)  = f(t,x,R_{x}v),$ for $x\in \p U$ where 
	\begin{equation} \label{Lf}
	Lf = -\frac{1}{\sqrt{\mu}} \big[ Q(\mu, f\sqrt{\mu}) + Q(f\sqrt{\mu}, \mu) \big],
	 \ \ \ 
	\Gamma(f,f) = \frac{1}{\sqrt{\mu}} Q(f\sqrt{\mu}, f\sqrt{\mu}).
	\end{equation}
	The linear operator $Lf$ can be decomposed into $Lf = \nu(v)f - Kf$, where the collisional frequency $\nu(v)$ is defined
	\begin{equation}
	\nu(v) := \int_{\mathbb{R}^{3}} \int_{\mathbb{S}^{2}} |(v-u)\cdot\o| \sqrt{\mu}(u) d\o du,
	\end{equation} 
	with estimate $C_{0}\langle v \rangle \leq \nu(v) \leq C_{1}\langle v \rangle, $ where $\langle v \rangle := \sqrt{1+|v|^{2}}$ for some $C_{0}, C_{1} > 0$. The linear operator $Kf$ is a compact operator on $L^{2}(\mathbb{R}^{3}_{v})$ with kernel $\mathbf{k}(v,\cdot)$,
	\begin{equation}
	Kf(v) := \int_{\mathbb{R}^{3}} \mathbf{k}(v,u)f(u) du.  \\
	\end{equation}
	
	We explain main scheme of the proof of Theorem~\ref{theorem_decay}. To apply $L^{p}-L^{\infty}$ bootstrap argument, we claim the uniform number of bounce for a finite travel length. Also, we classify some singular sets especially where trajectories belong to grazing sets on the boundary.  
	\subsection{Uniform number of bounce on analytic domain}
	Let us denote backward trajectory of a particle as $X(s;t,x,v)$ and $V(s;t,x,v)$, where $X$ and $V$ are position and velocity of the particle at time $s$, which was at position $x$ with velocity $v$ at time $t \geq s$. Also we use $(t^{k}, x^{k}, v^{k}) = (t^{k}(x,v), x^{k}(x,v), v^{k}(x,v))$ to denote $k-$th bouncing time, position, and velocity backward in time. From the specular BC, dynamics in $x_{2}$ direction (axial direction) is very simple, because we have $X_2(s;t,x,v) = x_{2} - (t-s)v_{2}$ and $V_{2}(s;t,x,v) = v_{2}$. So we suffice to analyze trajectory projected onto two-dimensional cross section $\O$, with $(x_{1}, x_{3})\in \O$. We also consider finite time interval $[0,T_{0}]$ and velocity $v\in\mathbb{R}^{2}$ with $\frac{1}{N} \leq |v| \leq N$ so that maximal travel length is uniform bounded by $NT_{0}$. Unlike to strictly convex domain, trajectory $(X(s),V(s))$ can graze at some bouncing time $t^{k}$. We split grazing set $\{ (t^{k}, x^{k}, v^{k}) : v^{k}\cdot \mathbf{n}(x^{k}) = 0 \}$ into three types: convex grazing, concave grazing, and inflection grazing, depending on whether $x^{k}\in\p\O$ belongs to convex region, concave region, and inflection points. See Definition~\ref{decom-grazing} for explicit definitions. The following simplified lemma is the crucial tool to control the number of bounce. 
	
	\vspace{2pt}
	
	\noindent \textit{ \textbf{Simple version of Lemma~\ref{finitebounces}} If a trajectory does not belong to inflection grazing set, infinite number of bouncing cannot happen for a finite travel length. }   

	We prove this lemma via contradiction argument. If infinite number of bounce happens for a finite travel length, we have converging sequence of boundary points $x^{k} \rightarrow x^{\infty}$. By analyticity, all inflection points on the boundary $\p\O$ are finite and distinct. 
	
	$(i)$ If $x^{\infty}$ is a point in convex or concave part of $\p\O$, we can choose small boundary neighborhood $B(x^{\infty}, \varepsilon)\cap \p\O$ so that boundary is uniformly convex or concave in the neighborhood. If it is concave, trajectory does not stay in this small neighborhood. If it is uniformly convex, it is well-known that normal component of $v^{k}$ is always uniformly comparable and there exist at most finite number of bounce for a finite time interval (or equivalently finite travel length). We refer \cite{KimLee} and \cite{Guo10}.  
	
	$(ii)$ Therefore, the only possible case is when $x^{\infty}$ is an inflection point. By analyticity, every inflection points are isolated and we consider a sequence $x^{k}$ which converges to $x^{\infty}$ through convex region, because the trajectory leave the small neighborhood if it is in concave region. Using analyticity and properties of inflection points, profile of $\p\O$ near inflection points is nearly linear, i.e. $\alpha^{\prime\prime}(\tau_{\infty}) = 0$. Using this linear property, we obtain $|x^{k-1} - x^{k}| \leq |x^{k} - x^{k+1}|$ which is contradiction to our assumption $x^{k} \rightarrow x^{\infty}$. See the first picture in Figure~\ref{sticky fig}.  \\
	\indent Above \textit{Simple version of Lemma~\ref{finitebounces}}	is in sharp contrast to non-analytic general smooth domain, where infinite bounce in finite travel length is possible. We refer section 3 in \cite{Strange} for an example of infinite number of bounce for finite travel length. 
	
	\vspace{2pt}
	
	
	Meanwhile, a trajectory with the specular boundary condition is always deterministic and we can collect all possible trajectories including inflection grazing set. For each points on these trajectories, we uniformly cut corresponding velocities off. From compactness argument, we can define infinite bounces set 
	\[
		\mathfrak{IB} = \bigcup_{i=1}^{l_{IB}} \big\{ B(x_{i}^{IB}, r_{i}^{IB}) \times \mathcal{O}_{i}^{IB} \big\},
	\]    	 
	where $\bigcup_{i=1}^{l_{IB}} B(x_{i}^{IB}, r_{i}^{IB})$ is an open cover for $\O$, and corresponding open sets $\mathcal{O}_{i}^{IB}$ are sufficiently small in velocity phase. Moreover, the trajectory from $(x,v) \in \{cl(\O) \times \{\frac{1}{N}\leq |v|\leq N\} \} \backslash \mathfrak{IB}$ is uniformly away from grazing bounce, where $cl(\O)$ means closure of $\O$ in standard $\mathbb{R}^{2}$ topology. 
	
	Since we excluded inflection grazing and convex grazing, the only possibility is concave grazing. When a trajectory has concave grazing, $(t^{k}, x^{k}, v^{k})$ is not continuous function of $(x,v)$. However, away from grazing points, the trajectory is alway continuous in $(x,v)$. Therefore, for small perturbation $|(y,u)-(x,v)| \ll 1$, bouncing phase $(x^{k}(y,u), v^{k}(y,u))$ must be very close to some $(x^{\ell}(x,v), v^{\ell}(x,v))$, where $\ell \geq k$ can be different to $k$ by multiple concave grazings (e.g. Figure~\ref{case2-1}), and this implies finite number of bouncing. At last, from compactness of $\{cl(\O) \times \{\frac{1}{N}\leq |v|\leq N\} \} \backslash \mathfrak{IB}$, we derive the uniform number of bounce for given finite travel length.  
	
	\begin{figure}[h] 
		\centering
		\begin{subfigure}[b]{0.45\textwidth}
			\includegraphics[width=\textwidth]{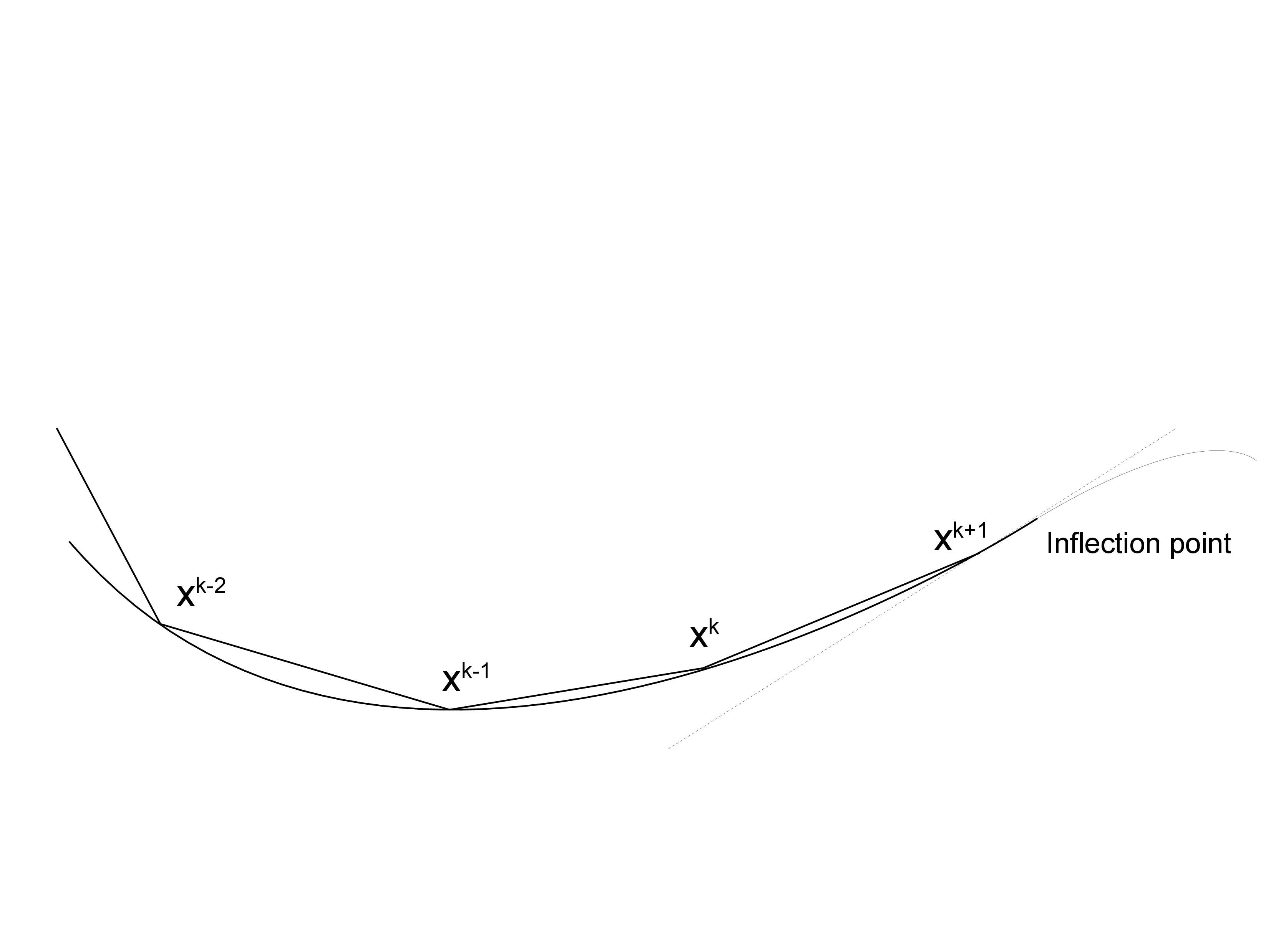}
			\caption*{Bouncing near inflection point}
		\end{subfigure}
		~ 
		\begin{subfigure}[b]{0.5\textwidth}
			\includegraphics[width=\textwidth]{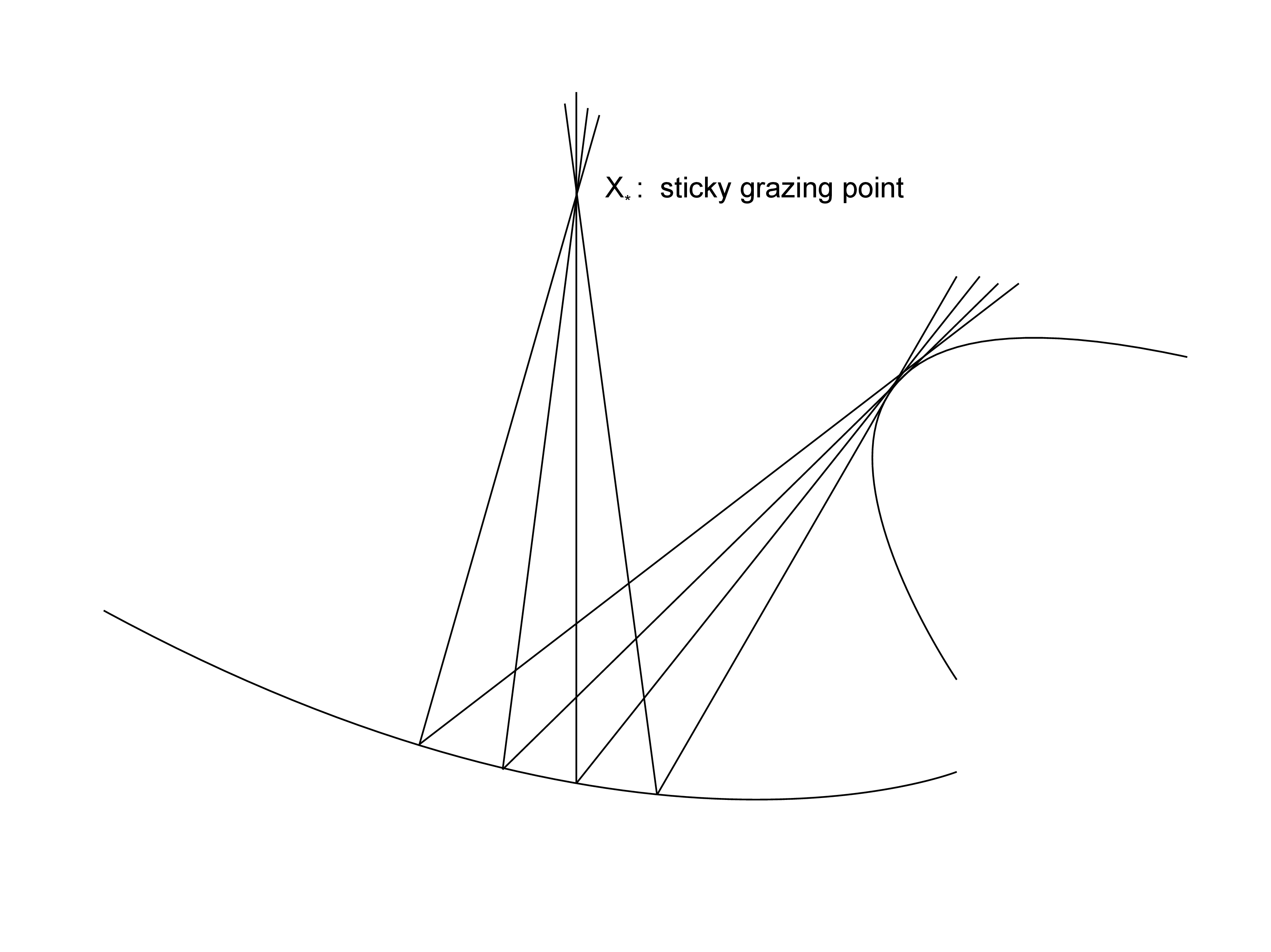}
			\caption*{Sticky grazing point}
		\end{subfigure}
		\caption{Bouncing near inflection and Sticky grazing point $\mathcal{SG}$}
		\label{sticky fig}
	\end{figure}

	\subsection{Sticky grazing set on analytic domain}
	For concave grazing set, we should consider another type of singular points. Let $\{ \alpha(\tau) : \tau_{1} < \tau < \tau_{2} \}$ be a local parametrization for concave boundary. Then $\big( \alpha(\tau), \alpha^{\prime}(\tau) \big)$ belongs to concave grazing set. Let us consider a set of trajectory
	\[
		X\big(s;t, \alpha(\tau), \alpha^{\prime}(\tau)\big), V\big(s;t, \alpha(\tau), \alpha^{\prime}(\tau)\big),\quad \tau_{1} < \tau < \tau_{2}.
	\]  
	If $\big( x^{1}(\alpha(\tau), \alpha^{\prime}(\tau)), v^{1}(\alpha(\tau), \alpha^{\prime}(\tau)) \big)$ is not grazing phase, we can use rigidity of analytic function to show that \textit{there could be a fixed point $x^{1}_{*} \in cl(\O)$ (we call $x^{1}_{*}$ a sticky grazing point)} such that
	\[
		x^{1}_{*} \in \overline{x^{1}(\alpha(\tau), \alpha^{\prime}(\tau)) x^{2}(\alpha(\tau), \alpha^{\prime}(\tau))},\quad \text{for all} \quad \tau\in (\tau_{1}, \tau_{2}).
	\]
	See second picture in Figure~\ref{sticky fig}. Also see Appendix for a concrete example of sticky grazing point. This implies that backward trajectory from 
	\[
		\big( x^{1}_{*}, -v^{1}(\alpha(\tau), \alpha^{\prime}(\tau)) \big),
	\]
	has grazing phase in the second bouncing for all $\tau_{1} < \tau < \tau_{2}$.   
	
	Similarly, for each $k$, if $\big( x^{i}(\alpha(\tau), \alpha^{\prime}(\tau)), v^{i}(\alpha(\tau), \alpha^{\prime}(\tau)) \big)$ is not grazing for all $1\leq i \leq k$, there could be a sticky grazing point $x^{k}_{*} \in cl(\O)$ such that
	\[
		x^{k}_{*} \in \overline{x^{k}(\alpha(\tau), \alpha^{\prime}(\tau)) x^{k+1}(\alpha(\tau), \alpha^{\prime}(\tau))},\quad \text{for all} \quad \tau\in (\tau_{1}, \tau_{2}).
	\]
	
	Now, we pick a point $x$ in 
	\begin{equation} \label{projec}
		\big\{ x\in cl(\O) : (x,v) \in \{cl(\O) \times \{\frac{1}{N}\leq |v|\leq N\} \} \backslash \mathfrak{IB}, \ \text{for some} \ v\in \{\frac{1}{N}\leq |v|\leq N\} \big\}.
	\end{equation}
	For above $x$ and $v\in \{\frac{1}{N}\leq |v|\leq N\}$, we have at most finite $K$ bounces for fixed travel length, so we can uniformly exclude a set of velocity so that trajectory $(X(s;t,x,v), V(s;t,x,v))$ avoids all concave grazings. Then the trajectory avoid all three types of grazing. Moreover, from the uniform number of bounce, set of all possible sticky grazing points, $\mathcal{SG}$, contains finite points at most. Excluding all small neighborhoods of the points in $\mathcal{SG}$ from (\ref{projec}), we can state the following lemma:
	
	\vspace{2pt}
	
	\noindent \textit{ \textbf{Simple version of Lemma~\ref{G_C unif}}} Excluding uniform $\varepsilon$ neighborhood of $\mathcal{SG}$, let us consider 
	\[
		(\ref{projec}) \backslash \bigcup_{x_{*}\in\mathcal{SG}} B(x_{*},\varepsilon) .
	\]
	Then we have a finite open cover $\bigcup_{i=1}^{l_{G}} B(x_{i}^{C}, r_{i}^{C})$ and corresponding small velocity sets $\{ \mathcal{O}_{i}^{IB} \big\}_{i=1}^{l_{G}}$ such that if 
	\[
		(x,v) \in B(x_{i}^{C}, r_{i}^{C}) \times \{ v : \frac{1}{N}\leq |v| \leq N \} \backslash \mathcal{O}_{i}^{IB},
	\]	   
	then $\big( x^{k}(x,v), v^{k}(x,v) \big)$ is uniformly non-grazing, i.e. 
	\[
		|v^{k}(x,v)\cdot \mathbf{n}(x^{k}(x,v))| > \delta > 0,\quad \text{for all}\quad 1\leq k \leq K,
	\] 
	where $K$ is uniformly finite number of bounce.

	\subsection{$L^{p}-L^{\infty}$ bootstrap and double iteration} 
	
	We consider linear Boltzmann equation,
	\begin{equation} \label{lin eq}
		\p_{t}f + v\cdot \nabla_{x} f + \nu(v) f 
		 = Kf,
	\end{equation}
	with the specular boundary condition. To apply $L^{p}-L^{\infty}$ bootstrap argument our aim is to claim 
	\begin{equation} \label{bootstrap} 
		\|f\|_{L^{\infty}} \lesssim \|f_{0}\|_{L^{\infty}} + \int_{0}^{T} \|f\|_{L^{2}}.
	\end{equation}
	This is obtained by trajectory analysis and change of variable. Let us explain using simplified version of (\ref{lin eq}) ,
	\begin{equation*} \label{simple}
	\p_{t}f + v\cdot \nabla_{x} f + f  = \int_{|u|\leq N} f du, \quad x\in\O. 
	\end{equation*}
	In the aspect of transport, $\p_{t}f + v\cdot\nabla_{x}f$ means simple transport of $f$ along trajectory and $f$ on the LHS means exponential decay effect along the trajectory. Therefore, Duhamel's principle gives
	\begin{equation*}\label{DH1}
	\begin{split}
	f(t,x,v ) &= e^{- t} f_{0}(x,v) + \int^{t}_{0} e^{- (t-s)} 
	\int_{|u| \leq N} f(s, X(s;t,x,v), u) \dd u
	\dd s.
	\end{split}
	\end{equation*}
	Applying this formula again (double iteration) to $f(s, X(s;t,x,v), u)$, we get
	\begin{equation*}\label{DH2}
	\begin{split}
	f(t,x,v) &= \textit{initial datum's contributions} + O(\frac{1}{N})  \\
	&\quad +  \int^{t}_{0} e^{- (t-s)} \int^{s-\e}_{0} e^{- (s-s^{\prime})}
	\iint_{|u|\leq N, |u^{\prime}| \leq N}
	f(s^{\prime}, X(s^{\prime}; s, X(s;t,x,v) , u), u^{\prime}) \dd u^{\prime}\dd u \dd s^{\prime}\dd s.  
	\end{split}
	\end{equation*}	
 The key step is to prove that the change of variable from $u$ to $X(s^{\prime}; s, X(s;t,x,v) , u)$ is valid. We apply geometric decomposition of trajectories as introduced in \cite{KimLee} to study a Jacobian marix of  
	\[
	\frac{X(s^{\prime})}{du},\quad\text{where}\quad X(s^{\prime}) = X(s^{\prime}; s, X(s) , u)\quad\text{and}\quad X(s) = X(s;t,x,v),
	\]
	We note that direct computations in Lemma~\ref{Jac_billiard} in t\textbf{}his paper is quite similar as Lemma 2.3 in \cite{KimLee}, because Lemma~\ref{Jac_billiard} assumes non-grazing bounce of a trajectory. Then Lemma~\ref{Jac_billiard} can be used to compute Jacobian of $\frac{X(s^{\prime})}{du}$. Note that, unlike to triple iteration scheme in \cite{KimLee}, we suffice to perform double iteration. Since we are assuming cylindrical domains, dynamics of axial component is very simple and $\frac{X(s^{\prime})}{d u_{2}} = -(t-s)$ gives rank one clearly. Then the problem is changed into claiming rank two in two-dimensional cross section $\O$. By decomposing $(u_{1}, u_{3})$ into speed and another independent directional variable, we can claim rank two.  
	
	Another main difference between strictly convex case (e.g. \cite{Guo10} and \cite{KimLee}) and non-convex case is the existence of sticky grazing set $\mathcal{SG}$. When trajectory $X(s^{\prime})$ hits boundary $\p\O$, we cannot perform change of variable so we should exclude such points, which should be chosen sufficiently small. However, for sticky grazing point $x\in\mathcal{SG}$, non-small portion of velocity phase	should be excluded. Therefore, we cannot cut this bad set off in velocity phase. Instead, we exclude $\mathcal{ {SG}}$ using small intervals in temporal integration. Because $\mathcal{ {SG}}$ contains finite set, if a particle speed is uniformly nonzero, say $|v| \geq \frac{1}{N}$, then we can choose sufficiently small neighborhoods near points in $\mathcal{ {SG}}$ so that a trajectory stays in these small neighborhoods of $\mathcal{SG}$ only for very short time at most.

	\section{Domain decomposition and notations}

	\subsection{Analytic non-convex domain and notations for trajectory}	
	
	Throughout this paper, cross section $\Omega$ is a connected and bounded open subset in $\mathbb{R}^2$. In this section, we denote the spatial variable $x=(x_{1},x_{3})\in cl({\Omega}) \subset \mathbb{R}^2$, where $cl({\Omega})$ denotes the closure of $\Omega$ in the standard topology of $\mathbb{R}^{2}$, and the velocity variable $v=(v_{1},v_{3})\in\mathbb{R}^{2}$. We also define standard inner product using dot product notation: $a\cdot b := (a_{1}, a_{3})\cdot (b_{1}, b_{3}) = a_{1} b_{1} + a_{3} b_{3}$.  \\
	
	The cross section boundary $\partial\Omega$ is a local image of some smooth regular curve. More precisely, for each $x\in\partial\Omega$, there exists $r>0$ and $\delta_1 < 0 < \delta_2$ and a curve $\alpha:= (\alpha_1, \alpha_3) : \{\tau \in\mathbb{R} : \delta_1<\tau<\delta_2\}\rightarrow \mathbb{R}^2$ such that 
	\begin{equation} \label{locallycurve}
	\partial\Omega\cap B(x,r) = \{\alpha(\tau)\in \mathbb{R}^{2} : \tau \in (\delta_{1}, \delta_{2})\},
	\end{equation}
	where $B(x,r) := \{ y\in \mathbb{R}^2 : |y-x|<r \}$ and
	$|\dot{\alpha}(\tau)|=[(\dot{\alpha}_{1}(\tau))^2+ (\dot{\alpha}_{3}(\tau))^2]^{1/2} := \Big[\big(\frac{d{\alpha}_{1}(\tau)}{d\tau}\big)^{2} +  \big(\frac{d{\alpha}_{3}(\tau)}{d\tau}\big)^{2}\Big]^{1/2}\neq 0,$
	for all $\tau \in (\delta_1,\delta_2)$. Without loss of generality, we can assume that $\alpha(\tau)$ is regularly parametrized curve, i.e. $|\dot{\alpha}(\tau)|=1$. For a smooth regularized curve $\alpha(\tau) = (\alpha_{1}(\tau),\alpha_{3}(\tau))\in\mathbb{R}^2$, we define the \textit{signed curvature} of $\alpha$ at $\tau$ by 
	\begin{equation} \label{singedcurvature}
	\kappa(\tau) := \ddot{\alpha}(\tau)\cdot \mathbf{n}(\alpha(\tau)) = \ddot{\alpha}_{1}(\tau)\dot{\alpha}_{3}(\tau) - \dot{\alpha}_{1}(\tau)\ddot{\alpha}_{3} (\tau) ,
	\end{equation}
	where $\mathbf{n}(\alpha(\tau)) = (\dot{\alpha}_{3}(\tau), - \dot{\alpha}_{1}(\tau))$ is outward unit normal vector on $\alpha(\tau)\in\p\O$.  \\
	
	Meanwhile, we assume that the curvature of $\p\O$ is uniformly bounded from above, so (\ref{locallycurve}) should be understood as simply connected curve, i.e. we can choose sufficiently small $r>0$ so that $\p\O\cap B(x,r)$ is simply connected curve for all $x\in\p\O$. Throughout this paper, we assume that a local parametrization of boundary satisfies (\ref{locallycurve}) as a simply connected curve. 
	
	We define convexity and concavity of $\alpha$ by the sign of $\kappa$ :
	\begin{definition} \label{SCD}
		Let $\Omega\subset\mathbb{R}^2$ be an open connected bounded subset of $\mathbb{R}^2$ and let the boundary $\partial\Omega$ be an image of smooth regular curve $\alpha \in C^{3}$ in (\ref{locallycurve}). For $\p\O\cap B(x,r) = \{\alpha(\tau) \ \vert \ \delta_1 < \tau < \delta_2 \}$, if
		\begin{equation*} \label{2Dconvexity}
		\kappa(\tau) < 0, \quad \delta_1 < \tau < \delta_2, 
		\end{equation*}
		then we say $\p\O\cap B(x,r)$ is locally \textit{convex}. Otherwise, if $\kappa(\tau) > 0$, we say it is locally \textit{concave}.  
	\end{definition}
	
	We denote the phase boundary of the phase space $\Omega \times \mathbb{R}^{3}$ as $%
	\gamma :=\partial \Omega \times \mathbb{R}^{3},$ and split into the outgoing
	boundary $\gamma _{+}$, the incoming boundary $\gamma _{-}$, and the
	grazing boundary $\gamma_{0}$ :
	\begin{equation} \label{phase def}
	\begin{split}
	\gamma _{0} &:= \{(x,v) \in \partial \Omega \times \mathbb{R}^{3}:\text{ \ }%
	\mathbf{n}(x)\cdot v=0\}, \\
	\gamma _{+} &:= \{(x,v) \in \partial \Omega \times \mathbb{R}^{3}:\text{ \ }%
	\mathbf{n}(x)\cdot v>0\}, \\
	\gamma _{-} &:= \{(x,v) \in \partial \Omega \times \mathbb{R}^{3}:\text{ \ }%
	\mathbf{n}(x)\cdot v<0\}. \\
	\end{split}
	\end{equation}
	
	Let us define trajectory. Given $(t,x,v)\in[0,\infty)\times cl({\Omega}) \times\mathbb{R}^3,$ we use  $[X(s),V(s)]=[X(s;t,x,v),V(s;t,x,v)]$ to denote position and velocity of the particle at time $s$ which was placed at $x$ at time $t$. Along this trajectory, we have
	\begin{equation*} \label{E_Ham} 
	\frac{d}{ds}X(s;t,x,v)   \ =  \ V(s;t,x,v),\ \ \  \frac{d}{ds} V(s;t,x,v)  \ = \ 0 . 
	\end{equation*} 
	with the initial condition: $(X(t;t,x,v),V(t;t,x,v)) = (x,v)$ .  
		
	\begin{definition} 
		\noindent We recall the standard notations from \cite{GKTT1}. We define 
		\begin{equation*}\begin{split}\label{backward_exit}
		\tb(t,x,v)  &:=  \sup \big\{  s \geq 0 :  X(\tau;t,x,v)  \in \Omega  \ \ \text{for all}  \ \tau \in ( t-s ,t)  \big\}, \\
		\xb(t,x,v)  &:=  X ( t- \tb(t,x,v);t,x,v), \\
		\vb(t,x,v)  &:=  \lim_{s\rightarrow\tb(t,x,v)}V ( t- s;t,x,v) ,  \end{split}
		\end{equation*}
		and similarly,
		\begin{equation*}\begin{split}\label{forward_exit}
		\tf(t,x,v)  &:=  \sup \big\{  s \geq 0 :  X(\tau;t,x,v)  \in \Omega  \ \ \text{for all}  \ \tau \in ( t,t+s)  \big\}, \\
		\xf(t,x,v)  &:=  X ( t+ \tf(t,x,v);t,x,v), \\
		\vf(t,x,v)  &:= \lim_{s\rightarrow\tf(t,x,v)} V ( t+ s;t,x,v). 
		\end{split}\end{equation*}
		Here, $\tb$ and $\tf$ are called the backward exit time and the forward exit time, respectively. We also define the the specular cycle as in \cite{GKTT1}. \\	
		We set $(t^{0}, x^{0}, v^{0}) = (t,x,v)$. When $\tb > 0$, we define inductively
		\begin{equation}\begin{split}\label{specular_cycles}
		t^{k} & = t^{k-1} - \tb (t^{k-1}, x^{k-1}, v^{k-1}),\\
		x^{k} &= X(t^{k}; t^{k-1}, x^{k-1}, v^{k-1}), \\
		v^{k}   &=  
		R_{x^{k}} V(t^{k}; t^{k-1}, x^{k-1}, v^{k-1}),
		\end{split}	
		\end{equation}
		where
		\begin{eqnarray*}
			R_{x^{k}} V(t^{k}; t^{k-1}, x^{k-1}, v^{k-1}) &=& V(t^{k}; t^{k-1}, x^{k-1}, v^{k-1}) 	\\
			&& - 2 \big( \mathbf{n}(x^{k})\cdot V(t^{k}; t^{k-1}, x^{k-1}, v^{k-1}) \big) \mathbf{n}(x^{k}).
		\end{eqnarray*}
		Since $t^{k}, x^{k}$, and $v^{k}$ depend on initial phase $(x,v)=(x^{0}, v^{0})$, we use $t^{k}(x,v), x^{k}(x,v)$, and $v^{k}(x,v)$ when we should denote initial phase. \\
		\noindent We define the specular characteristics as
		\begin{equation}\label{X_cl}
		\begin{split}
		X_{\mathbf{cl}} (s;t,x,v) &= \sum_{k} \mathbf{1}_{s \in ( t^{k+1}, t^{k}]}
		X(s;t^{k}, x^{k}, v^{k}),  \\
		V_{\mathbf{cl}} (s;t,x,v) &= \sum_{k} \mathbf{1}_{s \in ( t^{k+1}, t^{k}]}
		V(s;t^{k}, x^{k}, v^{k}).
		\end{split}
		\end{equation}
		For the sake of simplicity, we abuse the notation of (\ref{X_cl}) by dropping the subscription $\mathbf{cl}$ in this section.  \\
	\end{definition}

	\subsection{Decomposition of the grazing set and the boundary $\p\O$}
	In order to study the effect of geometry on particle trajectory, we further decompose the grazing boundary $\gamma_0$ (which was defined in (\ref{phase def})) more carefully:
	
	\begin{definition} \label{decom-grazing}
		Using disjoint union symbol $\sqcup$, we decompose grazing set $\gamma_{0}$:
		\[
		\gamma_{0} = \gamma_{0}^{C} \sqcup \gamma_{0}^{V} \sqcup \gamma_{0}^{I},\quad \gamma_{0}^{I} = \gamma_{0}^{I_{+}} \sqcup \gamma_{0}^{I_{-}}.
		\]
		$\gamma_{0}^{C}$ is \textit{concave(singular) grazing set:} 
		\begin{equation*}
		\gamma_{0}^{C} := \{(x,v)\in\gamma_0 : t_{\mathbf{b}}(x,v)\neq 0 \ \text{and} \ t_{\mathbf{b}}(x,-v)\neq 0\}.
		\end{equation*}
		$\gamma_{0}^{V}$ is \textit{convex grazing set:}
		\begin{equation*}
		\gamma_{0}^{V} := \{ (x,v)\in\gamma_0 : t_{\mathbf{b}}(x,v)=0 \ \text{and} \ \ t_{\mathbf{b}}(x,-v)=0
		\}.
		\end{equation*}
		$\gamma_{0}^{I_{+}}$ is \textit{outward inflection grazing set:} 
		\begin{equation*}
		\gamma_0^{I+} =\{ (x,v) \in \gamma_0 : t_{\mathbf{b}}(x,v) \neq 0 \ \text{and} \ t_{\mathbf{b}}(x,-v)=0 \ \text{and } \exists \ \delta>0 \ \text{such that } \ x+\tau v\in \mathbb{R}^{2}\backslash cl(\O) \ \text{ for } \ \tau\in (0,\delta)
		\}.
		\end{equation*}
		$\gamma_{0}^{I_{+}}$ is \textit{inward inflection grazing set:} 
		\begin{equation*}
		\gamma_0^{I-}=\{ (x,v) \in \gamma_0 : t_{\mathbf{b}}(x,v)=0 \ \text{and} \ t_{\mathbf{b}}(x,-v) \neq 0  \ \text{and } \exists \ \delta>0 \ \text{such that } \ x-\tau v\in \mathbb{R}^{2}\backslash cl(\O) \ \text{ for } \ \tau\in (0,\delta)
		\}.  \\
		\end{equation*}
	\end{definition}

	Recall that ${\Omega} \ := \  {\Omega}_0 \ \backslash \ \{  {\Omega}_1 \cup \cdots \cup  {\Omega}_M\},$ where each $ {\Omega}_{i}$ is an image of a unit-speed analytic curve $\alpha_i : [a_i,b_i]\rightarrow \mathbb{R}^2$. 
	Recall that $\kappa$ stands the signed curvature in Definition~\ref{singedcurvature}. Since the curvature $\kappa$ is continuous, the set $\{\tau \in [a_{i} , b_{i} ] : \kappa(\tau) > 0\}$ is an open subset of the interval $[a_i ,b_i]$ and therefore it is a countable union of disjoint open intervals, i.e. 
	$$
	\big\{\tau\in [a_{i} , b_{i} ] \ : \ \kappa(\tau) > 0 \big\}  \ = \ \bigsqcup_{j=1}^{\infty} \ \{\tau \in (a_{i,j}, b_{i,j}): a_{i,j}< \tau < b_{i,j}\}  .
	$$
	It is clear that $\kappa(a_{i,j})=0=\kappa(b_{i,j})$ for all $i,j$ : Suppose not, then there exists $\epsilon>0$ such that $(a_{i,j} -\epsilon,b_{i,j})\in\{\tau \in [a_i ,b_i] : \kappa(\tau) > 0\}$ or $(a_{i,j} ,b_{i,j} +\epsilon)\in\{\tau\in [a_i ,b_i] : \kappa(\tau) > 0\}$ which is a contradiction.
	
	On the other hand, the signed curvature $\kappa$ is analytic since the curve $\alpha_i$ is analytic. If $\kappa$ is identically zero then $\alpha_i$ is a straight line so that $\partial\Omega_i$ cannot be a boundary of a bounded set $\Omega$. Since the analytic function $\kappa$ have at most finite zeroes on a compact set $[a_i,b_i]$, there is a finite number $N_{i} < \infty$ such that
	\begin{equation}
	\{\tau\in [a_i ,b_i] \ : \ \kappa(\tau) > 0\} \ =  \ \bigsqcup_{j=1}^{N_{i}} \ \{\tau \in (a_{i,j}, b_{i,j}): a_{i,j}< \tau < b_{i,j}\}   , \nonumber
	\end{equation}
	which is a finite union of disjoint open intervals. \\
	
	Now we consider the closure of $\{\tau\in [a_i ,b_i] : \kappa(\tau)<0\}$ which is a union of closed intervals and there may exist two closed intervals which have same end point. For example $[a_{1,1},b_{1,1}]$ and $[a_{1,2},b_{1,2}]$ could have the same end point as $b_{1,1} = a_{1,2}$. In this case we can rewrite $[a_{1,1},b_{1,1}]\cup [a_{1,2},b_{1,2}] \ \equiv \ [{\tilde{a}}_{1,1},{\tilde{b}}_{1,1}]$ with ${\tilde{a}}_{1,1} = a_{1,1}$ and ${\tilde{b}}_{1,1} = b_{1,2}$. Therefore we can decompose the closure of ${\{\tau\in [a_i ,b_i] \ : \ \kappa(\tau)<0\}}$ as the disjoint union of $M_i(\leq N_i)$'s closed intervals:
	\begin{equation*}
	cl\big({\big\{\tau\in [a_i ,b_i] \ : \ \kappa(\tau) > 0\big\} }\big)  \ = \ \bigcup_{j=1}^{N_i} \ [a_{i,j}, b_{i,j}] \ = \ \bigsqcup_{j=1}^{M_i} \ [{\tilde{a}}_{i,j},{\tilde{b}}_{i,j}].\label{relabel}
	\end{equation*}
	For simplicity, we abuse the notation $a_{i,j}=\tilde{a}_{i,j}$ and $b_{i,j}=\tilde{b}_{i,j}.$
	\begin{definition}\label{CVI}
		Let $\Omega \ \subset \mathbb{R}^2$ be an analytic non-convex domain in Definition~\ref{AND}. We decompose the boundary $\partial {\Omega}$ into three parts;
		\begin{eqnarray*}
			&&\partial{ {\Omega}}^C \ := \ \bigsqcup_{i=0}^M \ \bigsqcup_{j=1}^{N_{i}}\big\{\alpha_i(\tau) : \tau \in ({ {a}}_{i,j} , { {b}}_{i,j})\big\}  \ := \ \bigsqcup_{l=1}^{M^{C}}\partial {\Omega}^C_{ l}, \ \ \ \ \ \  \  \text{(Concave boundary)}
			\\
			&&\partial{{\Omega}}^I \ := \ \bigsqcup_{i=1}^{M} \ \bigsqcup_{j=1}^{N_{i}} \ \{\alpha_i({ {a}}_{i,j}), \alpha_i ({ {b}}_{i,j})\},   \ \ \ \ \ \ \ \ \ \ \ \ \  \ \ \ \ \ \ \ \ \ \  \ \ \ \ \ \ \text{(Inflection boundary)}
			\\
			&&\partial{ {\Omega}}^V \ := \ \partial {\Omega} \ \backslash \ (\partial {\Omega}^{C}\cup\partial {\Omega}^{I}). \ \ \ \ \ \ \ \ \ \ \ \ \ \ \ \ \ \ \ \ \ \ \ \ \ \ \ \  \ \ \ \ \ \ \ \ \ \  \text{(Convex boundary)}
		\end{eqnarray*}
		The number $M^C \ = \ \sum_{i=0}^M N_{i}$ and the $l$-th concave part $\partial {\Omega}_{ l}^C$ for $l=1,2,\cdots,M^C$ is renumbered sequence of $ \big\{\alpha_i(\tau) : \tau \in [{ {a}}_{i,j} , { {b}}_{i,j}]\big\}$ for $i=0,1,\cdots ,M$ and $j=1,2,\cdots, N_{i}$. Therefore, we can define $M^{C}$ number of parametrization $\bar{\alpha}_{l}$ with $l=1,\cdots, M^{C}$ such that 
		\begin{equation} \label{renumber concave}
		\p\O^{C}_{l} = \{\bar{\alpha}_{l}(\tau) \ : \ \tau\in(\bar{a}_{l}, \bar{b}_{l}) \}.
		\end{equation}
		We further split $\partial\Omega^I = \partial\Omega^{I_+}  \cup \partial\Omega^{I_-}$ where $\partial\Omega^{I_+} := \bigsqcup_{i=0}^M \partial\Omega_{i}^{I_+}$ and $\partial\Omega^{I_-} := \bigsqcup_{i=0}^M \partial\Omega_{i}^{I_-}$ with
		\begin{eqnarray*}
			\partial\Omega_{i}^{I_+}=\{ \alpha_i (\tau) \in \partial\Omega_{i}^I : \exists \ \varepsilon > 0 \ \text{such that \ } \kappa_i(\tau^{\prime}) < 0 \ \text{for } \tau^{\prime}\in (\tau-\varepsilon, \tau) \ \text{and  } \kappa_{i}(\tau^{\prime}) > 0 \ \text{for  } \tau^{\prime} \in (\tau, \tau+\varepsilon)
			\},\\
			\partial\Omega_{i}^{I_-}=\{ \alpha_i (\tau) \in \partial\Omega_i^I : \exists \ \varepsilon > 0 \ \text{such that \ } \kappa_i(\tau^{\prime}) > 0 \ \text{for } \tau^{\prime}\in (\tau-\varepsilon, \tau) \ \text{and  } \kappa_{i}(\tau^{\prime}) < 0 \ \text{for  } \tau^{\prime} \in (\tau, \tau+\varepsilon)
			\}.
		\end{eqnarray*}
		
		\noindent Note that the following decomposition is compatible with those of Definition~\ref{decom-grazing}.
		\begin{eqnarray*}
			\gamma_0^{C} &=& \{(x,v)\in\gamma_0 : x \in \partial\Omega^C\}, \ \ \ \ \ \  \  \text{(Concave grazing set)}
			\\
			\gamma_0^{I} &=& \{(x,v)\in\gamma_0 : x \in \partial\Omega^I\}, \ \ \ \ \ \  \ \ \text{(Inflection grazing set)}
			\\
			\gamma_0^{V} &=& \{(x,v)\in\gamma_0 : x \in \partial\Omega^V\}. \ \ \ \ \ \  \ \ \text{(Convex grazing set)}
		\end{eqnarray*}
		
		\noindent Remark that from the definition, it is clear that
		\begin{eqnarray*}
			cl({\partial\Omega^{C}}) &=& \bigcup_{i=0}^{M} \bigcup_{j=1}^{N_i} cl\big( {\big\{ \alpha_i(\tau) : \tau \in (a_{i,j},b_{i,j})\big\}} \big)
			= \bigsqcup_{i=0}^{M} \bigsqcup_{j=1}^{N_i}  {\big\{ \alpha_i(\tau) : \tau \in [a_{i,j},b_{i,j}]\big\}}.
		\end{eqnarray*}
	\end{definition}

	\section{$L^{\infty}$ estimate}
	
	\subsection{Inflection grazing set}
	Trajectory of a particle is very simple for axial direction,
	\[
		V_{2}(s;t,x,v) = v_{2},\quad X_{2}(s;t,x,v) = x_{2} - (t-s)v_{2}. 
	\]
	Therefore, the characteristics of trajectories come from dynamics in two-dimensional cross section $\O$. In this subsection, we analyze trajectories in $\O \subset \mathbb{R}^{2}$. First, for fixed $N \gg 1$, we define the admissible set of velocity:
	\begin{equation*} \label{adset}
	\VN := \{ v\in\mathbb{R}^{2} : \frac{2}{N} \leq |v| \leq \frac{N}{2} \}	.
	\end{equation*}
	And $\mathfrak{m}_{2} : P(\O) \rightarrow \mathbb{R}$ is standard Lebesgue measure in $\mathbb{R}^{2}$. \\
	
	We control collection of bad phase sets those are nearly grazing set for each open covers contaning boundary $\p\O$.
	\begin{lemma} \label{near_boundary}
		Let $\Omega\subset\mathbb{R}^2$ be an analytic non-convex domain, defined in Definition~\ref{AND}. For $\varepsilon \ll 1, N \gg 1$, there exist finite points 
		\[
		\{{x}_{1}^{nB},\cdots, {x}_{l_{nB}}^{nB}\} \subset cl({\O}),
		\]
		and their open neighborhoods 
		\[
		B({x}_1^{nB}, r_1^{nB}), \cdots, B({x}_{l_{nB}}^{nB}, r_{l_{nB}}^{nB}) \subset \mathbb{R}^{2}, 
		\]
		as well as corresponding open sets 
		\[
		\mathcal{O}_1^{nB}, \cdots, \mathcal{O}_{l_{nB}}^{nB} \subset \mathbb{V}^N, 
		\]
		with $\mathfrak{m}_{2}(\mathcal{O}_i^{nB}) \leq \varepsilon$ for all $i=1,\cdots,l_{nB}$ such that for every $x\in cl({\O})$ there exists $i\in \{1,\cdots, l_{nB}\}$ with $x\in B({x}_i^{nB}, r_i^{nB})$ and satisfies either 
		\[
		B(x_i^{nB}, r_i^{nB}) \cap \partial\Omega = \emptyset \quad \text{or} \quad |v^{\prime}\cdot \mathbf{n}(x^{\prime})| > \varepsilon/ N^4,
		\] 
		for all $x^{\prime}\in B({x}^{nB}_i, r^{nB}_i)\cap\partial\Omega$ and $v^{\prime}\in \mathbb{V}^N  \backslash \mathcal{O}_i^{nB}$ .
	\end{lemma}
	\begin{proof}
		By Definition~\ref{AND}, $\partial{\Omega} \in \mathbb{R}^2$ is a compact set in $\mathbb{R}^2$ and a union of the images of finite curves. For ${x}\in{\Omega}$, we define $r_{{x}}>0$ such that $B(x,r_{{x}})\cap \partial{\Omega}=\emptyset$. For each ${x}\in\partial{\Omega}$, we can define the outward unit normal direction $\mathbf{n}(x)$ and the outward normal angle $\theta_n({x})\in [0,2\pi)$ specified uniquely by $\mathbf{n}({x}) = (\cos\theta_n({x}), \ \sin\theta_n({x}) )$. Using the smoothness and uniform boundedness of curvature of the boundary $\partial{\Omega}$, there exist uniform $r_{\varepsilon, N} > 0$ such that for $r_{x} \leq r_{\varepsilon, N}$, 
		\begin{equation} \label{near-paral}
		| \ \theta_{{n}}({x}^{\prime}) -  \theta_{{n}}({x}) \ | <  \varepsilon / 2N^2 \ \  \ \text{for all }  \ {x}^{\prime} \in {B}({x},r_{{x}}) \cap \partial{\Omega}, 
		\end{equation}
		and ${B}({x},r_{{x}}) \cap \partial{\Omega}$ is a simply connected curve. \\
		\indent By compactness, we have finite integer $l_{nB} > 0$,  points $\{x_{i}^{nB}\}_{i=1}^{l_{nB}}$, and positive numbers $\{r_{i}^{nB}\}_{i=1}^{l_{nB}}$ such that 
		\begin{equation*}
		cl({\O}) \subset \bigcup_{i=1}^{l_{nB}} B(x^{nB}_{i}, r^{nB}_{i}), \quad r_{i}^{nB} \leq r_{\varepsilon,N}.
		\end{equation*}
		By above construction, for each $1\leq i \leq l_{nB}$, we have either
		\begin{equation} \label{int}
		B(x_{i}^{nB}, r_{i}^{nB}) \cap \p\O = \emptyset,
		\end{equation}
		or 
		\begin{equation} \label{bdry}
		x_{i}^{nB} \in \p\O \quad \text{and} \quad r_{i}^{nB} < r_{\varepsilon, N} \quad \text{so that} \quad (\ref{near-paral})\quad \text{holds}.
		\end{equation}
		For $i$ with case (\ref{int}), we set $\mathcal{O}^{nB}_{i} = \emptyset$. For $i$ with case (\ref{bdry}), we define
		\begin{equation*} \label{def O}
		\begin{split}
		\mathcal{O}^{nB}_i &:= \Big\{ v\in \mathbb{V}^N :  {v}= \big(|{v}|\cos\theta, |{v}|\sin\theta\big)  \quad \text{where} \quad \theta \in \Big( \big(\theta_i \pm \frac{\pi}{2}\big) -\frac{\varepsilon}{N^3}, \ \big(\theta_i \pm \frac{\pi}{2}\big) + \frac{\varepsilon}{N^3} \Big) \Big\},  \\
		\end{split}
		\end{equation*}
		where we abbreviated $\theta_{n}(x_{i}^{nB}) = \theta_{i}$. Obviously, $\mathfrak{m}_{2}(\mathcal{O}_{i}^{nB}) \leq \pi \frac{N^{2}}{4}\frac{\varepsilon/N^{2}}{\pi} \leq \varepsilon$ and
		\begin{eqnarray*}
			|v^{\prime} \cdot \mathbf{n}(x^{\prime})| &\geq& |{v}^{\prime}| \times \Big|(\cos\theta^{\prime},\sin\theta^{\prime}) \cdot  (\cos \theta_{{n}}({x}^{\prime}), \sin \theta_{{n}}({x}^{\prime})) \Big|\\
			&\geq& \frac{2}{N} \times \Big|\cos\big(\frac{\pi}{2} + \frac{\varepsilon}{N^{3}} \big) \Big| = \frac{2}{N} \Big|\sin\big(\frac{\varepsilon}{N^{3}} \big) \Big| ,\quad \varepsilon/N^{3} \ll 1,  \\
			&\geq& \frac{\varepsilon}{N^{4}} ,
		\end{eqnarray*}
		for $x^{\prime}\in B({x}^{nB}_i, r^{nB}_i)$ and $v^{\prime} = |v^{\prime}|(\cos\theta^{\prime}, \sin\theta^{\prime})\in \mathbb{V}^N  \backslash \mathcal{O}^{nB}_i$.  \\
	\end{proof}

	We state critical property of analytical boundary for non-convergence of consecutive specular bouncing points. We use notation of the specular cycles $(x^{i}, v^{i})$ defined in (\ref{specular_cycles}). \\
	
	\begin{lemma} \label{finitebounces}
		Assume $\Omega\subset \mathbb{R}^2$ is the analytic non-convex domain of Definitiion~\ref{AND}. Choose ${x}\in cl({\O})$ and nonzero ${v} \in \VN$. If $ [ {x}^i( {x}, {v}), {v}^{i-1}( {x}, {v}) ] \notin \gamma_0^{I} $ for all $i=0,1,2,\cdots$, then
		\begin{equation*}
		\sum_{i=0}^{\infty} | x^i( {x},  {v}) - x^{i+1}( {x},  {v}) | = \infty \ .
		\end{equation*}
	\end{lemma}
	\begin{proof}
		We prove this lemma by contradiction argument: suppose $[x^i ({x},{v}), v^i ( {x},{v})]\notin \gamma_0^I$ for all $i=0,1,2,...$ and
		\begin{equation} \label{finitelength}
		\sum_{i=0}^{\infty} |x^i ({x},{v}) - x^{i+1}({x},{v})| < \infty.
		\end{equation}
		Then $x^i(x,{v})\rightarrow x^{\infty}$ and $x^{\infty} = \lim_{i\rightarrow \infty} \alpha(\tau_{i}) = \alpha(\tau_\infty)\in \partial\Omega$ using that $\partial\Omega$ is closed set. For $i \gg 1$, we assume $x^i(x,v) \in \{\alpha_j(\tau) : \tau \in [a_j,b_j]\}$ for some fixed $j\in\mathbb{N}$ in Definition~\ref{AND}. Otherwise $x^i(x,v)$ can not converge because $dist(\partial\Omega_{j_1},\partial\Omega_{j_2})> \delta>0$ for $j_1\neq j_2.$ Therefore we drop index $j$ and denote $\alpha(\tau_i)=\alpha_j(\tau_i) = x^i ({x},{v})$ in this proof.  \\
		
		\noindent\textit{Step 1}. Let us drop notation of fixed $(x,v)$ and assume that
		\[
		x^{i} = \alpha(\tau_{i}), \quad x^{i+1} = \alpha(\tau_{i+1}),\quad x^{i+2} = \alpha(\tau_{i+2}). 
		\]
		We claim that if $\tau_{i} < \tau_{i+1}$, then $\tau_{i+1} < \tau_{i+2}$ for sufficiently large $i \gg 1$. As explained in (\ref{locallycurve}), we can find $r^{*} \ll 1$ such that if $r \leq r^{*}$, then $B(x,r)\cap\p\O$ is simply connected curve for $x\in \p\O$. Also for $x\in\p\O$, we can find $r^{**} \ll 1$ such that if $r\leq r^{**}$, then $\{B(x,r)\cap\p\O\} \cap N(x) = \{x\}$ where $N(x) = \{c\mathbf{n}(x) : c\in\mathbb{R}\}$, the normal line crossing $x\in\p\O$. For $r = \min(r^{*}, r^{**})$ we can decompose 
		\begin{equation} \label{disj decomp}
			B(x^{i+1},r)\cap\p\O = \underbrace{ \Big\{ \{\alpha(\tau) : \tau < \tau_{i+1} \}\cap \p\O \Big\} }_{:=B_{-}} \ \bigsqcup \ \{x^{i+1}\} \ \bigsqcup \ \underbrace{ \Big\{ \{\alpha(\tau) : \tau > \tau_{i+1} \}\cap \p\O \Big\} }_{:=B_{+}}.
		\end{equation}
		From (\ref{finitelength}), for any $\varepsilon < \frac{1}{2}\min(r^{*}, r^{**})$, we can choose $R \gg 1$ such that 
		\begin{equation} \label{shortdist}
		|x^{i} - x^{i+1}| < \varepsilon,\quad \forall i > R.
		\end{equation}   \\
		If we consider $B(x^{i+1}, \min(r^{*}, r^{**}))$, both $x^{i}$ and $x^{i+2}$ are in $B(x^{i+1}, \min(r^{*}, r^{**})) \cap \p\O$ by (\ref{shortdist}). If $\tau_{i} < \tau_{i+1}$, then $\tau_{i} \in B_{-}$. Combining this fact with disjoint decomposition (\ref{disj decomp}), we know that $v^{i+1}\cdot \dot{\alpha}(\tau_{i+1}) > 0$. Therefore, $x^{i+2} \notin B_{-}$ and we already know that $x^{i+2} \neq x^{i+1}$. Finally we get 
		\[
		x^{i+2}\in \{ B(x^{i+1}, \min(r^{*}, r^{**})) \cap \p\O \} \backslash \{B_{-}\sqcup \{x^{i+1}\} \} := B_{+}.
		\]  
		By definition of $B_{+}$, $\tau_{i+1} < \tau_{i+2}$.  \\
		
		\noindent\textit{Step 2}. We split $\tau_{\infty}$ into three cases and study possible cases for (\ref{finitelength}). Without loss of generality, we assume that $\varepsilon$ and $i > R$ in the rest of this proof satisfy (\ref{shortdist}). \\
		\textit{(i)} If $\kappa(\tau_{\infty}) < 0$, $\exists \ \varepsilon > 0$ such that $\kappa(\tau) < 0$ for $\tau\in (\tau_{\infty} - \varepsilon, \tau_{\infty} + \varepsilon)$. While boundary is convex, we can apply velocity lemma, Lemma 1 in \cite{Guo10} or Lemma 2.6 in \cite{KimLee}. From the velocity lemma, normal velocity at bouncing points are equivalent, especially,
		\begin{equation} \label{velo lem}
		\begin{split}
			e^{C_{\O}(|v|+1)t^{i}} (v^{i}\cdot \mathbf{n}(x^{i})) &\leq e^{C_{\O}(|v|+1)t^{i+1}} (v^{i+1}\cdot \mathbf{n}(x^{i+1}))  \\
			e^{-C_{\O}(|v|+1)t^{i}} (v^{i}\cdot \mathbf{n}(x^{i})) &\geq e^{-C_{\O}(|v|+1)t^{i+1}} (v^{i+1}\cdot \mathbf{n}(x^{i+1})) . \\
		\end{split}
		\end{equation}
		Since nonzero speed $|v|$ is constant, (\ref{finitelength}) implies finite time stop of the trajectory. From (\ref{velo lem}), $v^{i}\cdot \mathbf{n}(x^{i})$ cannot be zero at finite time. So this is contradiction.  \\	
		\textit{(ii)} If $\kappa(\tau_{\infty}) > 0$, $\exists \ \varepsilon > 0$ such that $\kappa(\tau) > 0$ for $\tau\in (\tau_{\infty} - \varepsilon, \tau_{\infty} + \varepsilon)$. Without loss of generality, we choose $\varepsilon \leq \min(r^{*}, r^{**})$ which as chosen in \textit{Step 1}. By concavity, 
		\[
		(\alpha(\tau) - x^{i+1})\cdot\mathbf{n}(x^{i+1}) > 0 \quad\text{for}\quad \tau \in (\tau_{\infty}, \tau_{\infty} + \varepsilon)  \quad \text{where}\quad R_{x^{i+1}}v^{i+1}\cdot\mathbf{n}(x^{i+1}) < 0.
		\]
		This implies, $\tau_{i} \in (\tau_{i+1} - \varepsilon, \tau_{i+1}]$ then $\tau_{i+2} \notin [\tau_{i+1}, \tau_{i+1} + \varepsilon)$. This is contradiction.  \\
		\textit{(iii)} If $\kappa(\tau_{\infty}) = 0$ and $\kappa(\tau) > 0$ for $\tau\in(\tau_{\infty} - \varepsilon, \tau_{\infty}]$, this case is exactly same as case \textit{(ii)}.  \\
		\textit{(iv)} If $\kappa(\tau_{\infty}) = 0$ and $\kappa(\tau) = 0$ for $\tau\in(\tau_{\infty} - \varepsilon, \tau_{\infty}]$, then $\kappa(\tau) = 0 $ for $\tau \in [a_{j}, b_{j}]$ by analyticity. So, $\O$ must be half plane and we get contradiction.  \\
		\textit{(v)} Assume $\kappa(\tau_{\infty}) = 0$ and $\kappa(\tau) < 0$ for $\tau\in(\tau_{\infty} - \varepsilon, \tau_{\infty}]$.  \\
		
		\noindent\textit{Step 3}. We derive contradiction for the last case \textit{(v)} by claiming
		\begin{equation} \label{li}
		l_{i+1} = |x^{i+1}-x^{i}| \leq |x^{i+2}-x^{i+1}| = l_{i+2} < \varepsilon ,\quad i\geq R,
		\end{equation}
		for $\varepsilon$ and $R$ is what we have chosen in (\ref{shortdist}). As explained in (\ref{locallycurve}), we can assume that $B(x^{\infty}, \varepsilon) \cap \p\O$ is a graph of analytic function $\varphi(s)$. From the argument of \textit{Step 1}, we assume $s_{\infty}-\varepsilon < s_{i} < s_{i+1} < s_{i+2} < s_{\infty}$. Moreover, up to tranlation and rotation, we can assume that $\varphi(s_{i+1}) = \varphi^{\prime}(s_{i+1}) = 0$ and $\varphi^{\prime\prime}(s) > 0$ on $s\in(s_{\infty} - \varepsilon, s_{\infty})$. There exist $n_{0}\in\mathbb{N}$ such that
		\begin{eqnarray*}
			\frac{d^{n_{0}} \varphi}{ds^{n_{0}}} (s_{\infty}) \neq 0 \quad \text{and} \quad \frac{d^i \varphi}{ds^i}(s_{\infty}) = 0 \quad \text{for all } 0 \leq i < n_{0}.
		\end{eqnarray*}
		If $n_0 = \infty$, $\p\O$ is straight line so contradiction as explained in \textit{(iv)} of \textit{Step 2}. Also by definition of inflection point, $n_{0} \geq 3$. For finite $n_{0}\in\mathbb{N}$, for $|s| < \varepsilon \ll 1$, 
		\begin{equation} \label{expre1}
		\begin{split}
		\varphi^{\prime\prime}(s) &= c_{n_{0}-2} (s - s_{\infty})^{n_{0}-2} \big( 1 + O(|s-s_{\infty}|)\big) \rightarrow 0 \quad\text{as} \quad s\rightarrow s_{\infty}^{-}.
		\end{split}		
		\end{equation}
		To claim $|x^{i+1}-x^{i}| \leq |x^{i+2}-x^{i+1}|$, we suffice to claim $s_{i+1} - s_{i} \leq s_{i+2} - s_{i+1}$, because absolute values of slopes of $\overline{x^{i}x^{i+1}}$ and $\overline{x^{i+1}x^{i+2}}$ are same by the specular boundary condition. Since we assume $\varphi^{\prime}(s_{i+1}) = 0$, from the specular boundary condition,
		\begin{equation*} 
		\begin{split}
		\frac{\varphi(s_{i+2}) - \varphi(s_{i+1})}{s_{i+2} - s_{i+1}} = \frac{\varphi(s_{i}) - \varphi(s_{i+1})}{s_{i+1} - s_{i}},
		\end{split}
		\end{equation*}
		\begin{equation} \label{slope}
		\begin{split}
		\frac{1}{s_{i+2} - s_{i+1}} \int_{s_{i+1}}^{s_{i+2}} \int_{s_{i+1}}^{t}\varphi^{\prime\prime}(r) dr dt = \frac{1}{s_{i+1} - s_{i}} \int^{s_{i+1}}_{s_{i}} \int_{t}^{z_{i+1}}\varphi^{\prime\prime}(r) dr dt
		\end{split}
		\end{equation}
		It is important that near inflection point, from (\ref{expre1}), $\varphi^{\prime\prime}>0$ is monotone decreasing to zero on $s\in(s_{\infty} - \varepsilon, s_{\infty})$ for $\varepsilon \ll 1$. Therefore,
		\begin{equation} \label{compare}
		\begin{split}
		\frac{1}{s_{i+2} - s_{i+1}} \int_{s_{i+1}}^{s_{i+2}} \int_{s_{i+1}}^{t}\varphi^{\prime\prime}(r) dr dt &\leq \frac{1}{2} (s_{i+2} - s_{i+1})\varphi^{\prime\prime}(s_{i+1}),  \\
		\frac{1}{s_{i+1} - s_{i}} \int^{s_{i+1}}_{s_{i}} \int_{t}^{z_{i+1}}\varphi^{\prime\prime}(r) dr dt &\geq \frac{1}{2} (s_{i+1} - s_{i})\varphi^{\prime\prime}(s_{i+1}).  \\
		\end{split}
		\end{equation}
		From (\ref{slope}) and (\ref{compare}), we get $s_{i+1} - s_{i} \leq s_{i+2} - s_{i+1}$ and justify (\ref{li}). We proved contradictions for all possible cases listed in \textit{Step 2}, and finish the proof.
	\end{proof}	
	Remark that this fact is non-trivial because we can observe the infinitely many bounces of the specular cycles in a finite time interval even in some convex domains \cite{Strange}. Moreover in the case of non-convex domains we need to treat carefully the trajectories hit the inflection part (Definition~\ref{CVI}) tangentially. The analyticity assumption is essential in the proof. \\
	
	Using Lemma~\ref{finitebounces}, we define and control bad phase sets where their cycles may hit inflection grazing sets $\gamma_{0}^{I}$, defined in Definition~\ref{decom-grazing} or \ref{CVI}.
	
	\begin{lemma} \label{near_inflectionset}
		Let $\Omega\subset\mathbb{R}^2$ be an analytic non-convex domain in Definition~\ref{AND}. For $T_{0} > 0, \varepsilon \ll 1, N \gg 1$, there exist finite points 
		\[
		\{ {x}_{1}^{nI},\cdots, {x}_{l_{nI}}^{nI} \} \subset cl({\O}),
		\]
		and open balls 
		\[
		B({x}_1^{nI},r_1^{nI}), \cdots, B({x}_{l_{nI}}^{nI},r_{l_{nI}}^{nI}) \subset \mathbb{R}^{2},
		\]
		as well as corresponding open sets 
		\[
		\mathcal{O}_1^{nI}, \cdots, \mathcal{O}_{l_{nI}}^{nI} \subset \VN, 
		\]
		with $\mathfrak{m}_2(\mathcal{O}_i^{nI}) \lesssim \varepsilon$ for all $i=1,\cdots,l_{nI}$ such that for every $x\in cl({\O})$ there exists $i\in \{1,\cdots, l_{nI}\}$ with $x\in B({x}_i^{nI},r_i^{nI})$ and, for $v\in \mathbb{V}^N \backslash \mathcal{O}_i^{nI} $, the following holds.
		\begin{equation*}
		\left[X(s;T_0,x,v), V(s;T_0,x,v)\right]\notin \gamma_0^I \ \ \ \text{for all} \ s\in [0,T_0] \ .\label{cannotbeininflectionset}
		\end{equation*}
	\end{lemma}	
	\begin{proof}
		With the specular boundary condition, an particle trajectory is always reversible in time. Therefore, we track backward in time trajectory which depart from inflection grazing phase. Recall from Definition~\ref{CVI} that the inflection boundary $\partial\Omega^I$ is a set of finite points and denote $\partial\Omega^I = \{x_1^I,x_2^I,\cdots, x_{M^I}^I\}$. Define
		\begin{equation*} \label{trackfromI}  
		\bigcup_{j=1 }^{M^I}\Big\{ \big(X(s;T_0,x_j^I,v),  {V}(s;T_0,x_j^I,v)\big) \in cl({\O}) \times \mathbb{R}^2 : s\in [0,T_0], \ \ (x_j^I,v) \in \gamma_0^{I}, \ \ v \in \mathbb{V}^N
		\Big\}.
		\end{equation*}
		Now we fix one point of the inflection boundary $x_j^I\in\partial\Omega^I$ and a velocity $v_j^I\in\mathbb{R}^{2}$ with $|{v}_j^I|=1$ such that $(x_j^I,v_j^I) \in \gamma_0^I$. More precisely, for $x_j^I=\alpha_i(\tau)\in \partial\Omega^{I_+}_i$ with some $i=1,\cdots, M$ in Definition~\ref{CVI}, we choose ${v}_j^I= -\dot{\alpha}_i(\tau)$, and for $x_j^I=\alpha_i(\tau)\in \partial\Omega^{I_-}_i$ we choose ${v}_j^I = \dot{\alpha}_i(\tau)$ so that $(x_j^I,v_j^I) \in \gamma_{0}^{I_{+}}$ and backward in time trajectory is well-defined for short time $(T_{0}-\varepsilon, T_{0}], \ \varepsilon \ll 1$ at least. 
		
		Since $| {V}(s;T_0,x_j^I,v_j^I)| = |{v}_j^I| \leq \frac{N}{2}$ for $v_j^I\in\mathbb{V}^N$, possible total length of the specular cycles is bounded by $\frac{NT_0}{2}$. By Lemma~\ref{finitebounces}, number of bounce cannot be infinite for finite travel length without hitting inflection grazing phase. Moreover, if trajectory hit inward inflection grazing phase, $\gamma_{0}^{I_{-}}$, particle cannot propagate anymore. Therefore, number of bounce for finite travel length is always bounded. This implies 
		\begin{eqnarray*}
			{\verb"m"(x_j^I)} := \inf \Big\{ m \in\mathbb{N} \ : \ \sum_{i=1}^m \big| x^i(x_j^I,v_j^I) - x^{i+1}(x_j^I,v_j^I)\big| > \frac{NT_0}{2}
			\Big\} < +\infty,
		\end{eqnarray*}
		which actually depends on $N$ for fixed $\O$ and $T_{0}$. Therefore the set (\ref{trackfromI}) is a subset of
		\begin{equation*} \label{alltrajec}
		\mathcal{A} := \bigcup_{j=1}^{M^I}\bigcup_{i=0}^{\verb"m"(x_j^I)}\Big\{  (y,u) \in cl({\O})\times \mathbb{V}^N \ : \ y \in \overline{x^i(x_j^I,v_j^I) x^{i+1}(x_j^I,v_j^I)} \ \ \ \text{and} \ \ \ \frac{{u}}{|{u}|}= \pm v^{i}(x^I_j, {v}_j^I)
		\Big\},
		\end{equation*}
		which is a set of all particle paths from all inflection grazing phase. Now, we define projection of $\mathcal{A}$ on spatial dimension,
		\[
		\mathcal{P}_{x} (\mathcal{A}) := \bigcup_{j=1}^{M^I}\bigcup_{i=0}^{\verb"m"(x_j^I)}\Big\{  y \in cl({\O}) \ : \ y \in \overline{x^i(x_j^I,v_j^I) x^{i+1}(x_j^I,v_j^I)}
		\Big\},
		\]
		
		Now we construct open coverings : For $x \in cl({\O})\backslash\mathcal{P}_{x}(cl({\mathcal{A}}))$, we pick $r_{x}>0$ so that $B(x,r_{x})\cap \mathcal{P}_{x}(cl({\mathcal{A}})) = \emptyset$. For $x \in \mathcal{P}_{x}(cl({\mathcal{A}}))$, we pick $r_{x}>0$ to generate covering for $\mathcal{P}_{x}(cl({\mathcal{A}}))$. By compactness, we have finite open covering $B(x_{1}^{nI},r_1^{nI}),\cdots ,B(x_{l_{nI}}^{nI},r_{l_{nI}}^{nI})$. From above construction, for each $1\leq i \leq l_{nI}$, we have either
		\begin{equation} \label{int nI}
		B(x_{i}^{nI},r_{i}^{nI}) \cap \mathcal{P}_{x}(cl({\mathcal{A}})) = \emptyset,
		\end{equation}
		or 
		\begin{equation} \label{bdry nI}
		x_{i}^{nI} \in \mathcal{P}_{x} (cl({\mathcal{A}})).
		\end{equation}
		For $i$ with (\ref{int nI}) case, we set $\mathcal{O}^{nI}_{i} = \emptyset$. For $i$ with (\ref{bdry nI}) case, there are finite number of straight segments (may intersect each other) of $\mathcal{P}_{x}(\mathcal{A})$. This number of segments are bounded by $ M^{I} \times \max_{i} \verb"m"(x_{i}^{I}) < \infty$ for $i=1,\cdots, M^{I}$. By $\mathcal{O}^{nI}_{i}$ with $i$ satisfies (\ref{bdry nI}), we mean
		\begin{equation} \label{def O nI}
		\begin{split}
		\mathcal{O}^{nI}_{i} = \Big\{ u \in \mathbb{V}^N \ : \ \big| { {u}}/{| {u}|} \pm v^{i}(x^I_j, {v}_j^I)\big| < C_N \varepsilon , \ \ \forall(i,j) \ \ \text{s.t} \ \ \overline{x^i(x_j^I,v_j^I) x^{i+1}(x_j^I,v_j^I)}\cap B(x_{i}^{nI},r_{i}^{nI}) \neq \emptyset
		\Big\}.
		\end{split}
		\end{equation}
		Obviouly $\mathfrak{m}_{2}(\mathcal{O}_{i}^{nI}) \lesssim \pi \frac{N^{2}}{4}\frac{C_{N}}{2\pi}\varepsilon M^{I} \times \max_{i} \verb"m"(x_{i}^{I}) \lesssim \varepsilon$ by choosing $C_{N} \lesssim \frac{1}{N^{4}}$ for sufficiently large $N \gg 1$. 
		
		Now we prove (\ref{cannotbeininflectionset}). Since trajectory is reversible in time, $[X(s;T_{0},x,v), V(s;T_{0},x,v)] \notin \gamma_{0}^{I}$ if $(x,v) \notin \mathcal{A}$. By definition of (\ref{def O nI}), if $x\in B(x_i^{nI},r_i^{nI})$, $v\in \mathbb{V}^N \backslash \mathcal{O}_i^{nI}$, and $s\in [0,T_0]$, then $(x,v)\notin \mathcal{A}$. This finishes proof. 
	\end{proof}

	The following lemma comes from Lemma~\ref{near_boundary} and Lemma~\ref{near_inflectionset}.
	\begin{lemma} \label{infinite_bounces_set}
		Consider $\Omega$ as defined in Definition~\ref{AND}. For $\varepsilon \ll 1, N \gg 1$, there exist finite points 
		\[
		\{ {x}_1^{IB},\cdots, {x}_{l_{IB}}^{IB} \} \subset cl({\O}), 
		\]
		and open balls 
		\[
		B( {x}_1^{IB},r_1^{IB}),\cdots, B( {x}_{l_{IB}}^{IB},r_{l_{IB}}^{IB}) \subset \mathbb{R}^{2},
		\]
		as well as corresponding open sets 
		\[
		\mathcal{O}_1^{IB},\cdots, \mathcal{O}_{l_{IB}}^{IB} \subset \VN
		\]
		with $\mathfrak{m}_2(\mathcal{O}_i^{IB}) < C\varepsilon$ (for uniform constant $C>0$) for all $i=1,\cdots,l_{IB}$ such that for every $x\in cl({\O})$, there exists $i\in \{1,\cdots,l_{IB}\}$ with $x\in B({x}_i^{IB},r_i^{IB})$ and, for $v\in \VN \backslash \mathcal{O}_i^{IB}$, 
		\begin{equation*} \label{cond 1}
		|v\cdot \mathbf{n}(x)|> \frac{\varepsilon}{N^4} \ ,
		\end{equation*}
		for all $x\in \partial\Omega \cap B({x}_i^{IB},r_i^{IB})$ and
		\begin{equation*} \label{re-cannotbeininflectionset}
		\big( X(t_{k};T_0,x,v), V(t_{k};T_0,x,v) \big) \ \notin \ \gamma_{0}^{I} \quad  \text{for all} \ t_{k}\in [0,T_0] .  \\
		\end{equation*}
	\end{lemma}
	
	\bigskip
	
	Using above lemma, we define \textit{the infinite-bounces set} $\mathfrak{IB}$ as
	\begin{equation} \label{IB}
	\begin{split}
	\mathfrak{IB} &:= \Big\{ \ (x,v) \in cl({\O}) \times \VN  \ : \ 
	v\in \mathcal{O}^{IB}_{i} \ \text{for some} \ i\in \{1,2,\cdots,l_{IB}\} \ \text{satisfying} \ x\in B({x}_i^{IB},r_i^{IB})
	\Big\}.
	\end{split}	
	\end{equation}
		
	The most important property of the infinite-bounces set (\ref{IB}) is that the bouncing number of the specular backward trajectories on $\{ cl({\O}) \times \VN \} \backslash \mathfrak{IB}$ is uniformly bounded.

	\begin{definition}\label{number}
		When $L > 0, x\in cl({\O}) \subset \mathbb{R}^2$, and nonzero $v \in \mathbb{R}^{2}$ are given, we consider a set 
		\[
		\big\{
		k \in \mathbb{N}  \ : \ \sum_{j=1}^{k} \big| {x}^{j-1}(x,v)- {x}^{j}(x,v)\big| > L
		\big\}\subset \mathbb{N}.
		\]
		If this set is not empty, then we define $\mathfrak{N}(x,v,L)\in \mathbb{N}$ as following, 
		\begin{eqnarray*}
			\mathfrak{N}(x,v,L) := \inf \Big\{
			k \in \mathbb{N}   \ : \ \sum_{j=1}^{k} \big| {x}^{j-1}(x,v)- {x}^{j}(x,v)\big|>L
			\Big\}.
		\end{eqnarray*}
		Otherwise, if the set is empty, it means backward trajectory is trapped in $\gamma_{0}^{I_{-}}$, so we define 
		\[
			\mathfrak{N}(x,v,L) := \inf  \{ i \in \mathbb{N} : ( {x}^i(x,v), {v}^{i-1}(x,v)) \in \gamma_0^{I_{-}} \}. 
		\]
	\end{definition}

	From Lemma~\ref{infinite_bounces_set}, we have $\mathfrak{N}(x,v,\frac{NT_{0}}{2}) < \infty$ for $(x,v)\in \{ cl({\O}) \times \VN \} \backslash \mathfrak{IB}$. To improve this finite result into uniform bound, we use compactness and continuity arguments. 
	\begin{lemma} \label{local conti}
		Let $(x,v) \in cl({\O})\times\VN$. Then $(x^{k}(x,v), v^{k}(x,v))$ is a locally continuous function of $(x,v)$ if 
		\[
		(x^{i}(x,v), v^{i}(x,v)) \notin \gamma_{0},\quad \forall i\in\{1,2,\cdots,k\},
		\]	
		i.e. for any $\varepsilon > 0$, there exist $\delta_{x,v,\varepsilon} > 0$ such that if $|(x,v)-(y,u)|<\delta_{x,v,\varepsilon}$, then 
		\[
		\big| (x^{i}(x,v), v^{i}(x,v)) - (x^{i}(y,u), v^{i}(y,u)) \big| < \varepsilon,\quad \forall i\in\{1,2,\cdots,k\}.
		\]
		Moreover $(x^{i}(y,u), v^{i}(y,u)) \notin\gamma_{0}$ for $i\in\{1,\cdots,k\}$.
	\end{lemma}
	\begin{proof}
		First we claim continuity of $(x^{1}(x,v), v^{1}(x,v))$. Using trajectory notation and lower bound of speed in $\VN$, we know
		\[
		x^{1}(x,v) = X(\tb(t,x,v);t,x,v),\quad \tb \leq CN,
		\]
		for uniform $C$ which depend on the size of $\O$. Let us assume that $|(x,v)-(y,u)| \leq \delta$. Then
		\begin{equation} \label{split}
		\begin{split}
		| x^{1}(x,v) - x^{1}(y,u) | &\leq  | x^{1}(x,v) - x^{1}(x,u) | + | x^{1}(x,u) - x^{1}(y,u) | \\
		\end{split}
		\end{equation}
		Let $x^{1}(x,v) = \alpha_{j}(\tau^{*})$. Since $(x^{1}(x,v), v)\notin\gamma_{0}$, $\big| \frac{v}{|v|}\cdot \dot{\alpha}(\tau^{*}) \big| < 1$. Then we can choose sufficiently small $r_{x,v} \ll 1$ such that $\p\O\cap B(x^{1}(x,v),r_{x,v})$ is simply connected and intersects with line $\{x + sv : s\in\mathbb{R}\}$ in only one point non-tangentially, because $\dot{\alpha}\vert_{x^{1}(x,v)}$ is not parallel to $v$. Since $x + sv$ is continuous on $v$, $x + su$ must intersect to  $\p\O\cap B(x^{1}(x,v),r_{x,v})$ at some $\alpha_{j}(\tau) \in \p\O\cap B(x^{1}(x,v),r)$ whenever $|u-v| \ll \delta_{v,\varepsilon}$. This shows $| x^{1}(x,v) - x^{1}(x,u) | < O(\delta_{v,\varepsilon})$. And 
		\begin{equation} \label{nontan}
		\begin{split}
		\big| \frac{u}{|u|}\cdot \dot{\alpha}(\tau) - \frac{v}{|v|}\cdot \dot{\alpha}(\tau^{*}) \big| &= \big| \frac{u}{|u|}\cdot \dot{\alpha}(\tau) - \frac{u}{|u|}\cdot \dot{\alpha}(\tau^{*}) \big| + \big| \frac{u}{|u|}\cdot \dot{\alpha}(\tau^{*}) - \frac{v}{|v|}\cdot \dot{\alpha}(\tau^{*}) \big|  \\
		&\leq C|\tau - \tau^{*}| + N\delta \\
		&\leq C_{N} \delta \leq \frac{1}{2} \Big( 1 - \big| \frac{v}{|v|}\cdot \dot{\alpha}(\tau^{*}) \big| \Big)
		\end{split}
		\end{equation}
		for sufficiently small $\delta \ll 1$. This implies $\big| \frac{u}{|u|}\cdot \dot{\alpha}(\tau) \big| < 1$, i.e. $(x^{1}(x,u),u)\notin\gamma_{0}$.  \\
		Now, there exist small $r_{x,u} \ll 1$ such that $\p\O\cap B(x^{1}(x,u),r_{x,u})$ is simply connected and intersects with line $\{x + su : s\in\mathbb{R}\}$ in only one point non-tangentially by (\ref{nontan}). So there exist $\delta_{x,u,\varepsilon} \ll 1$ such that line $y + su$ hits $\p\O\cap B(x^{1}(x,u),r_{x,u})$ if $|x-y| < \delta_{x,u,\varepsilon}$. It is obvious that $| x^{1}(x,u) - x^{1}(y,u) | < r_{x,u}$. By far we showed continuity of $x^{1}(x,\cdot)$ and $x^{1}(\cdot,u)$. So continuity of $x^{1}$ follows from (\ref{split}).  \\
		
		To claim continuity of $v^{1}(x,v)$, we use continuity of $x^{1}(x,v)$. When $|(x,v)-(y,u)|<\delta_{x,v,\varepsilon} \ll 1$, we have $|x^{1}(x,v) - x^{1}(y,u)| = O(\delta_{x,v,\varepsilon})$ and therefore $\tau_{1}-\tau_{2} = O(\delta_{x,v,\varepsilon})$, where $\alpha(\tau_{1}) = x^{1}(x,v)$ and $\alpha(\tau_{2}) = x^{2}(x,v)$. By smoothness of $\alpha : \tau \mapsto \mathbb{R}^{2}$, $\mathbf{n}(\alpha(\tau_{1})) - \mathbf{n}(\alpha(\tau_{2}))$ is size of $O(\delta_{x,v,\varepsilon})$ as well as $\dot{\alpha}(\tau_{1}) - \dot{\alpha}(\tau_{2})$. By the specular boundary condition, we have
		\begin{equation*}
		\begin{split}
		|v^{1}(x,v) - v^{1}(y,u)| &\leq \big| R_{\alpha(\tau_{1})}v - R_{\alpha(\tau_{2})}u \big|   \\
		&\leq \big| (I - 2\mathbf{n}(\alpha(\tau_{1})) \otimes \mathbf{n}(\alpha(\tau_{1})) ) v - (I - 2\mathbf{n}(\alpha(\tau_{2})) \otimes \mathbf{n}(\alpha(\tau_{2})) ) u\big|  \\
		&\leq O(\delta_{x,v,\varepsilon}).  \\
		\end{split}
		\end{equation*} 
		Moreover, $v^{1}\cdot\mathbf{n}_{\tau_{1}}$ is also continuous function of $(x,v)$, so $(x^{1}(y,u), v^{1}(y,u)) \notin\gamma_{0}$ when $(y,u)$ are sufficiently close to $(x,v)$. Case of $i=2,\cdots,k$ are easily gained by chain rule, applying above argument several times. 	 
	\end{proof}
	
	\begin{lemma}\label{uniformbound}
		Let $\Omega\subset \mathbb{R}^2$ satisfies Definition~\ref{AND}. Then
		\begin{equation*} \label{numberofbounces}
		\begin{split}
		\mathfrak{N}^{*}_{\varepsilon, N , T_0} &:= \sup_{(x,v)\in \{ cl({\O}) \times \VN \} \backslash \mathfrak{IB} } \mathfrak{N}(x,v,NT_{0}) \ \leq \ C_{\varepsilon, NT_0}, 
		\end{split}
		\end{equation*}
		where $\mathfrak{N}(x,v,NT_{0})$ is defined in Definition~\ref{number} and $\varepsilon$-dependence comes from $\{\mathcal{O}_{i}^{IB}\}_{i=0}^{l_{IB}}$, which was defined in Lemma~\ref{infinite_bounces_set}.
	\end{lemma}
	\begin{proof}
		From Lemma~\ref{finitebounces} and \ref{infinite_bounces_set}, trajectory does not belongs to inflection grazing set during time $[0,T_{0}]$. $\mathfrak{N}(x,v,\cdot)$ is nondecreasing function for fixed $(x,v)\in \{ cl({\O}) \times \VN \} \backslash \mathfrak{IB} $ and we can assume $|v|=1$, because $NT_{0}$ is fixed maximal travel lenght during time interval $[0,T_{0}]$ with $v\in\VN$.  \\ 
		
		\noindent\textit{Step 1}. We study cases depending on concave grazing.  \\
		\noindent (\textit{Case 1}) If $\mathbf{n}(x^{i}(x,v))\cdot v^{i}(x,v) \neq 0$ for $i=1,\cdots,\mathfrak{N}(x,v,NT_{0})$, trajectory $(X(s;T_{0},x,v), V(s;T_{0},x,v))$ is continuous in $(x,v)$ by Lemma~\ref{local conti}. Therefore, we can choose $\delta_{x,v,\varepsilon,NT_{0}} \ll 1$, such that if $|(x,v)-(y,u)| < \delta_{x,v,\varepsilon,NT_{0}}$, then $|\xb(x,v) - \xb(y,u)| < O(\delta_{x,v,\varepsilon,NT_{0}})$, where $O(\delta_{x,v,\varepsilon,NT_{0}}) \rightarrow 0$ as $\delta_{x,v,\varepsilon,NT_{0}} \rightarrow 0$. Therefore, 
		\begin{equation*} \label{x control y}
		\mathfrak{N}(y,u,NT_{0}) \leq 1 + \mathfrak{N}(x,v,NT_{0}),  \\
		\end{equation*}	 
		for  $|(x,v)-(y,u)| < \delta_{x,v,\varepsilon,NT_{0}} \ll 1$. Moreover, we have 
		\begin{equation*} \label{position close}
		|x^{i}(x,v) - x^{i}(y,u)| < O(\delta_{x,v,\varepsilon,NT_{0}}),
		\end{equation*}	
		for $i=1,\cdots,\mathfrak{N}(x,v,NT_{0})$.  \\
		\noindent (\textit{Case 2}) 
		Assume that $(x^{i}(x,v), v^{i}(x,v))$ belongs to grazing set $\gamma_{0}$ for some $i \in \{ 1,\cdots,\mathfrak{N}(x,v,NT_{0}) \}$. Especially, $(x^{i}(x,v), v^{i}(x,v)) \in \gamma_{0}^{C}$, because $\gamma_{0}^{I}$ is not gained from $\{ cl({\O}) \times \VN \} \backslash \mathfrak{IB}$ as proved in Lemma~\ref{infinite_bounces_set}, and $\gamma_{0}^{V}$ is the stopping point for both forward/backward in time. Let us assume that $i \in \{0,\cdots,\mathfrak{N}(x,v,NT_{0}) \}$ is the smallest bouncing index satisfying $(x^{i}(x,v), v^{i}(x,v)) \in \gamma_{0}^{C}$. Eventhough there are consecutive convex grazings, it must stop at some $(x^{k}(x,v), v^{k}(x,v)$, because $\O$ is analytic and bounded domain, i.e. there exist $i,k \in\mathbb{N}$ such that \\
		\begin{equation} \label{cases}
		\begin{cases}
		(x^{j}(x,v), v^{j}(x,v)) \notin \gamma_{0}^{C},\quad \forall j < i,  \\
		(x^{j}(x,v), v^{j}(x,v)) \in \gamma_{0}^{C},\quad i \leq 
		\forall j \leq k-1,  \\
		(x^{i}(x,v), v^{i}(x,v)) \notin \gamma_{0}^{C},\quad j=k.  \\
		\end{cases}
		\end{equation}	 
		
		When $j < i$, the bouncing number can be counted similar as \textit{Step 1} ,
		\begin{equation*} \label{st2 x control y}
		\mathfrak{N}(y, u, N(T_{0} - t_{i-1}(x,v))) \leq 1 + \mathfrak{N}(x, v, N(T_{0} - t_{i-1}(x,v))),  \\
		\end{equation*}
		for $|(x,v)-(y,u)| < \delta_{x,v,\varepsilon,NT_{0}}$ for some $\delta_{x,v,\varepsilon,NT_{0}} \ll 1$. Now we consider consecutive multiple grazing.  \\
		
		When $i\leq j \leq k-1$, (consecutive convex grazing), we split into two cases, Case 2-1 and Case 2-2. \\
		
		\noindent (\textit{Case 2-1}) We assume $\mathbf{n}(x^{i}(x,v)) = \mathbf{n}(x^{i+1}(x,v)) = \cdots = \mathbf{n}(x^{k-1}(x,v))$. When $|(x,v)-(y,u)| < \delta_{x,v,\varepsilon,NT_{0}} \ll 1$, we have
		\[
			\big| (x^{i-1}(x,v), v^{i-1}(x,v)) - (x^{i-1}(y,u), v^{i-1}(y,u)) \big| < O(\delta_{x,v,\varepsilon,NT_{0}}) \ll 1.
		\] 
		from Lemma~\ref{local conti}. When trajectorys $\big( X(s;y,u,T_{0}), V(s;y,u,T_{0} \big)$ passes near $x^{i}(x,v)$, we split into several cases. \\ 
		
		\begin{figure}[h]
			\centering
			\begin{subfigure}[b]{0.4\textwidth}
				\includegraphics[width=\textwidth]{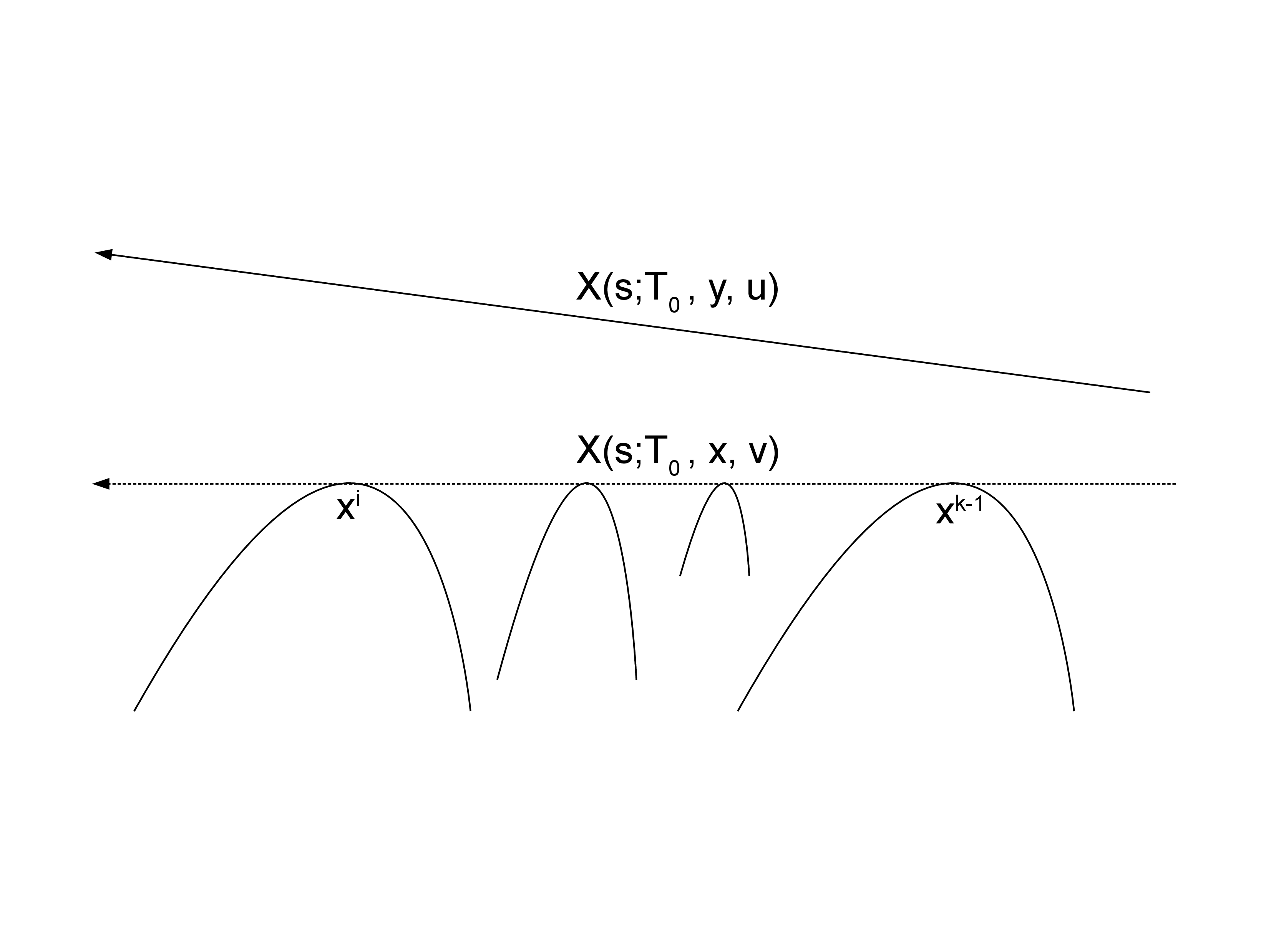}
				\caption*{(i)}
			\end{subfigure}
			\begin{subfigure}[b]{0.4\textwidth}
				\includegraphics[width=\textwidth]{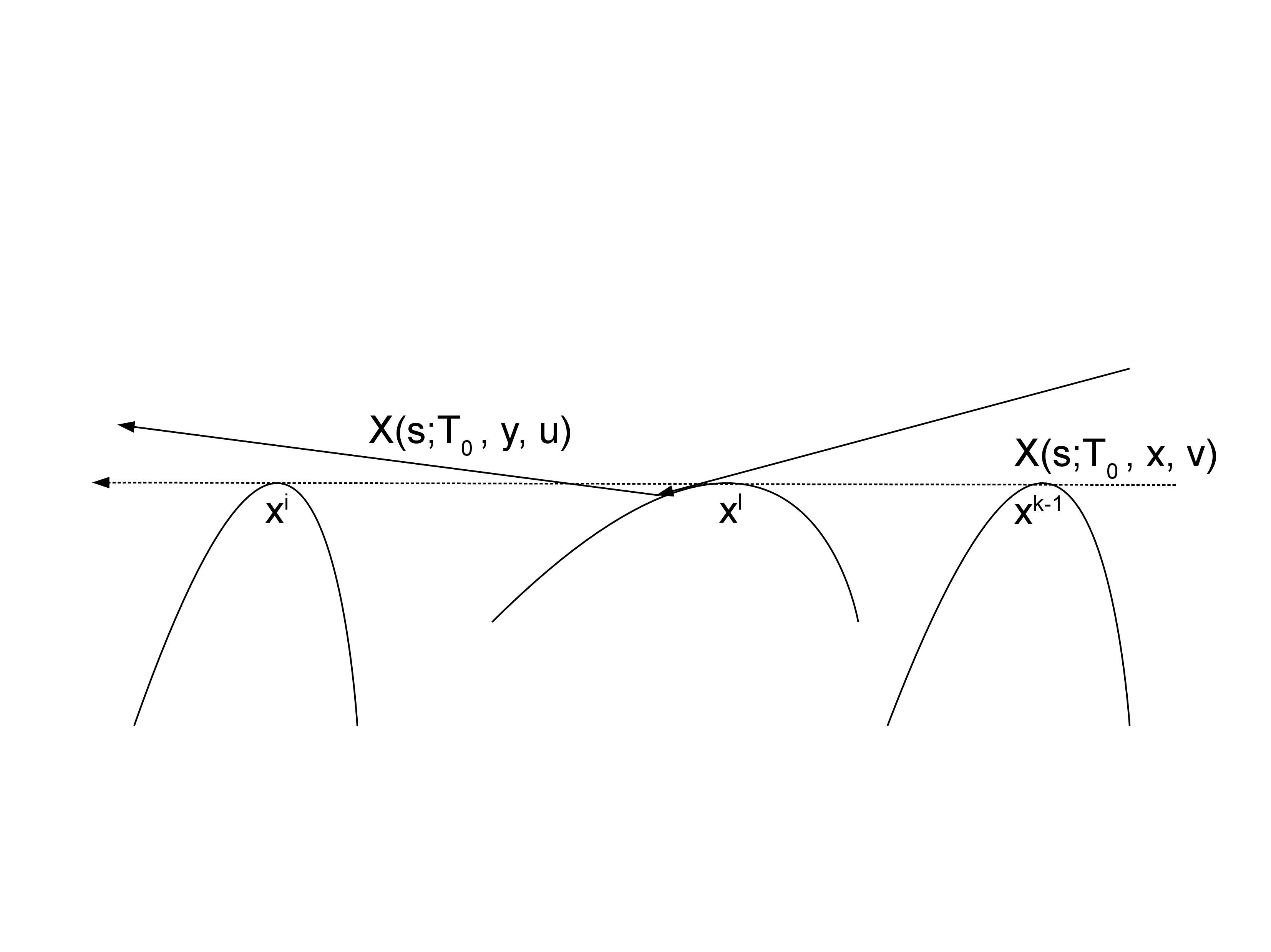}
				\caption*{(ii)}
			\end{subfigure}
			\begin{subfigure}[b]{0.4\textwidth}
				\includegraphics[width=\textwidth]{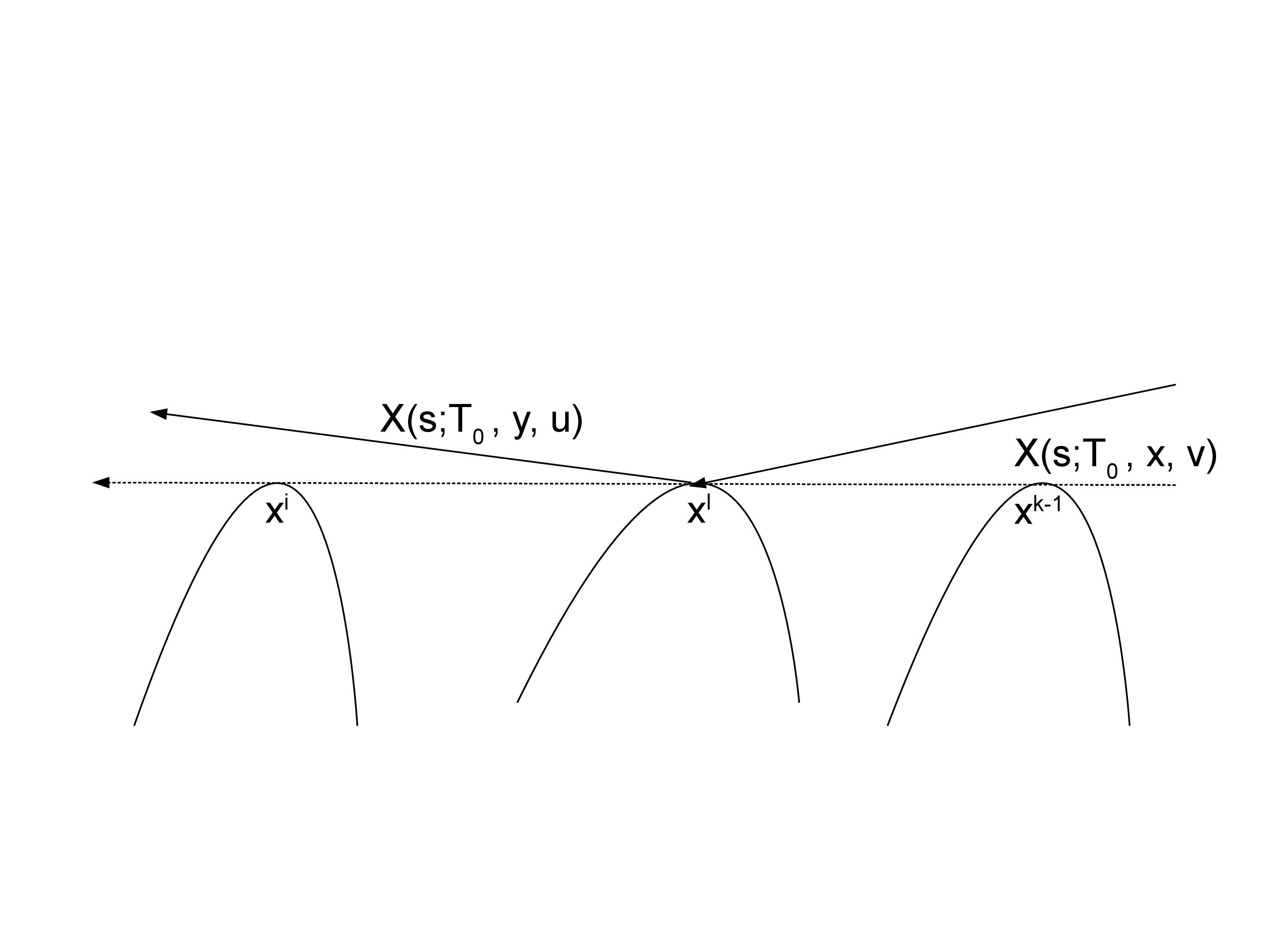}
				\caption*{(iii)}
			\end{subfigure}
			\begin{subfigure}[b]{0.4\textwidth}
				\includegraphics[width=\textwidth]{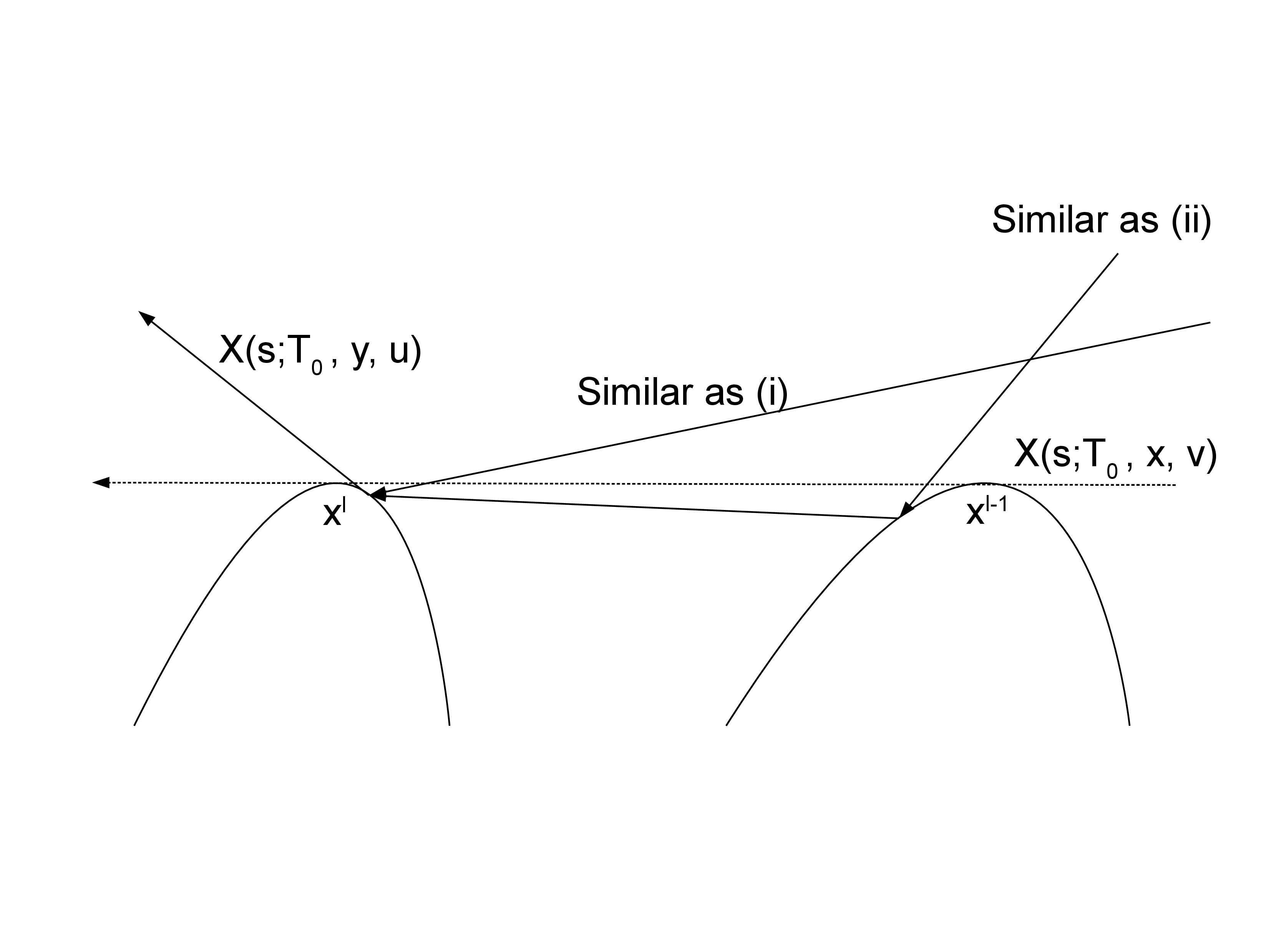}
				\caption*{(iv)}
			\end{subfigure}
			\caption{Case 2-1}\label{case2-1}
		\end{figure}

		We claim that 
		\begin{equation} \label{Case21}
		\begin{split}
		\mathfrak{N}(y, u, N(T_{0} - t_{k}(x,v))) &\leq 1 + \mathfrak{N}(x, v, N(T_{0} - t_{k}(x,v))),  \\
		\end{split}
		\end{equation}
		holds for all following $(i)-(iv)$ cases.  \\
		
		\noindent $(i)$ If $\overline{x^{i-1}(x,v) x^{k}(x,v)}$ does not bounce near $x^{j}(x,v)$ for all $j\in\{i,\cdots,k-1\}$, then obviously we get (\ref{Case21}). \\
		
		If case $(i)$ does not hold, we can assume that the backward trajectory $\big( X(s;y,u,T_{0}), V(s;y,u,T_{0} \big)$ hits near $x^{\ell}(x,v)$ without hitting near $x^{j}(x,v)$ for $i\leq j \leq \ell-1$. Without loss of generality, we parametrize $B(x^{\ell}(x,v), \varepsilon) \cap \p\O, \ \varepsilon \ll 1 $ by regularized curve $\big\{ \beta^{\ell}(\tau) : \tau^{\ell} - \delta_{1} < \tau <  \tau^{\ell} + \delta_{2}, \ \ \beta^{\ell}(\tau^{\ell}) =  x^{\ell}(x,v) \big\}, \ \ 0 \leq \delta_{1}, \delta_{2} \ll 1$. \\
		
		\noindent $(ii)$ Let $x^{i}(y,u) = \beta^{\ell}(\tau)$ with $\tau^{\ell} - \delta_{1} < \tau < \tau^{\ell}$. Without loss of generality, we assume multigrazing dashed line as $x$-axis. By the specular BC, the trajectory $\big( X(s;y,u,T_{0}), V(s;y,u,T_{0} \big)$ must be above tangential line $\{ x^{i}(y,u) + s \dot{\beta}^{\ell}(\tau) \ \vert \ s\in\mathbb{R} \}$ near $x^{i}(y,u)$. Moreover, from the specular BC,    
		\begin{equation} \label{direc compare}
		\begin{split}
			&\Big| \frac{v^{i}(y,u)}{|v^{i}(y,u)|}\cdot \dot{\beta}^{\ell}(\tau) \Big| = \Big| \frac{v^{i-1}(y,u)}{|v^{i-1}(y,u)|}\cdot \dot{\beta}^{\ell}(\tau) \Big|  \\
			&\leq \Big| \Big( \frac{v^{i-1}(y,u)}{|v^{i-1}(y,u)|} - \frac{v^{\ell}(x,v)}{|v^{\ell}(x,v)|} \Big) \cdot \dot{\beta}^{\ell}(\tau) \Big| +
			\Big| \frac{v^{\ell}(x,v)}{|v^{\ell}(x,v)|} \cdot \big( \dot{\beta}^{\ell}(\tau) - \dot{\beta}^{\ell}(\tau^{\ell}) \big) \Big| + 
			\Big| \frac{v^{\ell}(x,v)}{|v^{\ell}(x,v)|} \cdot \dot{\beta}^{\ell}(\tau^{\ell}) \Big|  \\
			&\leq 1 + O(\delta_{x,v,\varepsilon,NT_{0}}).
		\end{split}
		\end{equation}
		This implies that the angle between $v^{i-1}(y,u)$ and tangential line $\{ x^{i}(y,u) + s \dot{\beta}^{\ell}(\tau) \ \vert \ s\in\mathbb{R} \}$ are very small, so we can apply the argument of $(i)$ again and we obtain (\ref{Case21}). \\
		
		\noindent $(iii)$ When $x^{i}(y,u) = \beta^{\ell}(\tau^{\ell})$, we must have    \\
		\begin{equation} \label{dir comp3}
		\begin{split}
		&\Big| \frac{v^{i}(y,u)}{|v^{i}(y,u)|}\cdot \dot{\beta}^{\ell}(\tau^{\ell}) \Big| = \Big| \frac{v^{i-1}(y,u)}{|v^{i-1}(y,u)|}\cdot \dot{\beta}^{\ell}(\tau^{\ell}) \Big| = \Big| \frac{v^{i-1}(y,u)}{|v^{i-1}(y,u)|}\cdot \frac{v^{\ell}(x,v)}{|v^{\ell}(x,v)|} \Big| = 1 + O(\delta_{x,v,\varepsilon,NT_{0}}).
		\end{split}
		\end{equation}	
		So the angle between $v^{i-1}(y,u)$ and $v^{\ell}(x,v)$ are very small. Moreover, trajectory $\big( X(s;y,u,T_{0}), V(s;y,u,T_{0} \big)$ must be above dash tagential line, we can apply $(i)$ to derive (\ref{Case21}).  \\
		
		\noindent $(iv)$ When $x^{i}(y,u) = \beta^{\ell}(\tau)$ with $\tau^{\ell} < \tau < \tau^{\ell} + \delta_{2} $, angle between $\dot{\beta}^{\ell}(\tau)$ and $\dot{\beta}^{\ell}(\tau^{\ell})$ is very small, since $\delta_{2} \ll 1$. Moreover, angle between $v^{i-1}(y,u)$ and $\dot{\beta}^{\ell}(\tau)$ is also small from (\ref{direc compare}). Therefore the angle between $v^{i-1}(y,u)$ and $\dot{\beta}^{\ell}(\tau^{\ell})$ is also small, i.e. $v^{i-1}(y,u)$ is nearly parallel with dashed line in Fig 1. Therefore only $(i)$ and $(ii)$ cases are possible for $x^{i+1}(y,u)$. For both cases, we gain (\ref{Case21}).  \\ 
		
		\noindent (\textit{Case 2-2}) Assume that there exist $\{ p_{1}, p_{2}, \cdots, p_{q} \} \in \{i+1,\cdots,k-1\} $ with $p_{1} < p_{2} < \cdots < p_{q}$ such that 
		\begin{equation*} 
		\begin{cases}
			\mathbf{N} := \mathbf{n}(x^{i}(x,v)) = \mathbf{n}(x^{i+1}(x,v)) = \cdots = \mathbf{n}(x^{p_{1}-1}(x,v))  \\
			-\mathbf{N} = \mathbf{n}(x^{p_{1}}(x,v)) = \mathbf{n}(x^{p_{1}+1}(x,v)) = \cdots = \mathbf{n}(x^{p_{2}-1}(x,v))  \\
			\mathbf{N} = \mathbf{n}(x^{p_{2}}(x,v)) = \mathbf{n}(x^{p_{2}+1}(x,v)) = \cdots = \mathbf{n}(x^{p_{3}-1}(x,v))  \\
			-\mathbf{N} = \mathbf{n}(x^{p_{3}}(x,v)) = \mathbf{n}(x^{p_{3}+1}(x,v)) = \cdots = \mathbf{n}(x^{p_{4}-1}(x,v))  \\
			\quad\quad\quad\quad\quad\quad\quad\quad\quad\quad\quad\quad\cdots		
		\end{cases}
		\end{equation*} 
			
		\begin{figure} [h]  
			\centering
			\begin{subfigure}[h]{0.4\textwidth}
				\includegraphics[width=\textwidth]{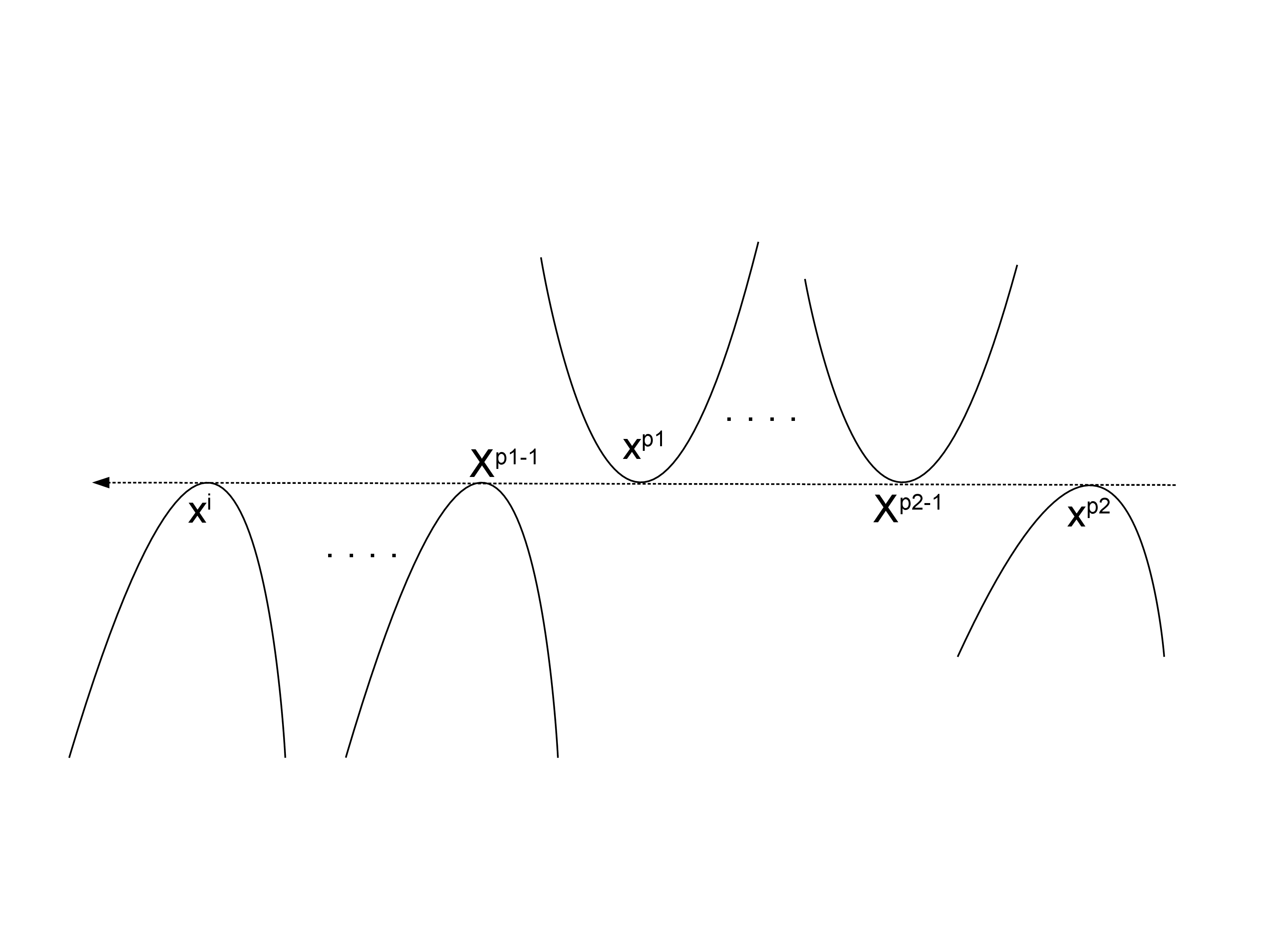}
			\end{subfigure}
			\begin{subfigure}[h]{0.4\textwidth}
				\includegraphics[width=\textwidth]{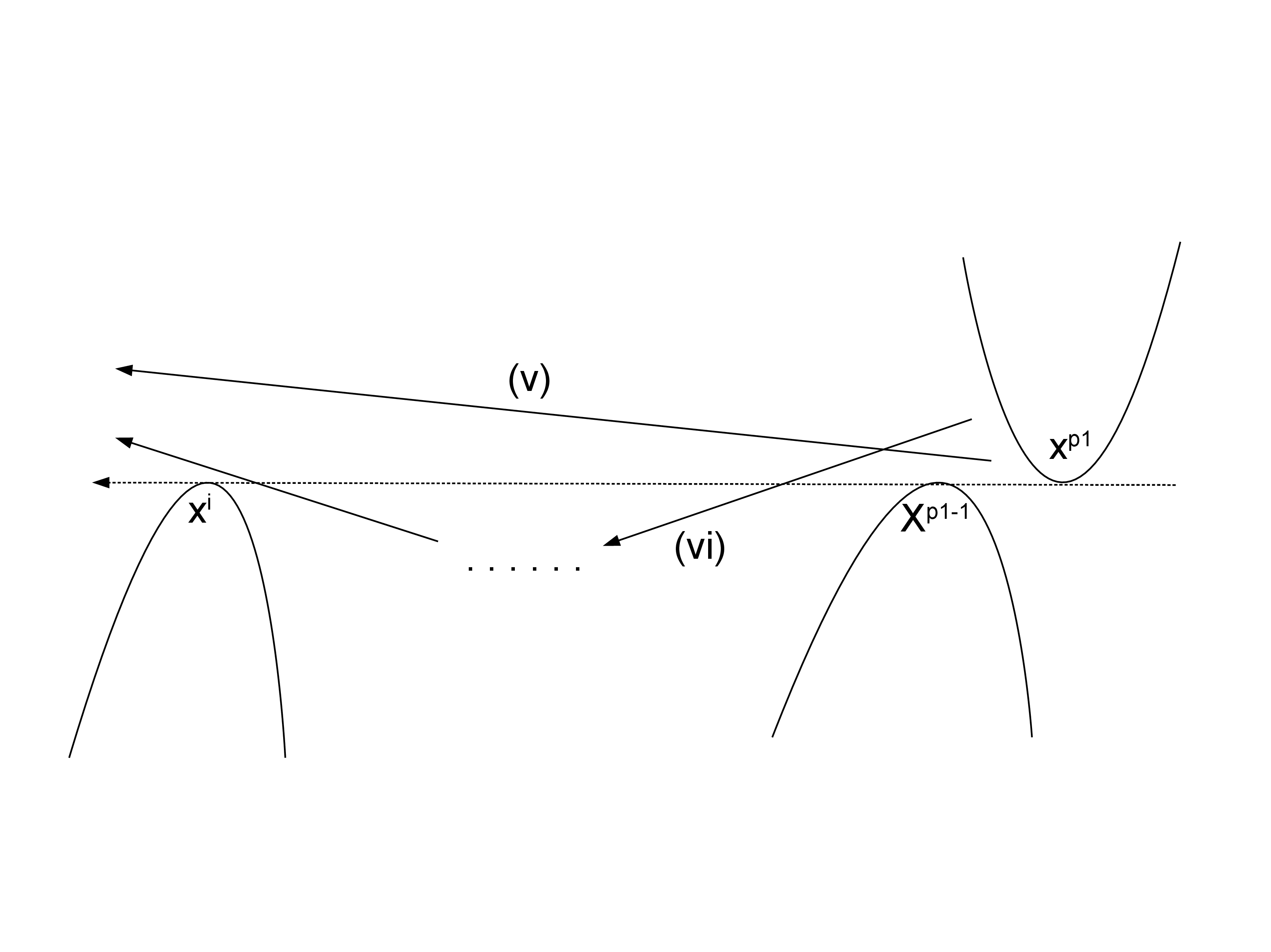}
				\caption*{(v) and (vi)}
			\end{subfigure}
			\caption{Case 2-2}\label{case2-2}
		\end{figure}

		 We split into cases and claim that
		\begin{equation*} \label{Case22}
		\begin{split}
		\mathfrak{N}(y, u, N(T_{0} - t_{k}(x,v))) &\leq 1 + \mathfrak{N}(x, v, N(T_{0} - t_{k}(x,v))),  \\
		\end{split}
		\end{equation*} 
		holds for all cases.  \\
		
		First we define $T_{p_{\ell}} := (t^{p_{\ell}-1}(x,v) - t^{p_{\ell}}(x,v) ) /2, \ 1 \leq \ell \leq q $, and choose $\delta_{x,v,\varepsilon,NT_{0}}$ so that 
		\begin{equation} \label{Tp}
			\delta_{x,v,\varepsilon,NT_{0}} \ll \frac{T_{p_{\ell}}}{N}, \quad \text{for all} \ \ell \in\{1,\cdots, q\},
		\end{equation}
		which implies that traveling time (or distance) between $x^{p_{1}}(x,v)$ and $x^{p_{1}-1}(x,v)$ is sufficiently larger than the size of $\delta_{x,v,\varepsilon, NT_{0}}$. We split into two cases $(v)$ and $(vi)$ as following.  \\
		
		\noindent $(v)$ If $x^{i}(y,u)$ does not hit near any of $x^{i}(x,v),\cdots,x^{p_{1}-1}(x,v)$, we have 
		\begin{equation} \label{before p1}
		\begin{split}
			\Big| \big( X(T_{p_{1}}; y,u,T_{0}), V(T_{p_{1}}; y,u,T_{0}) \big) - \big( X(T_{p_{1}}; x,v,T_{0}), V(T_{p_{1}}; x,v,T_{0}) \big) \Big| \leq O(\delta_{x,v,\varepsilon, NT_{0}}),
		\end{split}
		\end{equation} 
		by Lemma~\ref{local conti}.  \\
		
		\noindent $(vi)$ If $x^{i}(y,u)$ hits near one of $x^{i}(x,v),\cdots,x^{p_{1}-1}(x,v)$, then we can apply $(ii)$, $(iii)$, or $(iv)$ of Case 2-1 to claim that there are \textit{at most} 2 bouncings before trajectory $\big( X(s;y,u,T_{0}) , V(s;y,u,T_{0}) \big)$ approaches $x^{p_{1}}(x,v)$. Moreover, in any case of $(ii)$, $(iii)$, and $(iv)$, (assuming 2 bouncings WLOG), 
		\[
		| v^{i+2}(y,u) - v^{i}(x,v) | = | v^{i+2}(y,u) - v^{p_{1}}(x,v) | = O(\delta_{x,v,\varepsilon, NT_{0}}).
		\]
		And, since trajectory $ X(s;y,u,T_{0}) $ is very close to $X(s;x,v,T_{0})$,
		\[
			\Big| X(s; y,u,T_{0}) - X(s; x,v,T_{0}) \Big| \leq O(\delta_{x,v,\varepsilon, NT_{0}}),\quad t^{i-1}(x,v) \leq s \leq T_{p_{1}}.
		\]  
		Using above two estimates for both velocity and position, (\ref{before p1}) also holds for case $(vi)$.  \\
		
		Now let us derive uniform number of bounce of the second case in (\ref{cases}). For (\textit{Case 2-1}), we proved that (\ref{Case21}) holds. For (\textit{Case 2-2}) case, we change index $p_{1}-1 \leftrightarrow k-1$, and then apply the same argument of (\textit{Case 2-1}) to derive   
		\begin{equation*} \label{Case21 p1}
		\begin{split}
		\mathfrak{N}(y, u, N(T_{0} - T_{p_{1}}(x,v))) &\leq \mathfrak{N}(x, v, N(T_{0} - T_{p_{1}}(x,v))),  \\
		\end{split}
		\end{equation*}
		During $\big( t^{p_{2}}(x,v), t^{p_{1}}(x,v) \big)$, we can also apply same argument of (\textit{Case 2-1}) with help of (\ref{Tp}) and (\ref{before p1}) to obtain 
		\begin{equation*} \label{Case21 p2}
		\begin{split}
		\mathfrak{N}(y, u, N(T_{0} - T_{p_{1}}(x,v))) &\leq \mathfrak{N}(x, v, N(T_{0} - T_{p_{1}}(x,v))).  \\
		\end{split}
		\end{equation*}
		
		We iterate this process until $T_{p_{q}}$ to obtain
		\begin{equation*} \label{Case21 pq}
		\begin{split}
		\mathfrak{N}(y, u, N(T_{0} - T_{p_{q}}(x,v))) &\leq \mathfrak{N}(x, v, N(T_{0} - T_{p_{q}}(x,v))).  \\
		\end{split}
		\end{equation*}
		And since $(x^{k}(x,v), v^{k}(x,v))$ is non-grazing, we have 
		\begin{equation} \label{Case21 pk}
		\begin{split}
		\mathfrak{N}(y, u, N(T_{0} - t^{k}(x,v))) &\leq 1 + \mathfrak{N}(x, v, N(T_{0} - t^{k}(x,v))),  \\
		\end{split}
		\end{equation}
		by applying (\textit{Case 2-1}) for traveling from near $x^{p_{q}}(x,v)$ to $x^{k}(x,v)$ . \\
		
		\noindent\textit{Step 2}. When we encounter second consecutive convex grazings after $t^{k}(x,v)$, we can follow \textit{Step 1} to derive similar estimate as (\ref{Case21 pk}). Finally there exist $\delta_{x,v,\varepsilon, NT_{0}} \ll 1$ such that
		\begin{equation} \label{x control y re}
		\begin{split}
		\mathfrak{N}(y,u,NT_{0}) &\leq 1 + \mathfrak{N}(x,v,NT_{0}),  \\
		\end{split}
		\end{equation}
		where $(x,v)\in \{ cl({\O}) \times \VN \} \backslash \mathfrak{IB}$. Since $\mathfrak{IB}$ is open set from (\ref{IB}), $\{ cl({\O}) \times \VN \} \backslash \mathfrak{IB}$ is closed set. And then we use compactness argument to derive uniform boundness from (\ref{x control y re}). For each $(x,v)\in \{ cl({\O}) \times \VN \} \backslash \mathfrak{IB}$, we construct small balls $B((x,v),\delta_{x,v,\varepsilon,NT_{0}})$ near each points. For each $(y,u)\in B((x,v),\delta_{x,v,\varepsilon,NT_{0}})$, (\ref{x control y re}) holds. By compactness, there exist a finite covering $\bigcup_{i=1}^{\ell} B((x_{i},v_{i}),\delta_{x_{i},v_{i},\varepsilon,NT_{0}})$ for some finite $\ell < \infty$. Therefore, for any $(y,u)\in \{ cl({\O}) \times \VN \} \backslash \mathfrak{IB}$, 
		\begin{equation*} \label{unif cpt}
		\begin{split}
		\mathfrak{N}(y,u,NT_{0}) &\leq 1 + \max_{1\leq i\leq \ell} \mathfrak{N}(x_{i}, v_{i}, NT_{0})  \leq C_{\varepsilon, NT_{0}}.
		\end{split}
		\end{equation*}
	\end{proof}
	
	\begin{lemma}\label{uniformawayfrominflection}
		Let $\Omega\subset \mathbb{R}^2$ saftisfies Definition~\ref{AND}. For any $(x,v)\in \{ cl({\O}) \times \VN \} \backslash \mathfrak{IB}$, trajectory $(X(s;T_{0},x,v), V(s;T_{0},x,v))$ for $s\in[0,T_{0}]$ is uniformly away from inflection grazing set $\gamma_{0}^{I}$, i.e. there exists $\rho_{\varepsilon, NT_{0}} > 0$ such that
		\begin{equation} \label{dist def}
		\begin{split}
			D_{\mathcal{I}}(s,x,v) &:= dist(\partial\Omega^{\mathcal{I}}, X(s;T_0,x,v)) + |\mathbf{n}(X(s;T_0,x,v))\cdot V(s;T_0,x,v)| \\
			&\geq \rho_{\varepsilon, NT_{0}} > 0,
		\end{split}
		\end{equation}
		for all $s\in [0,T_0]$ such that $X(s;T_{0},x,v)\in\p\O$.
	\end{lemma}
	\begin{proof}
		By definition of $\mathfrak{IB}$ and Lemma~\ref{infinite_bounces_set}, 
		\[
		(X(s;T_{0},x,v), V(s;T_{0},x,v)) \notin \gamma_{0}^{I}.
		\]
		Therefore, 
		\[
			\min_{ t_{j} \in[0,T_{0}] } D_{\mathcal{I}}(t^{j}(x,v), x, v) > 0,
		\]
		where $D_{\mathcal{I}}(t^{j}(x,v), x, v)$ is defined in (\ref{dist def}). To derive uniform positivity, we use compactness argument again. From Lemma~\ref{uniformbound}, for $(x,v)\in \{ cl({\O}) \times \VN \} \backslash \mathfrak{IB} $, we know that
		\[
			\mathfrak{N}(x,v,NT_{0}) \leq C_{\varepsilon,NT_{0}}.
		\]
		Therefore, 
		\begin{equation} \label{fixed xv}
			\min_{ t_{j} \in[0,T_{0}] } D_{\mathcal{I}}(t^{j}(x,v), x, v) = \min_{ 1 \leq {j} \leq C_{x,v,\varepsilon,NT_{0}} } D_{\mathcal{I}}(t^{j}(x,v), x, v) := \rho_{x,v,\varepsilon,NT_{0}} > 0.
		\end{equation}
		for some uniform positive constant $\rho_{x,v,\varepsilon,NT_{0}} > 0$. Now we split into two cases.  \\
		
		\noindent \textit{Case 1.} If $(X(s;T_{0},x,v), V(s;T_{0},x,v)) \notin \gamma_{0}$, we have local continuity from Lemma~\ref{local conti}, so there exist $r_{x,v,\varepsilon, NT_{0}} \ll 1$ such that if $|(x,v)-(y,u)| < r_{x,v,\varepsilon, NT_{0}}$, 
		\begin{equation} \label{near distance}
		\min_{ 1 \leq {j} \leq C_{x,v,\varepsilon,NT_{0}} } \big| D_{\mathcal{I}}(t^{j}(x,v), x, v) -  D_{\mathcal{I}}(t^{j}(y,u), y, u) \big| < \frac{\rho_{x,v,\varepsilon, NT_{0}}}{2}.
		\end{equation}
		From (\ref{fixed xv}) and (\ref{near distance}),
		\[
			\min_{ 1 \leq {j} \leq C_{x,v,\varepsilon,NT_{0}} } D_{\mathcal{I}}(t^{j}(y,u),y,u) > \frac{\rho_{x,v,\varepsilon, NT_{0}}}{2},
		\]
		which implies uniform nonzero on a ball $cl\big( B((x,v),r_{x,v,\varepsilon, NT_{0}}) \big) $. By compactness, we have a finite open cover for $\{ cl({\O}) \times \VN \} \backslash \mathfrak{IB}$, which is written by $\bigcup_{i=1}^{\ell} B((x_{i},v_{i}),r_{x_{i},v_{i},\varepsilon, NT_{0}})$ for some finite $q$. Finally, we pick uniform positive number
		\[
		\rho_{\varepsilon, NT_{0}} := \min_{1\leq i\leq \ell}\frac{\rho_{x_{i},v_{i},\varepsilon, NT_{0}}}{2} > 0,
		\]
		to finish the proof.  \\
		
		\noindent \textit{Case 2.} If $(X(s;T_{0},x,v), V(s;T_{0},x,v)) \in \gamma_{0}$ for some $s\in[0,T_{0}]$, it must be concave grazing by definition of $\mathfrak{IB}$ and consider consecutive concave grazing cases of (\textit{Case 2-1}) in the proof of Lemma~\ref{uniformbound} again with Figure 1. Let us assume (\ref{cases}).  \\
		
		When $j < i$, using Lemma~\ref{local conti}, we have $r_{x,v,\varepsilon, NT_{0}} \ll 1$ such that if $|(x,v)-(y,u)| < r_{x,v,\varepsilon, NT_{0}}$, 
		\begin{equation*}  \label{before Cc}
		\min_{ 1 \leq {j} \leq i-1 } \big| D_{\mathcal{I}}(t^{j}(x,v), x, v) -  D_{\mathcal{I}}(t^{j}(y,u), y, u) \big| < \frac{\rho_{x,v,\varepsilon, NT_{0}}}{2}.  \\
		\end{equation*}
		
		When $i\leq j \leq k-1$, it is not reasonable to compare with same bouncing index, because we have discontinuity by convex grazing. However, since $D_{\mathcal{I}}$ is uniformly bounded from below by (\ref{fixed xv}), we suffice to compare $D_{\mathcal{I}}(t^{j}(y,u),y,u)$ with the nearset $D_{\mathcal{I}}(t^{\ell}(x,v),x,v)$ for some $j\leq \ell$. \\
		
		\noindent $(i)$ If $\overline{x^{i-1}(x,v) x^{k}(x,v)}$ does not bounce near $x^{j}(x,v)$ for all $j\in\{i,\cdots,k-1\}$, then from Lemma~\ref{local conti} again, we can redefine $r_{x,v,\varepsilon,NT_{0}} \ll 1$ so that if $|(x,v) - (y,u)| < r_{x,v,\varepsilon,N}$,
		\begin{equation*}
		\begin{split}
			& \big| D_{\mathcal{I}}(t^{k}(x,v), x, v) -  D_{\mathcal{I}}(t^{i}(y,u), y, u) \big| < \frac{\rho_{x,v,\varepsilon, NT_{0}}}{2},  \\
		\end{split}
		\end{equation*}
		holds. This implies
		\begin{equation} \label{unif small}
		\begin{split}
			& D_{\mathcal{I}}(t^{i}(y,u), y, u) \geq  D_{\mathcal{I}}(t^{k}(x,v), x, v) - \frac{\rho_{x,v,\varepsilon, NT_{0}}}{2} > \frac{\rho_{x,v,\varepsilon, NT_{0}}}{2},
		\end{split}
		\end{equation}
		from (\ref{fixed xv}).  \\
		
		\noindent $(ii)$ From Lemma~\ref{local conti}, there exist $r_{x,v,\varepsilon,NT_{0}} \ll 1$ so that if $|(x,v) - (y,u)| < r_{x,v,\varepsilon,N}$, $|x^{i}(y,u) - x^{\ell}(x,v)| = O(r_{x,v,\varepsilon,N})$. Moreover, from (\ref{direc compare}), $|v^{i}(y,u) - v^{\ell}(x,v)| = O(r_{x,v,\varepsilon,N})$ also holds. 
		\begin{equation*}
		\begin{split}
		& \big| D_{\mathcal{I}}(t^{\ell}(x,v), x, v) -  D_{\mathcal{I}}(t^{i}(y,u), y, u) \big| < \frac{\rho_{x,v,\varepsilon, NT_{0}}}{2},  \\
		\end{split}
		\end{equation*}
		holds and therefore, (\ref{unif small}) also holds by (\ref{fixed xv}).  \\
		
		\noindent $(iii)$ Obviosly, $|x^{i}(y,u) - x^{\ell}(x,v)| = 0$ and $|v^{i}(y,u) - v^{\ell}(x,v)| = O(r_{x,v,\varepsilon,N})$ also holds by (\ref{dir comp3}), so yields (\ref{unif small}), similarly.  \\
		
		\noindent $(iv)$ Near $x^{i}(y,u)$ (near $x^{\ell}(x,v)$) and $x^{i+1}(y,u)$ (near $x^{\ell+1}(x,v)$), we use argument of $(ii)$ for both bouncings to claim that .
		\begin{equation*}
		\begin{split}
			D_{\mathcal{I}}(t^{i}(y,u), y, u), \ D_{\mathcal{I}}(t^{i+1}(y,u), y, u) &\geq \frac{\rho_{x,v,\varepsilon, NT_{0}}}{2},  
		\end{split}
		\end{equation*}
		if $|(x,v) - (y,u)| < r_{x,v,\varepsilon,NT_{0}}$ , for some small $r_{x,v,\varepsilon,NT_{0}} \ll 1$.  \\
		
		From \textit{Step 2} in proof of Lemma~\ref{uniformbound}, number of interval of consecutive grazing is uniformly bounded becasue we assume Definition~\ref{AND}. And whenever we encounter consecutive grazing, we can split into cases $(i) \sim (iv)$ to gain unifrom positivity of $D_{\mathcal{I}}(t^{j}(y,u),y,u)$ for $0\leq t^{j}(y,u)\leq T_{0}$. And then we apply compactness argument of \textit{Case 1} in the proof of this Lemma to finish the proof.  \\

	\end{proof}
	
	\subsection{Dichotomy of sticky grazing} 	

	\begin{lemma} \label{dichotomy}
		Assume $\Omega\subset \mathbb{R}^2$ as defined in Definition~\ref{AND}. Assume that $(\alpha_{j}(\tau), \alpha^{\prime}_{j}(\tau))\in \gamma_{0}^{C}$ for some $j\in\{1,\cdots,M\}$ and $\tau\in (\tau^{*}-\delta, \tau^{*}+\delta)\subset [a_{j}, b_{j}]$. Also we assume that 
		\[
		(X(s;T_{0},\alpha_{j}(\tau), \alpha^{\prime}_{j}(\tau)), V(s;T_{0},\alpha_{j}(\tau), \alpha^{\prime}_{j}(\tau))) \notin \gamma_{0}
		\]
		for $s\in[0,T_{0}]$. Let us simplify notation:
		\[
			x^{i}(\tau) := x^{i}(\alpha_{j}(\tau), \alpha^{\prime}_{j}(\tau)),\quad v^{i}(\tau) := v^{i}(\alpha_{j}(\tau), \alpha^{\prime}_{j}(\tau)),\quad t^{i}(\tau) := t^{i}(\alpha_{j}(\tau), \alpha^{\prime}_{j}(\tau)),
		\]
		for $\tau\in (\tau^{*}-\delta, \tau^{*}+\delta)\subset [a_{j}, b_{j}]$. Then we have the folloiwng dichotomy. For each $k$,  \\
		(a) There exist unique $x^{*}\in cl(\O)$ such that $x^{*} \in \overline{x^{k}(\tau) x^{k+1}(\tau)}$ for all $\tau\in (\tau^{*}-\delta, \tau^{*}+\delta)\subset [a_{j}, b_{j}]$.  \\
		(b) For each $x\in cl(\O)$, the following set is finite
		\[
		\Big\{ \frac{v^{k}(\tau)}{|v^{k}(\tau)|} \in \mathbb{S}^{1} \ : \ x\in \overline{x^{k}(\tau) x^{k+1}(\tau)}, \ \tau\in (\tau^{*}-\delta, \tau^{*}+\delta) \Big\}.
		\] 
	\end{lemma}
	\begin{proof}
		Assume that we have some $x^{*}$ satisfying (a). If there exist another $y^{*}\neq x^{*}$,
		\[
			x^{k}(\tau) - x^{*} = |x^{k}(\tau) - x^{*}| \frac{v^{k}(\tau)}{|v^{k}(\tau)|},\quad x^{k}(\tau) - y^{*} = |x^{k}(\tau) - y^{*}| \frac{v^{k}(\tau)}{|v^{k}(\tau)|},\quad \tau \in (\tau^{*}-\delta, \tau^{*}-\delta).
		\]
		This gives
		\[
			x^{*} - y^{*} = \Big( |x^{k}(\tau) - y^{*}| - |x^{k}(\tau) - x^{*}| \Big) \frac{v^{k}(\tau)}{|v^{k}(\tau)|}.
		\]
		Therefore, $\frac{v^{k}(\tau)}{|v^{k}(\tau)|}$ is constant unit vector for $\tau \in (\tau^{*}-\delta, \tau^{*}-\delta)$. And since $(x^{k}(\tau), v^{k}(\tau)$ is not grazing, $x^{k}(\tau)$ is also constant for all $\tau \in (\tau^{*}-\delta, \tau^{*}-\delta)$. Since trajectory is deterministic forward/backward in time, $\dot{\alpha}_{j}(\tau)$ should be constant for $\tau \in (\tau^{*}-\delta, \tau^{*}-\delta)$ which implies $\alpha_{j}(\tau)$ is a part of straight line locally. This is contradiction, because $\O$ is analytic bounded domain. \\
		
		If there does not exist $x^{*}$ which satisfies (a) for $\tau \in (\tau^{*}-\delta, \tau^{*}-\delta)$, 
		\[
			\Big\{ \tau \in (\tau^{*}-\delta, \tau^{*}-\delta) \
			 \big\vert \ x^{k}(\tau) - x^{*} = |x^{k}(\tau) - x^{*}| \frac{v^{k}(\tau)}{|v^{k}(\tau)|} \Big\}
		\]
		is a finite set for any $x^{*}\in cl(\O)$ by rigidity of analytic function. This yields (b).  \\
	\end{proof}

	\subsection{Grazing set}
	
	In this section, we characterize the points of $\{ cl({\O}) \times \VN \} \backslash \mathfrak{IB}$ whose specular backward cycle grazes the boundary (hits the boundaries tangentially) at some moment. By definition of $\mathfrak{IB}$, this grazing cannot be inflection grazing $\gamma_{0}^{I}$. Moreover, Lemma~\ref{infinite_bounces_set} guarantees that convex grazing does not happen neither. Therefore, the only possible grazing is concave grazing $\gamma_{0}^{C}$. We will classify this concave grazing sets depending on the first(backward in time) concave grazing time. \\
	
	\begin{definition} \label{grazing sets}
		For $T_0>0$ and $(x,v)\in cl(\Omega)\times \mathbb{R}^{2}$, we define grazing set:
		\begin{eqnarray*}
		\mathfrak{G} \ := \ \Big\{ \ (x,v) \ \in \{ cl({\O}) \times \VN \} \backslash \mathfrak{IB} \ : \exists \ s\in[0,T_{0}) \ s.t \  \big( X(s;T_{0},x,v), V(s;T_{0},x,v) \big) \in \gamma_{0} \ \Big\} \ , 
		\end{eqnarray*}
	which is a set of phase $(x,v)$ whose trajectory grazes at least once for time interval $[0,T_{0}]$. We also define $\mathfrak{G}^{C}$, $\mathfrak{G}^{V}$, and $\mathfrak{G}^{I}$ by it grazing type, i.e. 
	\begin{equation*}
	\begin{split}
		\mathfrak{G}^{C} \ &:= \ \Big\{ \ (x,v) \ \in \{ cl({\O}) \times \VN \} \backslash \mathfrak{IB} \ : \exists \ s\in[0,T_{0}) \ s.t \  \big( X(s;T_{0},x,v), V(s;T_{0},x,v) \big) \in \gamma_{0}^{C} \ \Big\} \ ,   \\
		\mathfrak{G}^{V} \ &:= \ \Big\{ \ (x,v) \ \in \{ cl({\O}) \times \VN \} \backslash \mathfrak{IB} \ : \exists \ s\in[0,T_{0}) \ s.t \  \big( X(s;T_{0},x,v), V(s;T_{0},x,v) \big) \in \gamma_{0}^{V} \ \Big\} \ ,   \\
		\mathfrak{G}^{I} \ &:= \ \Big\{ \ (x,v) \ \in \{ cl({\O}) \times \VN \} \backslash \mathfrak{IB} \ : \exists \ s\in[0,T_{0}) \ s.t \  \big( X(s;T_{0},x,v), V(s;T_{0},x,v) \big) \in \gamma_{0}^{I} \ \Big\} \ .   \\
	\end{split}
	\end{equation*}
	By definitin of $\mathfrak{IB}$, we know that $\mathfrak{G}^{V} = \mathfrak{G}^{I} = \emptyset$. Therefore, we rewrite and decompose $\mathfrak{G}$ as
	\begin{equation*} \label{GC}
	\mathfrak{G} = \mathfrak{G}^{C} \ := \ \bigcup_{j}  \ \mathfrak{G}^{C,j} \
	:= \bigcup_{l=1}^{M^{C}}  \ \mathfrak{G}^{C}_{l} \ = \  \bigcup_{j} \ \bigcup_{l=1}^{M^C}  \ \mathfrak{G}^{C,j}_{l} \ ,
	\end{equation*}
	where 
	\begin{eqnarray*}
	\mathfrak{G}^{C,j} & := & \Big\{ (x,v) \in \mathfrak{G}^{C} \ : \
	( x^{j}(x,v), v^{j}(x,v) ) \in \gamma_{0}^{C} 
	\Big\} \label{Gcj} ,\\
	\mathfrak{G}^{C}_{l} & := & \Big\{ (x,v) \in \mathfrak{G}^{C} \ : \
	\exists k \ s.t. \ ( x^{k}(x,v), v^{k}(x,v) ) \in \gamma_{0}^{C} \ \text{and} \ x^{k}(x,v)\in\p\O^{C}_{l} 
	\Big\} \label{Gcl} ,\\
	\mathfrak{G}^{C,j}_{l} & := & \Big\{ (x, v) \in \mathfrak{G}^{C,j} \ : \
	x^{j}(x,v) \ \in \ \p\O^{C}_{l}
	\Big\} \label{GCjl} ,
	\end{eqnarray*}
	where $l\in\{1,\cdots, M^{C}\}$ which is defined in (\ref{renumber concave}).

	\end{definition}

	\begin{remark}
		Let us use renumbered notation (\ref{renumber concave}) and the sets defined in Definition~\ref{grazing sets}. If $({x}, {v})\in \mathfrak{G}^{C}_{l}$ then there exists $\tau\in (\bar{a}_{l}, \bar{b}_{l})$ and $k$ such that $(x^{k}(x,v), v^{k}(x,v)) \in \gamma^{C}_{0}$ and $x^{k}(x,v) = \bar{\alpha}_{l}(\tau)$. Due to Lemma \ref{uniformawayfrominflection}, such $\tau$ cannot be arbitrarily close to the end points $\bar{a}_{l}, \ \bar{b}_{l}$ which are inflection points $\kappa = 0$. Lemma~\ref{uniformawayfrominflection} implies that there exists $\bar{a}^{*}_{l} > \bar{a}_{l}$ and $\bar{b}_{l}^{*} < \bar{b}_{l}$ for each $l\in\{1,\cdots, M^{C}\}$ such that
		\begin{equation} \label{staredinterval}
		\Big\{ \ \tau \in (\bar{a}_{l}, \bar{b}_{l}) \ : \ \big( X(s; T_0, {x}, {v}) , V(s; T_0, {x}, {v}) \big) \in \gamma_{0}^{C}, \ \ X(s; T_0, {x}, {v}) = \bar{\alpha}_{l}(\tau) \ \ \text{for} \ \ ({x}, {v})\in \mathfrak{G}^{C}_{l} \
		\Big\} \ \subset \ [ \ \bar{a}^{*}_{l}, \ \bar{b}_{l}^{*} \ ] \ . 
		\end{equation}
	\end{remark}
	
	Throughout this subsection, we use some temporary symbols. Inspired by (\ref{specular_cycles}), we can also define $k$-th backward/forward exit time:
	\begin{equation*} \label{k-bf}
	\begin{split} 
	\tb(x,v)  &:=  \tb^{1}(t,x,v),  \\
	\tb^{k}(x,v)  &:=  t - t^{k}(t,x,v),  \\
	\xb^{k}(x,v)  &:=  x^{k}(t,x,v),  \\
	\tf(x,v)  &:=  \tf^{1}(t,x,v),  \\
	\tf^{k}(x,v)  &:= - t^{k}(0,x,-v),  \\
	\xf^{k}(x,v)  &:=  x^{k}(t,x,-v).  \\
	\end{split}
	\end{equation*}

	\subsubsection{$\mathbf{1^{\text{st}}-}$\textbf{Grazing Set}, $\mathfrak{G}^{C,1}$} 
	Let us use renumbered notation for concave part (\ref{renumber concave}). From the definition of $\mathfrak{G}_{l}^{C,1}$ and (\ref{staredinterval}),
	\begin{eqnarray*}
		\mathfrak{ {G}}_{l}^{C,1} &\subset&
		\bigcup_{p=\pm 1}\Big\{ \big( \bar{\alpha}_{l}(\tau)+s p |v|\dot{\bar{\alpha}}_{l}(\tau),  p |v| \dot{\bar{\alpha}}_{l}(\tau) \big) \in \big\{ \{ cl({\O}) \times \VN \} \backslash \mathfrak{IB} \big\} \ :   \\
		&&\quad \quad  \ \tau\in [\bar{a}_{l}^{*}, \bar{b}_{l}^{*} ] \ , \ v\in \VN \ , \ s\in[0,t_{\mathbf{f}}( \bar{\alpha}_{l}(\tau), p |v|\dot{\bar{\alpha}}_{l}(\tau))]
		\Big\}.	\\
	\end{eqnarray*}
	Since the signed curvature $\kappa$ is positive and bounded,but finite points, $\mathbb{S}^1 \cap \{ v\in\mathbb{R}^2 : ({x},{v})\in \mathfrak{{G}}_{l}^{C,1} \}$ has at most two points for fixed $x$. Since $M^{C}$ is uniformly bounded, $\mathbb{S}^1 \cap \{ v\in\mathbb{R}^2 : ({x},{v})\in \mathfrak{{G}}_{l}^{C,1} \}$ contains at most $2\times M^C$ points and therefore,
	\begin{equation} \label{g1zero}
	\mathfrak{m}_{2} \{ v\in\mathbb{R}^2 : ({x},{v})\in \mathfrak{{G}}_{l}^{C,1} \} \ = \ 0 \ .
	\end{equation}
		
	
	\begin{lemma}\label{G_C1}
		For any $\varepsilon>0$, there exist an open cover $\bigcup_{i=1}^{l_{1}} B(x^{C,1}_{i} , r^{C,1}_{i})$ for $\mathcal{P}_{x}\big( \{ cl({\O}) \times \VN \} \backslash \mathfrak{IB} \big)$, where $\mathcal{P}_{x}$ is projection on sptial space, and corresponding velocity set $\mathcal{O}^{C,1}_{i} \subset \VN$ with $\mathfrak{m}_{2}(\mathcal{O}^{C,1}_{i}) < \varepsilon$ such that \\
		(1) For any $({x},{v})\in \{ cl({\O}) \times \VN \} \backslash \mathfrak{IB}$, there exists ${x}^{C,1}_{i}$, ${r}^{C,1}_{i}$, and $\delta^{C,1} > 0$ such that ${x}\in B({x}^{C,1}_{i}, r^{C,1}_{i})$ and   \\
		(2) $\phi^{1}(x,v)  = | {v} \cdot \mathbf{n}(x_{\mathbf{b}}({x},{v})) \ | > \delta^{1} > 0$ holds for ${v}\in \VN\backslash\mathcal{O}^{C,1}_{i} $, for some uniformly positive $\delta^{1} > 0$.  \\
			
		From above, we define $\varepsilon-$ neighborhood of $\mathfrak{G}^{C,1}$:
		\begin{equation*}
		(\mathfrak{G}^{C,1})_{\varepsilon}  \ := \ \bigcup_{i=1}^{l_1} \ \{ B( {x}^{C,1}_{i} , {r}^{C,1}_{i} )\times \ \mathcal{O}^{C,1}_{i} \}
		\ .
		\end{equation*}
	\end{lemma}
	\begin{proof}
		Let $x\in\mathcal{P}_{x}\big( \{ cl({\O}) \times \VN \} \backslash \mathfrak{IB} \big)$. Then, there exist at most $2M^{C}$ distinct unit velocity $\frac{v_{i}}{|v_{i}|}, \ i\in\{1,\cdots, 2M^{C}\}$ such that $(x,v_{i}) \in \mathfrak{G}^{C,1}$. We define
		\begin{equation} \label{G1 O def}
		\mathcal{O}^{C,1}_{x} := \big\{ v\in\VN \ : \ \Big|\frac{v_{i}}{ |v_{i}|} - \frac{v}{|v|} \Big| < C_{1}(N)\varepsilon, \ \forall i\in\{1,\cdots,2M^{C}\} \big\} .  \\
		\end{equation}
		When $v\in \VN\backslash\mathcal{O}^{C,1}_{x}$, we can apply Lemma~\ref{local conti} to show that
		\begin{equation*} \label{phi1}
		\phi^1( {x}, {v}) \ := \ | {v} \cdot \mathbf{n}(x_{\mathbf{b}}( {x}, {v}))|
		\end{equation*}
		is well-defined and locally smooth, since $({x},v)\in \big\{ \{ cl({\O}) \times \VN \} \backslash \mathfrak{IB} \big\} \backslash \ \mathfrak{G}^{C,1}$. Using local continuity of Lemma~\ref{local conti} again, we can find $r^{C,1}_{x} \ll 1$ such that
		\[
			\phi^{1}(x,v) > \delta^{1}_{x} > 0, \quad\text{for}\quad (x,v)\in cl \big( B(x, r^{C,1}_{x}) \big) \times \VN\backslash\mathcal{O}^{C,1}_{x}.
		\]
		By compactness, we can find finite open cover $\bigcup_{i=1}^{l_{1}} B(x^{C,1}_{i} , r^{C,1}_{i})$ for $ \mathcal{P}_{x}\big( \{ cl({\O}) \times \VN \} \backslash \mathfrak{IB} \big) $ and corresponding $\mathcal{O}^{C,1}_{i}$ with small measure $\mathfrak{m}_{2}(\mathcal{O}^{C,1}_{i}) < \varepsilon$ by choosing (\ref{G1 O def}) with some proper small $C_{1}(N)$. Finally we choose
		\[
		\delta^{1} := \min_{1\leq i\leq l_{1}} \delta^{1}_{ x^{C,1}_{i} } > 0, 
		\]
		to finish the proof.   \\
	\end{proof}
	
	\subsubsection{$\mathbf{2^{\text{nd}}-}$\textbf{Grazing Set} $\mathfrak{G}^{C,2} $ } 
	From the definition of $\mathfrak{G}^{C,2}$ and (\ref{staredinterval}), the set $\mathfrak{G}^{C,2} \ \backslash \ (\mathfrak{G}^{C,1})_{\varepsilon}$ is a subset of
	\begin{equation} \label{g2}
	\begin{split}
	& \bigcup_{l=1}^{M^{C}}\bigcup_{p=\pm 1}\Big\{ \Big( \xf^{1}( \bar{\alpha}_{l}(\tau), p|v|\dot{\bar{\alpha}}_{l}(\tau)) + s \vf^{1}(\bar{\alpha}_{l}(\tau), p |v|\dot{\bar{\alpha}}_{l}(\tau)),  \vf^{1}(\bar{\alpha}_{l}(\tau), p |v| \dot{\bar{\alpha}}_{l}(\tau)) \Big) \in \{ cl({\O}) \times \VN \} \backslash \mathfrak{IB}  \ :  \\
	&\quad\quad\quad\quad \tau\in[\bar{a}_{l}^{*}, \bar{b}_{l}^{*}], \ v\in\VN, \ s\in[0, (\tf^{2} - \tf^{1})(\bar{\alpha}_{l}(\tau), p|v|\dot{\bar{\alpha}}_{l}(\tau)) ]  \Big\} \Big{\backslash} (\mathfrak{G}^{C,1})_{\varepsilon} . \\
	\end{split}
	\end{equation}
	Without loss of generality, we suffice to consider only $p=1$ case of (\ref{g2}), since $p=-1$ does not change any argument.  \\
	
	\noindent{\textbf{Step 1}} Fix $p=1$ and $l\in\{1,\cdots,M^{C}\}$. First, we remove $1^{st}$-grazing set by complementing $(\mathfrak{G}^{C,1})_{\varepsilon}$. \\
	Let us consider $(\bar{x}, \bar{v}) \in \mathfrak{G}^{C,1} \cap \mathfrak{G}^{C,2}_{l}$ and we write $\bar{\alpha}_{l}(\bar{\tau}) = x^{2}(\bar{x}, \bar{v})$. Then, from Lemma~\ref{G_C1} and Lemma~\ref{local conti}, there exist $i\in\{1,\cdots, l_{1}\}$ such that $(\bar{x}, \bar{v}) \in B(x^{C,1}_{i}, r^{C,1}_{i})\times\mathcal{O}^{C,1}_{i}$ and 
	\begin{equation} \label{near double graze}
	\begin{split}
	&\big\{ x^{2}(x,v)\in\p\O_{l}^{C} \ : \ \forall(x,v) \in cl\big( { B(x^{C,1}_{i}, r^{C,1}_{i})\times\mathcal{O}^{C,1}_{i} } \big) , \  \text{where} \ (\bar{x}, \bar{v}) \in B(x^{C,1}_{i}, r^{C,1}_{i})\times\mathcal{O}^{C,1}_{i} \big\} \\
	&\quad\subset \big[ \bar{\alpha}_{l}(\bar{\tau} - \delta_{-}), \bar{\alpha}_{l}(\bar{\tau} + \delta_{+}) \big],\quad \text{for}\quad 0 < \delta_{\pm} = O(r_{1,i}, \varepsilon) \ll 1.
	\end{split}
	\end{equation}
	Excluding (\ref{near double graze}) from $[\bar{a}_{l}^{*}, \bar{b}_{l}^{*}]$ for all $(\bar{x}, \bar{v})\in \mathfrak{G}^{C,1} \cap \mathfrak{G}^{C,2}_{l}$ yields a union of countable open connected intervals $\mathcal{I}$, i.e.
	\begin{equation*} \label{def I}
		\mathcal{I} := [ \bar{a}_{l}^{*} , d_{l,1} ) \cup  (  c_{l,2} ,  d_{l,2}  )  \cup \cdots  \subset [\bar{a}_{l}^{*}, \bar{b}_{l}^{*}],\quad \bar{a}_{l}^{*} < d_{1} < c_{2} < d_{2} < \cdots.
	\end{equation*}
	Now we claim that $\mathcal{I}$ contains only finite subintervals. If this union is not finite, there exist infinitly many distinct $\{\tau_{i}\}_{i=1}^{\infty}, \tau_{1} < \tau_{2} < \cdots $ such that   
	\[
	\dot{\bar{\alpha}}_{l}(\tau_{i})\cdot \mathbf{n}(x^{\mathbf{f}}(\bar{\alpha}_{l}(\tau_{i}), \dot{\bar{\alpha}}_{l}(\tau_{i}))) = 0, \quad i\in\mathbb{N}.
	\]
	We pick monotone increasing sequence $\tau_{1}, \tau_{2}, \cdots, \tau_{n},\cdots$ by choosing a point $\tau_{i}$ for each disjont closed interval. Since $\tau_{n} \leq b^{l}_{*}$ for all $n\in\mathbb{N}$, there exist a $\tau_{\infty}$ such that $\tau_{n} \rightarrow \tau_{\infty}$ up to subsequence. Let us assume that 
	\[
		\big( \xf(\bar{\alpha}_{l}(\tau_{n}), \dot{\bar{\alpha}}_{l}(\tau_{n})), \dot{\bar{\alpha}}_{l}(\tau_{n}) \big) \in \gamma_{0}^{C},\quad \xf(\bar{\alpha}_{l}(\tau_{n}), \dot{\bar{\alpha}}_{l}(\tau_{n})) \in \p\O^{C}_{p}.
	\]
	Since we have chosen $\tau_{n}$'s from each distinct intervals, there exist $\tau', \ \tau_{n} < \tau' < \tau_{n+1}$ such that
	\[
	\big(\xf(\bar{\alpha}_{l}(\tau'), \dot{\bar{\alpha}}_{l}(\tau')), \dot{\bar{\alpha}}_{l}(\tau') \big) \notin \gamma_{0}^{C}.
	\]
	By monotonicity of $\{\tau_{1},\cdots,\tau_{\infty}\}$ the fact that $\tau_{\infty}$ is accumulation implies that we have accumulating concave grazing phase $\{ \big( \xf(\bar{\alpha}_{l}(\tau_{n}), \dot{\bar{\alpha}}_{l}(\tau_{n})), \dot{\bar{\alpha}}_{l}(\tau_{n}) \big) \}_{i=1}^{\infty}$ near $\{ \big( \xf(\bar{\alpha}_{l}(\tau_{\infty}), \dot{\bar{\alpha}}_{l}(\tau_{\infty})), \dot{\bar{\alpha}}_{l}(\tau_{\infty}) \big) \}$. This is contradiction because $\p\O$ is analytic domain. Finally we can write $\mathcal{I}$ as disjoint union of finite $m^{l}_{2}$ intervals, i.e.
	\begin{equation} \label{interval split}
		\mathcal{I} := [ \bar{a}_{l}^{*} , d_{l,1} ) \cup  (  c_{l,2} ,  d_{l,2}  )  \cup \cdots  \cup  (  c_{l,m^{l}_{2}-1}  ,  d_{l,m^{l}_{2}-1}  ) \cup  (  c_{l,m^{l}_{2}}  , \bar{b}_{l}^{*} ].  \\
	\end{equation}
	
	\noindent{\textbf{Step 2}} Since we have chosen $\delta_{\pm}$ as nonzero in (\ref{near double graze}), we can include boundary points of each subinterval of (\ref{interval split}). Therefore, $\mathfrak{G}^{C,2} \ \backslash \ ( \mathfrak{G}^{C,1} )_{\varepsilon}$ is a subset of   \\
	\begin{equation} \label{new G2C}
	\begin{split}
	& \bigcup_{l=1}^{M^{C}} \Big\{ \Big( \xf^{1}(\bar{\alpha}_{l}(\tau), |v|\dot{\bar{\alpha}}_{l}(\tau)) + s \vf^{1}(\bar{\alpha}_{l}(\tau), |v|\dot{\bar{\alpha}}_{l}(\tau)), \vf^{1}(\bar{\alpha}_{l}(\tau), |v|\dot{\bar{\alpha}}_{l}(\tau)) \Big) \in \{ cl({\O}) \times \VN \} \backslash \mathfrak{IB}  \ :  \\
	&\quad\quad\quad\quad \tau \in [\bar{a}_{l}^{*} , d_{l,1}]\cup [c_{l,2},d_{l,2}]\cup\cdots\cup [c_{l,m^{l}_{2}-1}, d_{l,m^{l}_{2}-1}] \cup [c_{l,m^{l}_{2}}, \bar{b}_{l}^{*} ] ,  \\
	&\quad\quad\quad\quad v\in\VN, s\in[0, (\tf^{2} - \tf^{1})( \bar{\alpha}_{l}(\tau), |v|\dot{\bar{\alpha}}_{l}(\tau)) ]  \Big\}   \\
	\end{split}
	\end{equation}
	and for all $\tau \ \in \ [  \bar{a}_{l}^{*} , d_{l,1}]\cup [c_{l,2},d_{l,2}]\cup\cdots\cup [c_{l,m^{l}_{2}-1}, d_{l,m^{l}_{2}-1}] \cup [c_{l,m^{l}_{2}}, \bar{b}_{l}^{*} ]$,
	\begin{eqnarray*}
		|\dot{\bar{\alpha}}_{l}(\tau) \cdot \mathbf{n}(\xf(\bar{\alpha}_{l}(\tau),  \dot{\bar{\alpha}}_{l}(\tau))) | > \delta^{C,1} > 0,
	\end{eqnarray*}
	where $\delta^{C,1}$ was found in Lemma~\ref{G_C1}.
	Moreover, we can choose these subintervals so that measure of each punctures $\{ (d_{l,i}, c_{l,i+1}) \}_{i=1}^{m_{2}^{l}-1}$ are arbitrary small, because we can choose $\delta_{\pm} > 0$ arbitrary small in (\ref{near double graze}).  \\ 

	\noindent{\textbf{Step 3}} We construct $\mathbf{2^{\text{nd}}-}$\textbf{Sticky Grazing Set} $\mathcal{ {SG}}_{l}^{C,2}$ where all grazing rays from non-measure zero subset of $[ \bar{a}_{l}^{*} , d_{l,1} ) \cup  (  c_{l,2} ,  d_{l,2}  )  \cup \cdots  \cup  (  c_{l,m^{l}_{2}-1}  ,  d_{l,m^{l}_{2}-1}  ) \cup  (  c_{l,m^{l}_{2}}  , \bar{b}_{l}^{*} ]$ intersect at a fixed point in $\mathcal{P}_{x}\Big( \{ cl({\O}) \times \VN \} \backslash \mathfrak{IB} \Big)$ where $\mathcal{P}_{x}$ is projection on sptial domain. Choose any $i\in\{1,\cdots, m^{l}_{2}\}$ and corresponding sub interval $[c_{l,i}, d_{l,i}]$. We define   \\
	\begin{equation*} \label{line temp}
	\begin{split}
	\mathfrak{G}^{C,2}_{l,i} &:= \Big\{ \Big( \xf^{1}(\bar{\alpha}_{l}(\tau), |v|\dot{\bar{\alpha}}_{l}(\tau)) + s \vf^{1}(\bar{\alpha}_{l}(\tau), |v|\dot{\bar{\alpha}}_{l}(\tau)), \vf^{1}(\bar{\alpha}_{l}(\tau), |v|\dot{\bar{\alpha}}_{l}(\tau)) \Big) \in \{ cl({\O}) \times \VN \} \backslash \mathfrak{IB} \ :  \\
	&\quad\quad\quad\quad \tau \in [c_{l,i}, d_{l,i}], v\in\VN, \ s\in[0, (\tf^{2} - \tf^{1})(\bar{\alpha}_{l}(\tau), |v|\dot{\bar{\alpha}}_{l}(\tau)) ]  \Big\}.  \\
	\end{split}
	\end{equation*}
	Fix $x^{*}\in \overline{\O}$. If there does not exist $\tau\in [c_{l,i}, d_{l,i}]$ and $s\in[t_1^{\mathbf{f}}(\alpha^l(\tau), \dot{\alpha}^l(\tau)), \ t_2^{\mathbf{f}}(\alpha^l(\tau), \dot{\alpha}^l(\tau))]$ satisfying $x^{*} = \xf^{1}(\alpha^l(\tau), \dot{\alpha}^l(\tau))+s\vf^{1}(\alpha^l(\tau), \dot{\alpha}^l(\tau))$ then $\{ v\in\mathbb{R}^2 : ({x}^{*},{v})\in \mathfrak{G}_{l,i}^{C,2} \} = \emptyset$ with zero measure. Now suppose that there exist such $\tau$ and $s$. \\
	\indent Due to Lemma~\ref{dichotomy}, there are only two cases: (i) \textbf{sticky grazing:} for all $\tau\in [c_{l,i}, d_{l,i}]$, there exists $s = s(\tau)\in[\tf^{1}(\alpha^l(\tau), \dot{\alpha}^l(\tau)), \ \tf^{2}(\alpha^l(\tau), \dot{\alpha}^l(\tau))]$ and fixed $x^{*}\in cl(\O)$ such that 
	\begin{equation} \label{z2}
		x^{*} = \xf^{1}(\alpha^l(\tau), \dot{\alpha}^l(\tau))+s\vf^1 (\alpha^l(\tau)),
	\end{equation}
	or (ii) \textbf{isolated grazing:} there exists $\delta_-, \delta_+>0$ so that for $\tau^{\prime}\in(\tau-\delta_-,\tau+\delta_+)\backslash\{\tau\}$, there is no $s$ satisfying (\ref{z2}). We define $2^{\text{nd}}-$ sticky grazing set $\mathcal{SG}^{C,2}$ as collection of all such $x^{*}\in cl(\O)$ points , i.e.
	\begin{definition}
		Consider (\ref{new G2C}) and disjoint union of intervals $[\bar{a}_{l}^{*} , d_{l,1}]\cup [c_{l,2},d_{l,2}]\cup\cdots\cup [c_{l,m^{l}_{2}-1}, d_{l,m^{l}_{2}-1}] \cup [c_{l,m^{l}_{2}}, \bar{b}_{l}^{*} ] $. There are finite $i\in I_{\mathbf{sg}, l}^{2}\subset\{1,2,\cdots,m^{l}_{2}\}$ such that case (i) \textbf{sticky grazing} holds:
		\begin{equation*}
		\bigcap_{\tau\in [c_{l,i}, d_{l,i}]} {\overline{\xf^1(\bar{\alpha}_l(\tau), \dot{\bar{\alpha}}_{l}(\tau)) \ \xf^2(\bar{\alpha}_{l}(\tau), \dot{\bar{\alpha}}_{l}(\tau))}} =  {x}_{\mathbf{sg},l,i}^{2} \ \text{which is a point in} \ cl(\O),
		\end{equation*}
		by writing $\bar{a}^{*}_{l} = c_{l,1}, \ \bar{b}^{*}_{l} = d_{l,m_{2}^{l}}$. The $2^{\text{nd}}-$sticky grazing set is the collection of such points:
		\begin{equation} \label{def sg2c}
		\mathcal{SG}^{C,2} := \ \bigcup_{l=1}^{M^C} \ \mathcal{SG}^{C,2}_{l} := \ \bigcup_{l=1}^{M^C}  \big\{ {x}_{\mathbf{sg},l,i}^{2} \in \overline{\O} : i\in I_{\mathbf{sg},l}^{2} \big\}. 
		\end{equation}
		Note that $\mathcal{SG}^{C,2}$ is a set of finite points, from finiteness of $M^{C}$ and Lemma~\ref{dichotomy}. \\
	\end{definition}
	
	\noindent {\textbf{Step 4}} We claim 
	\begin{equation} \label{m zero away F}
	\mathfrak{m}_{2} \{ v\in\mathbb{R}^2 : ({x},{v})\in \mathfrak{G}^{C,2} \backslash (\mathfrak{G}^{C,1})_{\varepsilon} \} \ = \ 0 \ ,
	\end{equation}
	for all ${x}\in \mathcal{P}_{x}(\{ cl({\O}) \times \VN \} \backslash \mathfrak{IB})\backslash \mathcal{SG}^{C,2}$. Consider again the set (\ref{new G2C}) and fix $l\in\{1,\cdots,M^{C}\}$. For any $i\in \{1,2,\cdots, m^{l}_{2} \} \ \backslash \ I_{\mathbf{sg},l}^{2}$, we apply case (b) of Lemma~\ref{dichotomy} to  say that 
	\[
		\{ v\in\mathbb{R}^2 : ({x},{v})\in \mathfrak{G}^{C,2} \backslash (\mathfrak{G}^{C,1})_{\varepsilon} \} \cap \mathbb{S}^{1} = \text{finite points},
	\]	
	which gives (\ref{m zero away F}).  \\
	
	
	\begin{lemma}\label{G_C2}
		For any $\varepsilon>0$, there exist an open cover 
		\[
			\Big\{ \bigcup_{i=1}^{l_2} \ B(x^{C,2}_{i},r^{C,2}_{i}) \Big\} \ \bigcup \ \big\{ \bigcup_{y\in\mathcal{SG}^{C,2}} B( y , \varepsilon ) \big\}			
		\]
		for $\mathcal{P}_{x}\big( \{ cl({\O}) \times \VN \} \backslash \mathfrak{IB} \big)$ and corresponding velocity sets $\mathcal{O}^{C,2}_{i} \subset \VN$ with $\mathfrak{m}_{2}(\mathcal{O}^{C,2}_{i}) < \varepsilon$ such that \\
		(1) For any $({x},{v})\in \{ cl({\O}) \times \VN \} \backslash \mathfrak{IB}$, \\
		\[
		x\in B( {x}^{C,2}_{i}, r^{C,2}_{i}) \quad\text{or}\quad x\in B( y,\varepsilon), 
		\]
		for some $x^{C,2}_{i}$, $r^{C,2}_{i}$, and $y \in \mathcal{SG}^{C,2}$.   \\
		(2) Moreover, if $x\notin \bigcup_{y\in\mathcal{SG}^{C,2}} B( y , \varepsilon )$, ${x}\in \ B( {x}^{C,2}_{i}, r^{C,2}_{i}) $, and $ {v}\in \VN \backslash \mathcal{O}^{C,2}_{i}$, then
		\begin{equation*}
		\phi^{2}(x,v) = | v^{1}( {x}, {v}) \cdot \mathbf{n}(x^{2}( {x}, {v})) |  > \delta^{C,2} > 0 \quad\text{and} \quad \phi^{1}(x,v) = |{v} \cdot \mathbf{n}(x_1( {x}, {v})) | > \delta^{C,1} > 0 ,  \\
		\end{equation*}
		for some uniformly positive $\delta^{C,1}, \ \delta^{C,2} > 0$. \\
		
		From above, we define $\varepsilon-$ neighborhood of $\mathfrak{G}^{C,2}$:
		\begin{equation*} \label{G_C2 nbd}
		( \mathfrak{G}^{C,2} )_{\varepsilon} \ = \Big\{ \bigcup_{i=1}^{l_2} \ B(x^{C,2}_{i},r^{C,2}_{i})  \times  \mathcal{O}^{C,2}_{i} \Big\} \ \bigcup \ \big\{ \bigcup_{y\in\mathcal{SG}^{C,2}} B( y , \varepsilon )  \times  \VN \big\} .
		\end{equation*}
	\end{lemma}
	\begin{proof}
		From (\ref{def sg2c}), $\mathcal{SG}^{C,2}$ has only finite points so we make a cover with finite balls, $\bigcup_{y\in\mathcal{SG}^{C,2}} B( y , \varepsilon )$ for $\mathcal{SG}^{C,2}$. \\
		For $x\in \mathcal{P}_{x}(\{ cl({\O}) \times \VN \} \backslash \mathfrak{IB}) \ \backslash \ \bigcup_{y\in\mathcal{SG}^{C,2}} B( y , \varepsilon )$, there at most finite (at most $2M^{C} + 2\sum_{l=1}^{M^{C}} m^{l}_{2}$) unit vectors $\frac{v_{i}}{|v_{i}|}$ such that
		\begin{equation*} \label{bad direc1}
		(x,v_{i}) \in \mathfrak{G}^{2}_{C} \ \cup \ \mathfrak{G}^{1}_{C},
		\end{equation*}
		from (\ref{m zero away F}) and (\ref{g1zero}). So we define
		\begin{equation*} \label{tilde G2 O def}
		\mathcal{O}^{C,2}_{x} := \big\{ \ v\in\VN \ : \ \big| \frac{v_{i}}{|v_{i}|} - \frac{v}{|v|} \big| < C_{2}(N)\varepsilon, \ \forall v_{i} \ \ \text{s.t} \ \ (x,v_{i}) \in \mathfrak{G}^{C,1}\cup \mathfrak{G}^{C,2} \ \big\} .
		\end{equation*}

		When $v\in \VN\backslash\mathcal{O}^{C,2}_{x}$, trajectory does not graze within second bounces, so both 
		\begin{equation*} \label{phi12}
		\phi^1( {x}, {v}) = | {v} \cdot \mathbf{n}(x_{\mathbf{b}}( {x}, {v}))|,\quad \phi^2 ( {x}, {v}) := | v^1 ( {x}, {v}) \cdot \mathbf{n}( x^{2}(x,v) ) |
		\end{equation*}
		are well-defined and locally smooth, because $({x},v)\in \big\{ \{ cl({\O}) \times \VN \} \backslash \mathfrak{IB} \big\} \backslash \ \big( \mathfrak{G}^{C,1} \cup \mathfrak{G}^{C,2} \big)$ implies that trajectory does not graze in first two bounces. Using local continuity of Lemma~\ref{local conti} again, we can find $r^{C,2}_{x} \ll 1$ such that
		\[
		\phi^{1}(x,v) > \delta^{1}_{x} > 0,\quad \phi^{2}(x,v) > \delta^{2}_{x} > 0, \quad\text{for}\quad (x,v)\in cl \big( B(x,r^{C,2}_{x}) \big) \times \VN\backslash\mathcal{O}^{C,2}_{x}.
		\]
		By compactness, we can find finite open cover $\bigcup_{i=1}^{l_{2}} B(x^{C,2}_{i} , r^{C,2}_{i})$ for $\mathcal{P}_{x}\big( \{ cl({\O}) \times \VN \} \backslash \mathfrak{IB} \big) \ \backslash \ \bigcup_{y\in\mathcal{SG}^{C,2}} B( y , \varepsilon ) $ and corresponding $\mathcal{O}^{C,2}_{i}$ with small measure $\mathfrak{m}_{2}(\mathcal{O}^{C,2}_{i}) < \varepsilon$ by choosing (\ref{G1 O def}) with sufficiently small $C_{2}(N)$. Finally we choose
		\[
		\delta^{1} := \min_{1\leq i\leq l_{1}} \delta^{1}_{ x^{C,2}_{i} } > 0, \quad 
		\delta^{2} := \min_{1\leq i\leq l_{2}} \delta^{2}_{ x^{C,2}_{i} } > 0, 
		\]
		to finish the proof.   \\
	\end{proof}

	\subsubsection{$\mathbf{k^{\text{th}}-}$\textbf{Grazing Set}, $\mathfrak{G}^{C,k}$}
	Now we are going to construct, for $k>2$, the $\mathbf{k^{\text{st}}-}$\textbf{Grazing Set} and it's $\varepsilon-$neighborhood. We construct such sets via the mathematical induction. We assume Lemma~\ref{G_C2} holds for $\mathfrak{G}^{C,k-1}$, i.e. \\\
	
	\begin{assumption} \label{assump}
	For any $\varepsilon>0$, there exist $\mathcal{SG}^{C,k-1}$ which contains finite points in $cl(\O)$, and an open cover 
	\[
	\Big\{ \bigcup_{i=1}^{l_{k-1}} \ B(x^{C,k-1}_{i},r^{C,k-1}_{i}) \Big\} \ \bigcup \ \big\{ \bigcup_{y\in\mathcal{SG}^{C,k-1}} B( y , \varepsilon ) \big\}			
	\]
	for $\mathcal{P}_{x}\big( \{ cl({\O}) \times \VN \} \backslash \mathfrak{IB} \big)$ and corresponding velocity sets $\mathcal{O}^{C,k-1}_{i} \subset \VN$ with $\mathfrak{m}_{2}(\mathcal{O}^{C,k-1}_{i}) < \varepsilon$ such that \\
	(1) For any $({x},{v})\in \{ cl({\O}) \times \VN \} \backslash \mathfrak{IB}$, \\
	\[
	x\in B( {x}^{C,k-1}_{i}, r^{C,k-1}_{i}) \quad\text{or}\quad x\in B( y,\varepsilon), 
	\]
	for some $x^{C,k-1}_{i}$, $r^{C,k-1}_{i}$, and $y \in \mathcal{SG}^{C,k-1}$.   \\
	(2) Moreover, if $x \notin \bigcup_{y\in\mathcal{SG}^{C,k-1}} B( y , \varepsilon )$, $ {x}\in \ B( {x}^{C,k-1}_{i}, r^{C,k-1}_{i}) $, and $ {v}\in \VN \backslash \mathcal{O}^{C,k-1}_{i} $, then
	\begin{equation*}
	\phi^{s}(x,v) = | v^{s-1}( {x}, {v}) \cdot \mathbf{n}(x^{s}( {x}, {v})) |  > \delta^{C,s} > 0,  \\
	\end{equation*}
	for all $s\in\{1, \cdots, k-1\}$ some uniformly positive $\delta^{C,1}, \ \delta^{C,2}, \cdots, \delta^{C,k-1} > 0$. \\
	
	We define $\varepsilon-$ neighborhood of $\mathfrak{G}^{C,k-1}$:
	\begin{equation*} \label{G_Ck-1 nbd}
	( \mathfrak{G}^{C,k-1} )_{\varepsilon} \ = \Big\{ \bigcup_{i=1}^{l_{k-1}} \ B(x^{C,k-1}_{i},r^{C,k-1}_{i})  \times  \mathcal{O}^{C,k-1}_{i} \Big\} \ \bigcup \ \big\{ \bigcup_{y\in\mathcal{SG}^{C,k-1}} B( y , \varepsilon )  \times  \VN \big\} .  \\
	\end{equation*}
	\end{assumption}
	
	Now, under above assumption, we follow the steps in $\mathfrak{G}_{C}^{2}$. From the definition of $\mathfrak{G}^{C,k}$ and (\ref{staredinterval}), the set $\mathfrak{G}^{C,k} \ \backslash \ (\mathfrak{G}^{C,k-1})_{\varepsilon}$ is a subset of
	\begin{equation} \label{gk}
	\begin{split}
	& \bigcup_{l=1}^{M^{C}}\bigcup_{p=\pm 1}\Big\{ \Big( \xf^{k-1}(\bar{\alpha}_{l}(\tau), p|v|\dot{\bar{\alpha}}_{l}(\tau)) + s \vf^{k-1}(\bar{\alpha}_{l}(\tau), p|v|\dot{\bar{\alpha}}_{l}(\tau)),  \vf^{k-1}(\bar{\alpha}_{l}(\tau), p|v|\dot{\bar{\alpha}}_{l}(\tau)) \Big) \in \{ cl({\O}) \times \VN \} \backslash \mathfrak{IB} \ :  \\
	&\quad\quad\quad \tau\in[\bar{a}^{*}_{l}, \bar{b}^{*}_{l}], \ v\in\VN, \ s\in[0, (\tf^{k} - \tf^{k-1})(\bar{\alpha}_{l}(\tau), p|v|\dot{\bar{\alpha}}_{l}(\tau)) ]  \Big\} \Big{\backslash} (\mathfrak{G}^{C,k-1})_{\varepsilon}.  \\
	\end{split}
	\end{equation}
	Without loss of generality, we suffice to consider only $p=1$ case of (\ref{gk}).  \\
	
	\noindent{\textbf{Step 1}} Fix $p=1$ and $l\in\{1,\cdots,M^{C}\}$. First, we remove $k-1^{st}$-grazing set by complementing $(\mathfrak{G}^{C,k-1})_{\varepsilon}$. \\
	Let us consider $(\bar{x}, \bar{v}) \in \mathfrak{G}^{C,k-1} \cap \mathfrak{G}^{C,k}_{l}$ and we write $\bar{\alpha}_{l}(\bar{\tau}) = x^{k}(\bar{x}, \bar{v})$. Then, from Assumption \ref{assump}, there exist $i\in\{1,\cdots, l_{k-1}\}$ such that $(\bar{x}, \bar{v}) \in B(x^{C,k-1}_{i},r^{C,k-1}_{i})\times\mathcal{O}^{C,k-1}_{i}$ and 
	\begin{equation} \label{near double graze k}
	\begin{split}
	&\Big\{ x^{k}(x,v)\in\p\O_{l}^{C} \ : \ \forall(x,v) \in cl\big( { B(x^{C,k-1}_{i},r^{C,k-1}_{i})\times\mathcal{O}^{C,k-1}_{i} } \big) , \  \text{where} \ (\bar{x}, \bar{v}) \in B(x^{C,k-1}_{i},r^{C,k-1}_{i})\times\mathcal{O}^{C,k-1}_{i} \Big\} \\
	&\quad\subset \big[ \bar{\alpha}_{l}(\bar{\tau} - \delta_{-}), \bar{\alpha}_{l}(\bar{\tau} + \delta_{+}) \big],\quad \text{for}\quad 0 < \delta_{\pm} = O(r_{k-1,i}, \varepsilon) \ll 1.
	\end{split}
	\end{equation}
	Excluding (\ref{near double graze k}) from $[\bar{a}_{l}^{*}, \bar{b}_{l}^{*}]$ for all $(\bar{x}, \bar{v})\in \mathfrak{G}^{C,k-1} \cap \mathfrak{G}^{C,k}_{l}$ yields a union of countable open connected intervals $\mathcal{I}^{k}$, i.e.
	\begin{equation*} \label{def I k}
	\mathcal{I} := [ \bar{a}_{l}^{*} , d^{k}_{l,1} ) \cup  (  c^{k}_{l,2} ,  d^{k}_{l,2}  )  \cup \cdots  \subset [\bar{a}_{l}^{*}, \bar{b}_{l}^{*}],\quad \bar{a}_{l}^{*} < d_{1} < c_{2} < d_{2} < \cdots.
	\end{equation*}
	Using exactly same argument of \textbf{Step 1} in $\mathbf{2^{\text{nd}}-}$\textbf{Grazing Set} $\mathfrak{G}^{C,2}$, we know that this should be finite union of sub intervals and write 
	\begin{equation} \label{interval split k}
	\mathcal{I}^{k} := [ \bar{a}_{l}^{*} , d^{k}_{l,1} ) \cup  (  c^{k}_{l,2} ,  d^{k}_{l,2}  )  \cup \cdots  \cup  (  c^{k}_{l,m^{l}_{k}-1}  ,  d^{k}_{l,m^{l}_{k}-1}  ) \cup  (  c^{k}_{l,m^{l}_{k}}  , \bar{b}_{l}^{*} ].  \\
	\end{equation}
	
	\noindent{\textbf{Step 2}} Since we have chosen $\delta_{\pm}$ as nonzero in (\ref{near double graze k}), we can include boundary points of each subinterval of (\ref{interval split k}). Therefore, $\mathfrak{G}^{C,k} \ \backslash \ ( \mathfrak{G}^{C,k-1} )_{\varepsilon}$ is a subset of   \\
	\begin{equation} \label{new GkC}
	\begin{split}
	& \bigcup_{l=1}^{M^{C}} \Big\{ \Big( \xf^{k-1}(\bar{\alpha}_{l}(\tau), |v|\dot{\bar{\alpha}}_{l}(\tau)) + s \vf^{k-1}(\bar{\alpha}_{l}(\tau), |v|\dot{\bar{\alpha}}_{l}(\tau)), \vf^{k-1}(\bar{\alpha}_{l}(\tau), |v|\dot{\bar{\alpha}}_{l}(\tau)) \Big) \in \{ cl({\O}) \times \VN \} \backslash \mathfrak{IB}  \ :  \\
	&\quad\quad\quad\quad \tau \in [\bar{a}_{l}^{*} , d^{k}_{l,1}]\cup [c^{k}_{l,2}, d^{k}_{l,2}]\cup\cdots\cup [c^{k}_{l,m^{l}_{k}-1}, d^{k}_{l,m^{l}_{k}-1}] \cup [c^{k}_{l,m^{l}_{k}}, \bar{b}_{l}^{*} ] ,  \\
	&\quad\quad\quad\quad v\in\VN, s\in[0, (\tf^{k} - \tf^{k-1})( \bar{\alpha}_{l}(\tau), |v|\dot{\bar{\alpha}}_{l}(\tau)) ]  \Big\} ,  \\
	\end{split}
	\end{equation}
	and for all $\tau \ \in \ [  \bar{a}_{l}^{*} , d^{k}_{l,1}]\cup [c^{k}_{l,2}, d^{k}_{l,2}]\cup\cdots\cup [c^{k}_{l,m^{l}_{k}-1}, d^{k}_{l,m^{l}_{k}-1}] \cup [c^{k}_{l,m^{l}_{k}}, \bar{b}_{l}^{*} ]$,
	\begin{eqnarray*}
		|\dot{\bar{\alpha}}_{l}(\tau) \cdot \mathbf{n}(\xf(\bar{\alpha}_{l}(\tau),  \dot{\bar{\alpha}}_{l}(\tau))) | > \min_{i=1,\cdots,k-1}\delta^{C,i} > 0,
	\end{eqnarray*}
	where $\delta^{C,i} > 0$ were found in Assumption~\ref{assump}. Moreover, we can choose these subintervals so that measure of each punctures $\{ (d^{k}_{l,i}, c^{k}_{l,i+1}) \}_{i=1}^{m_{k}^{l}-1}$ are arbitrary small, because we can choose $\delta_{\pm} > 0$ arbitrary small in (\ref{near double graze k}).  \\ 
	
	\noindent{\textbf{Step 3}} We construct $\mathbf{k^{\text{th}}-}$\textbf{Sticky Grazing Set} $\mathcal{SG}_{l}^{C,k}$ where all grazing rays from non-measure zero subset of $[ \bar{a}_{l}^{*} , d^{k}_{l,1} ) \cup  (  c^{k}_{l,2} ,  d^{k}_{l,2}  )  \cup \cdots  \cup  (  c^{k}_{l,m^{l}_{k}-1}  ,  d^{k}_{l,m^{l}_{k}-1}  ) \cup  (  c^{k}_{l,m^{l}_{k}}  , \bar{b}_{l}^{*} ]$ intersect at a fixed point in $\mathcal{P}_{x}\Big( \{ cl({\O}) \times \VN \} \backslash \mathfrak{IB} \Big)$, where $\mathcal{P}_{x}$ is projection on sptial domain. Choose any $i\in\{1,\cdots, m^{l}_{k}\}$ and corresponding sub interval $[c^{k}_{l,i}, d^{k}_{l,i}]$. We define   \\
	\begin{equation*} \label{line temp k-1}
	\begin{split}
	\mathfrak{G}^{C,k}_{l,i} &:= \Big\{ \Big( \xf^{k-1}(\bar{\alpha}_{l}(\tau), |v|\dot{\bar{\alpha}}_{l}(\tau)) + s \vf^{k-1}(\bar{\alpha}_{l}(\tau), |v|\dot{\bar{\alpha}}_{l}(\tau)), \vf^{k-1}(\bar{\alpha}_{l}(\tau), |v|\dot{\bar{\alpha}}_{l}(\tau)) \Big) \in \{ cl({\O}) \times \VN \} \backslash \mathfrak{IB} \ :  \\
	&\quad\quad\quad\quad \tau \in [c_{l,i}, d_{l,i}], v\in\VN, \ s\in[0, (\tf^{2} - \tf^{1})(\bar{\alpha}_{l}(\tau), |v|\dot{\bar{\alpha}}_{l}(\tau)) ]  \Big\}.  \\
	\end{split}
	\end{equation*}
	Fix $x^{*}\in \overline{\O}$. If there does not exist $\tau\in [c^{k}_{l,i}, d^{k}_{l,i}]$ and $s\in[t^{k-1}_{\mathbf{f}}(\bar{\alpha}_l(\tau), \dot{\bar{\alpha}}_l(\tau)), \ t^{k}_{\mathbf{f}}(\bar{\alpha}_l(\tau), \dot{\bar{\alpha}}_l(\tau))]$ satisfying $x^{*} = \xf^{k-1}(\bar{\alpha}_l(\tau), \dot{\bar{\alpha}}_l(\tau)) + s\vf^{1}(\bar{\alpha}_l(\tau), \dot{\bar{\alpha}}_l(\tau))$, then $\{ v\in\mathbb{R}^2 : ({x}^{*},{v})\in \mathfrak{G}_{l,i}^{C,k} \} = \emptyset$ with zero measure. Now suppose there exist such $\tau$ and $s$. \\
	\indent Due to Lemma~\ref{dichotomy}, there are only two cases: (i) \textbf{sticky grazing:} for all $\tau\in [c^{k}_{l,i}, d^{k}_{l,i}]$, there exists $s = s(\tau)\in[\tf^{k-1}(\bar{\alpha}_{l}(\tau), \dot{\bar{\alpha}}_l(\tau)), \ \tf^{k}(\bar{\alpha}_l(\tau), \dot{\bar{\alpha}}_l(\tau))]$ and points $\{x^{*,r}\}_{r=1}^{k-1} \in cl(\O)$ such that 
	\begin{equation} \label{zkr}
	x^{*,r} = \xf^{r}(\bar{\alpha}_l(\tau), \dot{\bar{\alpha}}_l(\tau))+s\vf^{r} (\bar{\alpha}_l(\tau)),\quad\text{for all}\quad \tau\in [c^{k}_{l,i}, d^{k}_{l,i}]
	\end{equation}
	or (ii) \textbf{isolated grazing:} there exists $\delta_-, \delta_+>0$ so that for $\tau^{\prime}\in(\tau-\delta_-,\tau+\delta_+)\backslash\{\tau\}$ there is no $s$ satisfying (\ref{zkr}). We define $k^{\text{th}}-$ sticky grazing set $\mathcal{SG}^{C,k}$ as collection of all such $x^{*,r}\in cl(\O)$ points.
	\begin{definition}
		Consider (\ref{new GkC}) and disjoint union of intervals $[\bar{a}_{l}^{*} , d^{k}_{l,1}]\cup [c^{k}_{l,2}, d^{k}_{l,2}]\cup\cdots\cup [c^{k}_{l,m^{l}_{k}-1}, d^{k}_{l,m^{l}_{k}-1}] \cup [c^{k}_{l,m^{l}_{k}}, \bar{b}_{l}^{*} ] $. There are finite $i\in I^{k}_{\mathbf{sg},l} \subset \{1, 2, \cdots, (k-1)m^{l}_{k}\}$ such that case (i) \textbf{sticky grazing} holds:
		\begin{equation*}
		\bigcap_{\tau\in [c^{k}_{l,i}, d^{k}_{l,i}]} {\overline{\xf^{r-1}(\bar{\alpha}_l(\tau), \dot{\bar{\alpha}}_{l}(\tau)) \ \xf^{r}(\bar{\alpha}_{l}(\tau), \dot{\bar{\alpha}}_{l}(\tau))}} =  {x}_{\mathbf{sg},l,i,r}^{k} \in cl(\O),\quad \text{for some}\quad r=1,\cdots, k-1,
		\end{equation*}
		by writing $\bar{a}^{*}_{l} = c^{k}_{l,1}, \ \bar{b}^{*}_{l} = d^{k}_{l,m_{2}^{l}}$. When above intersection is nonempty we collect all those points to obtain $k^{\text{th}}-$ sticky grazing set:
		\begin{equation} \label{def sgkc}
		\mathcal{SG}^{C,k} := \ \bigcup_{l=1}^{M^C} \ \mathcal{SG}^{C,k}_{l} := \ \bigcup_{l=1}^{M^C} \bigcup_{i=1}^{m^{l}_{k}} \bigcup_{r=1}^{k-1} \big\{ {x}_{\mathbf{sg},l,i,r}^{k} \in cl(\O) \big\}. 
		\end{equation}
		Note that $\mathcal{SG}^{C,k}$ has at most $(k-1)M^{C} m^{l}_{k}$ points, from index $i, l $, and $r$.  \\
	\end{definition}
	
	\noindent {\textbf{Step 4}} We claim 
	\begin{equation} \label{k m zero away F}
	\mathfrak{m}_{2} \{ v\in\mathbb{R}^2 : ({x},{v})\in \mathfrak{G}^{C,k} \backslash (\mathfrak{G}^{C,k-1})_{\varepsilon} \} \ = \ 0 \ ,
	\end{equation}
	for all ${x}\in \mathcal{P}_{x}(\{ cl({\O}) \times \VN \} \backslash \mathfrak{IB})\backslash \mathcal{SG}^{C,k}$. Consider again the set (\ref{new GkC}) and fix $l\in\{1,\cdots,M^{C}\}$. For any point $x\in cl(\O)$ such that $i\in \{1,2,\cdots, (k-1)m^{l}_{k} \} \ \backslash \ I_{\mathbf{sg},l}^{k}$, we apply case (b) of Lemma~\ref{dichotomy} to  say that 
	\[
	\{ v\in\mathbb{R}^2 : ({x},{v})\in \mathfrak{G}^{C,k} \backslash (\mathfrak{G}^{C,k-1})_{\varepsilon} \} \cap \mathbb{S}^{1} = \text{finite points},
	\]	
	which gives (\ref{k m zero away F}).  \\
	
	\begin{lemma}\label{G_Ck}
		We assume Assumption~\ref{assump}. Then, for any $\varepsilon>0$, there exist an open cover 
		\[
		\Big\{ \bigcup_{i=1}^{l_k} \ B(x^{C,k}_{i},r^{C,k}_{i}) \Big\} \ \bigcup \ \big\{ \bigcup_{y\in\mathcal{SG}^{C,k}} B( y , \varepsilon ) \big\}			
		\]
		for $\mathcal{P}_{x}\big( \{ cl({\O}) \times \VN \} \backslash \mathfrak{IB} \big)$ and corresponding velocity sets $\mathcal{O}^{C,k}_{i} \subset \VN$ with $\mathfrak{m}_{2}(\mathcal{O}^{C,k}_{i}) < \varepsilon$ such that \\
		(1) For any $({x},{v})\in \{ cl({\O}) \times \VN \} \backslash \mathfrak{IB}$, \\
		\[
		x\in B( {x}^{C,k}_{i}, r^{C,k}_{i}) \quad\text{or}\quad x\in B( y,\varepsilon), 
		\]
		for some $x^{C,k}_{i}$, $r^{C,k}_{i}$, and $y \in \mathcal{SG}^{C,k}$.   \\
		(2) Moreover, if $x \notin \bigcup_{y\in\mathcal{SG}^{C,k}} B( y , \varepsilon ) $, $ {x}\in \ B( {x}^{C,k}_{i}, r^{C,k}_{i}) $, and $ {v}\in \VN \backslash \mathcal{O}^{C,k}_{i} $, then
		\begin{equation*}
		\phi^{r}(x,v) = | v^{r-1}( {x}, {v}) \cdot \mathbf{n}(x^{r}( {x}, {v})) |  > \delta^{C,r} > 0 ,  \quad r=1,\cdots, k, \\
		\end{equation*}
		for some uniformly positive $\delta^{C,r} > 0$, $r=1,\cdots, k$. \\
		
		From above, we define $\varepsilon-$ neighborhood of $\mathfrak{G}^{C,k}$:
		\begin{equation*} \label{G_Ck nbd}
		( \mathfrak{G}^{C,k} )_{\varepsilon} \ = \Big\{ \bigcup_{i=1}^{l_k} \ B(x^{C,k}_{i}, r^{C,k}_{i})  \times  \mathcal{O}^{C,k}_{i} \Big\} \ \bigcup \ \big\{ \bigcup_{y\in\mathcal{SG}^{C,k}} B( y , \varepsilon )  \times  \VN \big\} .
		\end{equation*}
	\end{lemma}
	\begin{proof}
		We suffice to follow the scheme of proof of Lemma~\ref{G_C2}. From (\ref{def sgkc}), $\mathcal{SG}^{C,k}$ has finite points so we make a cover with finite balls, $\bigcup_{y\in\mathcal{SG}^{C,k}} B( y , \varepsilon )$ for $\mathcal{SG}^{C,k}$. \\
		For $x\in \mathcal{P}_{x}(\{ cl({\O}) \times \VN \} \backslash \mathfrak{IB}) \ \backslash \ \bigcup_{y\in\mathcal{SG}^{C,k}} B( y , \varepsilon )$, there at most finite (at most $2M^{C} + 2\sum_{r=1}^{k}\sum_{l=1}^{M^{C}} m^{l}_{r}$) unit vectors $\frac{v_{i}}{|v_{i}|}$ such that
		\begin{equation*} \label{bad direck}
		(x,v_{i}) \in \mathfrak{G}^{k}_{C} \ \cup \ \mathfrak{G}^{k-1}_{C} \ \cup \ \cdots \ \cup \  \mathfrak{G}^{1}_{C},
		\end{equation*}
		from (\ref{k m zero away F}). So we define
		\begin{equation} \label{tilde Gk O def}
		\mathcal{O}^{C,k}_{x} := \big\{ \ v\in\VN \ : \ \big| \frac{v_{i}}{|v_{i}|} - \frac{v}{|v|} \big| < C_{k}(N)\varepsilon, \ \forall v_{i} \ \ \text{s.t} \ \ (x,v_{i}) \in \cup_{r=1}^{k}\mathfrak{G}^{C,r} \ \big\} .
		\end{equation}
		
		When $v\in \VN\backslash\mathcal{O}^{C,k}_{x}$, trajectory does not graze within second bounces, so 
		\begin{equation*} \label{phik}
		\phi^{r}( {x}, {v}) = | {v} \cdot \mathbf{n}(x_{\mathbf{b}}( {x}, {v}))|,\quad 1\leq r \leq k
		\end{equation*}
		are well-defined and locally smooth, because $({x},v)\in \big\{ \{ cl({\O}) \times \VN \} \backslash \mathfrak{IB} \big\} \backslash \ \big( \cup_{r=1}^{k}\mathfrak{G}^{C,r} \big)$ implies that trajectory does not graze in first $k$ bounces. Using local continuity of Lemma~\ref{local conti} again, we can find $r^{C,k}_{x} \ll 1$ such that
		\[
		\phi^{r}(x,v) > \delta^{r}_{x} > 0,\quad \text{for}\quad 1\leq r \leq k\quad \text{and}\quad (x,v)\in cl \big( B(x,r^{C,k}_{x}) \big) \times \VN\backslash\mathcal{O}^{C,k}_{x}.
		\]
		By compactness, we can find an open cover $\bigcup_{i=1}^{l_{k}} B(x^{C,k}_{i} , r^{C,k}_{i})$ for $\mathcal{P}_{x}\big( \{ cl({\O}) \times \VN \} \backslash \mathfrak{IB} \big) \ \backslash \ \bigcup_{y\in\mathcal{SG}^{C,2}} B( y , \varepsilon ) $ and corresponding $\mathcal{O}^{C,k}_{i}$ with small measure $\mathfrak{m}_{2}(\mathcal{O}^{C,k}_{i}) < \varepsilon$ by choosing (\ref{tilde Gk O def}) with sufficiently small $C_{k}(N)$. Finally we choose
		\[
		\delta^{r} := \min_{1\leq i\leq l_{k}} \delta^{r}_{ x^{C,k}_{i} } > 0, \quad  1\leq r \leq k.
		\]
		to finish the proof.   \\
	\end{proof}
	
	\begin{proposition}\label{G_C unif}
		For any $\varepsilon>0$ , we have the $\varepsilon-$neighborhood of $\mathfrak{{G}}$:
		\begin{equation*} \label{G_C unif nbd}
		( \mathfrak{G} )_{\varepsilon} \ = \Big\{ \bigcup_{i=1}^{l_G} \ B(x^{C}_{i}, r^{C}_{i})  \times  \mathcal{O}^{C}_{i} \Big\} \ \bigcup \ \big\{ \bigcup_{j=1}^{l_{sg}} B( y_{j}^{C} ,  \varepsilon )  \times  \VN \big\},
		\end{equation*}
		with $\mathcal{O}^{C}_{i} \subset \VN$,  $\mathfrak{m}_2( \mathcal{O}^{C}_{i} )< \varepsilon $ for all $i=1,2,\cdots,l_G < \infty, $ and $j=1,2,\cdots,l_{sg} < \infty$. For any $({x}, {v})\in \{ cl({\O}) \times \VN \} \backslash \mathfrak{IB}$, 
		\[
		x\in B( {x}^{C}_{i}, r^{C}_{i}) \quad\text{or}\quad x\in B( {y}^{C}_{j}, \varepsilon),
		\]
		for some $x^{C}_{i}$ or $y^{C}_{j}$. Moreover, if $x\notin \bigcup_{j=1}^{l_{sg}} B( y_{j}^{C} , \varepsilon )$, ${x}\in B( {x}^{C}_{i}, r^{C}_{i}) $, and $ {v}\in \VN \backslash \mathcal{O}^{C}_{i} $, then
		\begin{equation*}
		| v^{k-1}(T_{0},{x}, {v}) \cdot \mathbf{n}(x^{k}(T_{0}, {x}, {v})) |  > \delta > 0,\quad \forall t^{k}(T_{0},x,v) \in [0,T_{0}].  \\
		\end{equation*}
	\end{proposition}
	\begin{proof}
		We use mathematical induction. We already proved $k=1$ case in Lemma~\ref{G_C1}, when there is no sticky grazing set. From $k=2$, sticky grazing set appears and we proved Lemma~\ref{G_C2}. From Assumption~\ref{assump} and Lemma~\ref{G_Ck}, we know that Lemma~\ref{G_Ck} holds for any finite $k\in\mathbb{N}$. Moreover, number of bounce is uniformly bounded from Lemma~\ref{uniformbound}. So we stop mathematical induction in the maximal possible number of bouncing on $[0,T_{0}]$.  \\
	\end{proof}

	\subsection{Transversality and double Duhamel trajectory}
	We introduce local parametrization for $U = \O\times [0,H]$. Since we should treat three-dimensional trajectory from this subsection, we introduce the following notation to denote two-dimensional points in cross section, 
	\begin{equation*} \label{underbar}
		\underline{x} = (x_{1}, x_{3}),\quad \underline{v} = (v_{1}, v_{3}).
	\end{equation*}
	where missing $x_{2}$ and $v_{2}$ are components for axis direction. So we can write
	\[
		x = (\underline{x}, x_{2})\in\O,\quad v = (\underline{v}, v_{2})\in\mathbb{R}^{3}.
	\] 
	Especially for the points near boudnary, we define local parametrization, i.e. for $p\in\p\O$,  \\
	\begin{equation}\label{eta}
	\begin{split}
		&\eta_{{p}}:  \{ {\X}_{p} \in \mathbb{R}^{3}:  \X_{p,3}<0  \} \cap B(0, \delta_{1}) \ \rightarrow  \ \Omega \cap B({p}, \delta_{2}),\\
		&{\X}_{p} = ( \underline{\X}_{p}, \X_{p,2})	 \ \mapsto \  {x} := \eta_{p}  ( {\X}_{p} ),  \\
		&\eta_{p} (0,0,0) = p,\quad {x} = \eta_{p}  ( {\X}_{p} ) = ( \underline{\eta}_{p}( \underline{\X}_{p} ), \X_{p,2}  ) ,  \\
		&\underline{x} = \underline{\eta}_{p}  ( \underline{\X}_{p} ) = \underline{\eta}_{p}(\X_{p,1}, 0) + \X_{p,3} \mathbf{n} ( \underline{\eta}_{p}(\X_{p,1}, 0) ),  \\
	\end{split}
	\end{equation} 
	and $\eta_{p}(\X_{p}) \in \p \Omega$ if and only if $\X_{p,3}=0$. $\mathbf{n}\big( \underline{\eta}_{p}(\X_{p,1}, 0) \big)$ is outward unit normal vector at $\big( \underline{\eta}_{p}(\X_{p,1}, 0), \X_{p,2} \big) \in\p\O$. Since $\O$ is cylindrical, unit normal vector $\mathbf{n}$ is indepedent to $\X_{p,2}$. We use the following derivative symbols,
	\begin{equation*}\label{loc deriva sym}
	\begin{split}
		&\p_{i}\eta_{p} := \frac{\p}{\p \X_{i}} \eta_{p},\quad \underline{\nabla} := (\p_{1}, \p_{3}) = ( \frac{\p}{\p \X_{p,1}}, \frac{\p}{\p \X_{p,3}} ), \quad \p_{x_{i}} = \frac{\p}{\p x_{i}}, \quad  \underline{\nabla}_{x} := (\p_{x_{1}}, \p_{x_{3}}) = ( \frac{\p}{\p x_{1}}, \frac{\p}{\p x_{3}} ),   \\
		& \nabla = (\underline{\nabla}, \p_{2}),\quad \nabla_{x} = (\underline{\nabla}_{x}, \p_{2}),
	\end{split}
	\end{equation*} 
	where $x \in cl(\O)$ and ${\X_{p}} \in \eta^{-1}_{p}(\O)$. Note that it is easy to check $\eta_{p}$ is locally triple orthogonal system, i.e.
	\begin{equation} \label{orthonormal_eta}
	\begin{split}
		\langle \p_{i}\eta_{p}, \p_{j}\eta_{p} \rangle  = 0,\quad \text{for all}\quad i\neq j ,\quad \X\in \{\X_{p} \in \mathbb{R}^{3}:  \X_{p,3}<0  \} \cap B(0, \delta_{1}).
	\end{split}
	\end{equation}
	We also use standard notations 
	$
	g_{p,ij} := \langle \p_{i}\eta_{p}, \p_{j}\eta_{p} \rangle
	$
	and transformed velocity $\V_{p}$ is defined by    
	\begin{equation*}
	\mathbf{v}_{p,i}(\X_{p}) : = \frac{\p_{i} \eta_{{p}} (\X_{{p}})}{\sqrt{g_{{p},ii} (\X_{{p}})}} \cdot v,
	\end{equation*}	
	or equivalently,
	\begin{equation} \label{v_v}
	\V_{p} = 
	\begin{bmatrix}
	\V_{p,{1}}\\
	\V_{p,{2}}\\
	\V_{p,{3}}
	\end{bmatrix}
	=
	Q^{T}
	\begin{bmatrix}
	v_{{1}}\\
	v_{{2}}\\
	v_{{3}}
	\end{bmatrix} = Q^{T}v ,
	\quad\text{where}\quad
	Q := 
	\begin{bmatrix}
	\frac{\p_{1} \eta_{p,1}}{\sqrt{g_{p,11}}} & \frac{\p_{2} \eta_{p,1}}{\sqrt{g_{p,22}}} & \frac{\p_{3} \eta_{p,1}}{\sqrt{g_{p,33}}}  \\
	\frac{\p_{1} \eta_{p,2}}{\sqrt{g_{p,11}}} & \frac{\p_{2} \eta_{p,2}}{\sqrt{g_{p,22}}} & \frac{\p_{3} \eta_{p,2}}{\sqrt{g_{p,33}}}  \\
	\frac{\p_{1} \eta_{p,3}}{\sqrt{g_{p,11}}} & \frac{\p_{2} \eta_{p,3}}{\sqrt{g_{p,22}}} & \frac{\p_{3} \eta_{p,3}}{\sqrt{g_{p,33}}}  \\
	\end{bmatrix}.
	\end{equation}

	We compute transversality between two consecutive bouncings using local parametrization (\ref{eta}) and transformed velocity (\ref{v_v}). To denote bouncing index, we define
	\begin{equation}\label{x^k}
	\begin{split}
		\X^{k}_{{p}^{k}}&:=(\X^{k}_{{p}^{k},1}, \X^{k}_{{p}^{k},2} , 0 )  \  \text{ 
		such that } \    x^{k}= \eta_{{p}^{k}} (\X^{k}_{{p}^{k}}), \\   
		\V^{k}_{p^{k},i} &:= \frac{\p_{i} \eta_{{p}^{k}} (\X^{k}_{{p}^{k}})}{\sqrt{g_{{p}^{k},ii} (\X^{k}_{{p}^{k}})}} \cdot v^{k},  \\
	\end{split}
	\end{equation}
	where $p^{k}$ is a point on $\p\O$ near bouncing point $x^{k}$.  \\

	Since dynamics in $x_{2}$ direction is independent to the dynamics in cross section, we focus on the dynamics of two-dimensional cross section $\O$, for fixed $x_{2}$.  
	\begin{lemma}\label{Jac_billiard} Assume that $\Omega$ are $C^{2}$ (not necessarily convex) and $|\V^{k}_{p^{k},3}|, \ | \V^{k+1}_{p^{k+1},3}| > 0$. Consider $(t^{k+1}, \X^{k+1}_{p^{k+1},1}, \underline{\V}^{k+1}_{p^{k+1}} )$ as a functin of $(t^{k+1}, \X^{k}_{p^{k},1}, \underline{\V}^{k}_{p^{k}} )$  \\ 	
		\begin{equation} \label{Est-1}
		\begin{split} 
		\frac{\p (t^k - t^{k+1})}{\p \X^{k}_{{p}^k , 1}} 
		&=  \frac{-1}{   \V^{k+1}_{ {{p}^{k+1}} ,3 }   }
		\frac{  \p_{{3}} \underline{\eta}_{p^{k+1} }( \underline{x}^{k+1} )}{\sqrt{g_{{p}^{k+1},33}( \underline{x}^{k+1} )}} \cdot 
		\bigg[  \p_{1} \underline{\eta}_{{p}^{k}} ( \X^{k}_{{p}^{k},1}, 0 ) - (t^{k}-t^{k+1})
		\frac{\p \underline{v}^{k}   }{\p \X^{k}_{{p}^{k},1}}
		\bigg] ,
		\\ 
		\end{split}
		\end{equation}
		
		\begin{equation} \label{Est-2}
		\begin{split}
		\frac{\p \X_{{p}^{k+1},1}^{k+1}}{\p { \X_{{p}^{k},1}^{k}}}   
		&=   
		\frac{1}{\sqrt{g_{p^{k}, 11} ( \underline{x}^{k+1}  )}} \Big[\frac{\p_{1} \underline{\eta}_{p^{k+1} } (\underline{x}^{k+1} )}{ \sqrt{g_{p^{k}, 11} (\underline{x}^{k+1} )}} + 
		\frac{\V^{k+1}_{p^{k+1}, 1} }{   \V^{k+1}_{ {{p}^{k+1}} ,3 }   }
		\frac{  \p_{{3}} \underline{\eta}_{{p}^{k+1}}(\underline{x}^{k+1} )}{\sqrt{g_{{p}^{k+1},33}(\underline{x}^{k+1} )}} 
		\Big] 	\\
		&\quad \cdot 
		\bigg[  \p_{1} \underline{\eta}_{{p}^{k}} ( \underline{x}^{k}  ) - (t^{k}-t^{k+1})
		\frac{\p \underline{v}^{k}   }{\p \X^{k}_{{p}^{k},1}}
		\bigg] ,
		\\
		\end{split}
		\end{equation}
		\begin{equation} \label{Est-3}
		\begin{split}
		\frac{\p  {\V}^{k+1}_{{p}^{k+1}, 1}}{\p \X^{k}_{{p}^{k}, 1}}  
		&=  
		\frac{ \p \underline{v}^{k}  }{ \p \X^{k}_{{p}^{k}, 1} } \cdot \frac{\p_{1} \underline{\eta}_{{p}^{k+1}} (\underline{x}^{k+1})}{\sqrt{g_{{p}^{k+1}, 11} (\underline{x}^{k+1})}}
		+   {\underline{v}^{k} } \cdot \frac{\p \X^{k+1}_{{p}^{k+1}, 1}}{\p \X^{k}_{{p}^{k}, 1}} \frac{\p}{\p \X^{k+1}_{{p}^{k+1}, 1}}
		\Big( \frac{\p_{1} \underline{\eta}_{{p}^{k+1}} (\underline{x}^{k+1})}{\sqrt{g_{{p}^{k+1}, 11} (\underline{x}^{k+1})}}
		\Big)	,	\\
		\end{split}
		\end{equation}
		\begin{equation} \label{Est-4}
		\begin{split}
		\frac{\p  {\V}^{k+1}_{{p}^{k+1}, 3}}{\p \X^{k}_{{p}^{k }, 1}} 
		&=  
		- \frac{\p \underline{v}^{k}  }{\p \X^{k}_{{p}^{k}, 1}} \cdot \frac{\p_{3} \underline{ \eta}_{{p}^{k+1}}^{k+1}  (\underline{ x}^{k+1} )}{\sqrt{g_{{p}^{k+1}, 33} (\underline{ x}^{k+1} )}}
		-   {\underline{ v}^{k}  } \cdot \frac{\p \X^{k+1}_{{p}^{k+1}, 1}}{\p \X^{k}_{{p}^{k}, 1}} \frac{\p}{\p \X^{k+1}_{{p}^{k+1}, 1}}
		\Big( \frac{\p_{3} \underline{ \eta}_{{p}^{k+1}}^{k+1}  (\underline{ x}^{k+1} )}{\sqrt{g_{{p}^{k+1}, 33} (\underline{ x}^{k+1} )}}
		\Big) ,   \\
		\end{split}
		\end{equation}
		where
		\begin{equation}\label{V_X}
		\frac{\p v^{k}_{i}}{\p \X^{k}_{{p}^{k},1}} = \sum_{\ell=1}^{3} \V^{k}_{{p}^{k},\ell} \sum_{r (\neq \ell)}
		\frac{\sqrt{g_{{p}^{k}, rr} (\underline{ x}^{k} ) }}{ \sqrt{g_{{p}^{k}, \ell\ell} (\underline{ x}^{k} ) }
		}
		\Gamma_{{p}^{k}, \ell 1}^{r} (\underline{ x}^{k} ) \frac{\p_{r} \underline{\eta}_{{p}^{k},i} (\underline{ x}^{k} )}{\sqrt{g_{{p}^{k}, rr} (\underline{ x}^{k} )}},\quad i=1,3.
		\end{equation}
		
		\noindent For $i=1$ and $   j=1,3, $
		\begin{equation} \label{Est-5}
		\begin{split} 
		\frac{\p (t^k - t^{k+1}) }{\p \V^{k }_{{p}^{k} , j}} &    =   
		\frac{   (t^k - t^{k+1}) }{\V^{k+1}_{{p}^{k+1}, 3}  } 
		\bigg[
		\frac{\p_{j} \underline{\eta}_{{p}^{k}}  (\underline{x}^{k}) }{\sqrt{g_{{p}^{k},jj} (\underline{x}^{k} )}}
		 \bigg] \cdot 	\frac{\p_{3} \underline{\eta}_{{p}^{k+1}}^{k+1}  (\underline{x}^{k+1} )}{ \sqrt{g_{{p}^{k+1},33} (\underline{x}^{k+1} )}}, 
		\end{split}
		\end{equation}
		\begin{equation} \label{Est-6}
		\begin{split} 
		\\
		\frac{\p \X^{k+1}_{{p}^{k+1}, 1}}{\p \V^{k}_{{p}^{k}, j}} 
		& =   - (t^k - t^{k+1}) \frac{\p_{j} \underline{\eta}_{{p}^{k}}  (\underline{ x}^{k} ) }{\sqrt{g_{{p}^{k}, jj} (\underline{ x}^{k})} } \cdot 
		\frac{1}{\sqrt{g_{{p}^{k+1}, 11} (\underline{ x}^{k+1} )  }}
		\Big[
		\frac{\p_{1} \underline{ \eta}_{{p}^{k+1}}^{k+1}  (\underline{ x}^{k+1} )  }{\sqrt{g_{{p}^{k+1}, 11}  (\underline{ x}^{k+1} )  }} 	\\
		&\quad 
		+\frac{\V^{k+1}_{{p}^{k+1}, 1}}{\V^{k+1}_{{p}^{k+1},3}} \frac{\p_{3} \underline{ \eta}_{{p}^{k+1}}^{k+1}  (\underline{ x}^{k+1} ) 
		}{\sqrt{g_{{p}^{k+1}, 33}} (\underline{ x}^{k+1} )  }
		\Big] ,
		\end{split}
		\end{equation}
		\begin{equation} \label{Est-7}
		\begin{split} 
		\\
		%
		%
		%
		%
		\frac{\p {\V}^{k+1}_{{p}^{k+1},{1} }}{\p {\V}^{k}_{{p}^{k},j}}  
		& =  \sum_{\ell=1}^{2} \frac{\p \X^{k+1}_{{p}^{k+1}, \ell}}{\p \V^{k}_{{p}^{k}, j}}
		\p_{\ell} \Big( \frac{\p_{1} \underline{ \eta}_{{p}^{k+1}}^{k+1} } {\sqrt{g_{{p}^{k+1}, 11}}}\Big)\Big|_{\underline{ x}^{k+1} } \cdot \underline{ v}^{k}   + \frac{\p_{1} \underline{\eta}_{{p}^{k+1}} (\underline{ x}^{k+1} ) }{\sqrt{g_{{p}^{k+1}, 11} (\underline{ x}^{k+1} )}} \cdot \frac{\p_{j} \underline{\eta}_{{p}^{k}} ( \underline{x}^{k} )}{\sqrt{g_{{p}^{k}, jj} ( \underline{x}^{k} )}} , 	\\
		\end{split}
		\end{equation}

		\begin{equation} \label{Est-8}
		\begin{split}
		\frac{\p {\V}^{k+1}_{{p}^{k+1},{3} }}{\p {\V}^{k}_{{p}^{k},j}}  
		& =  - \sum_{\ell=1}^{2} \frac{\p \X^{k+1}_{{p}^{k+1}, \ell}}{\p \V^{k}_{{p}^{k}, j}}
		\p_{\ell} \Big( \frac{\p_{3} \underline{ \eta}_{{p}^{k+1}}^{k+1}  }{\sqrt{g_{{p}^{k+1},33}}}\Big)\Big|_{\underline{ x}^{k+1} } \cdot \underline{ v}^{k}   - \frac{\p_{3} \eta_{{p}^{k+1} } (\underline{ x}^{k+1} ) }{\sqrt{g_{{p}^{k+1}, 33} (\underline{ x}^{k+1} )}} \cdot \frac{\p_{j} \underline{\eta}_{{p}^{k}} ( \underline{x}^{k} )}{\sqrt{g_{{p}^{k}, jj} ( \underline{x}^{k} )}} ,	\\
		.\end{split}\end{equation}
	\end{lemma}

	\begin{proof} \textit{Proof of (\ref{Est-1})}. By the definitions (\ref{eta}), (\ref{specular_cycles}), and our setting (\ref{x^k}) and (\ref{E_Ham}),
		\begin{equation}\begin{split} \label{position identity}
		\underline{ x}^{k+1} (\X^{k+1}_{{p}^{k+1},1}, 0)
		&=	\underline{\eta}_{{p}^{k}}(\X^{k}_{{p}^{k},1}, 0) + \int_{t^{k}}^{ t^{k+1}} \underline{v}^{k} . \\
		\end{split}\end{equation}
		We take $\frac{\p}{\p \X^{k}_{{p}^{k}, 1}}$ to above equality to get
		\begin{equation} \label{diff pos iden}
		\begin{split}
		&  \frac{\p \X^{k+1}_{{p}^{k+1}, 1}}{\p \X^{k}_{{p}^{k}, 1}} \frac{\p\underline{ \eta}_{{p}^{k+1}}^{k+1} }{\p \X^{k+1}_{{p}^{k+1}, 1}}\Big\vert_{\underline{ x}^{k+1} }   
		=  - (t^{k}-t^{k+1}) \frac{\p \underline{ v}^{k}  }{\p \X_{p^{k}, 1}^{k}} -  \frac{\p (t^{k }-t^{k+1}) }{\p {\X^{k}_{{p}^{k}, 1}}} \underline{ v}^{k}   + 
		\p_{1} \underline{\eta}_{{p}^{k}}( \X^{k}_{{p}^{k},1}, 0 )
		\\
		\end{split}
		\end{equation}
		and then take an inner product with $\frac{\p_{{3}} \underline{ \eta}_{{p}^{k+1}}^{k+1} }{\sqrt{g_{{p}^{k+1},33}}}\Big|_{\underline{ x}^{k+1} }  $ to have 
		\begin{equation*} \label{diff pos iden dot 3}
		\begin{split}
		&  \frac{\p \X^{k+1}_{{p}^{k+1}, 1}}{\p \X^{k}_{{p}^{k}, 1}} \frac{\p\underline{ \eta}_{{p}^{k+1}}^{k+1} }{\p \X^{k+1}_{{p}^{k+1}, 1}}\Big\vert_{\underline{ x}^{k+1} }  \cdot
		\frac{\p_{{3}} \underline{ \eta}_{{p}^{k+1}}^{k+1}  }{\sqrt{g_{{p}^{k+1},33}}}   \Big\vert_{\underline{ x}^{k+1} } 
		\\
		&=  - (t^{k}-t^{k+1}) \frac{\p \underline{ v}^{k}  }{\p \X_{p^{k}, 1}^{k}} \cdot \frac{\p_{3} \underline{ \eta}_{{p}^{k+1}}^{k+1} }{\sqrt{g_{{p}^{k+1},33}} }  \Big|_{\underline{ x}^{k+1} }  
		 -   \frac{\p (t^{k }-t^{k+1}) }{\p {\X^{k}_{{p}^{k}, 1}}}  \underline{ v}^{k}    	
		\cdot \frac{ \p_{3} \underline{ \eta}_{{p}^{k+1}}^{k+1} }{\sqrt{g_{{p}^{k+1},33}}}\Big|_{ \underline{ x}^{k+1} } \\
		&\quad + 
		\p_{1} \underline{\eta}_{{p}^{k}}( \X^{k}_{{p}^{k},1}, 0 ) \cdot \frac{ \p_{3} \underline{ \eta}_{{p}^{k+1}}^{k+1} }{
			\sqrt{g_{p^{k+1},33}}
		}\Big|_{\underline{ x}^{k+1} } ,
		\end{split}
		\end{equation*}
		where we abbreviated $\underline{X}(s) = \underline{X}(s;t^{k},\underline{ x}^{k} ,\underline{ v}^{k} )$ and $\underline{V}(s) = \underline{V}(s;t^{k},\underline{ x}^{k} ,\underline{ v}^{k} )$. Due to (\ref{orthonormal_eta}) the LHS equals zero. Now we consider the RHS. From (\ref{v_v}), we prove (\ref{V_X}). 
		We also note that 
		\begin{equation}\label{down_V}
		\lim_{s\downarrow t^{k+1}} V (s;t^{k}, \underline{ x}^{k} , \underline{ v}^{k} ) = \underline{ v}^{k}   . 
		\end{equation}
		Therefore, from (\ref{specular_cycles}) and (\ref{x^k}), 
		$$  v^{k }   \cdot \frac{ \p_{3} \underline{ \eta}_{{p}^{k+1}}^{k+1} }{\sqrt{g_{{p}^{k+1},33}}}\big|_{ \underline{ x}^{k+1} } = -\V^{k+1}_{p^{k+1},3}. $$  
		Dividing both sides by $\underline{ v}^{k}  	
		\cdot \p_{3} \underline{ \eta}_{{p}^{k+1}}^{k+1}  \big|_{\underline{ x}^{k+1} } = \V^{k+1}_{p^{k+1},3}$, we get (\ref{Est-1}). \\
		
		\vspace{4pt}
		
		\noindent \textit{Proof of (\ref{Est-2})}. We take inner product with $\frac{\p_{1} \underline{ \eta}_{{p}^{k+1}}^{k+1} }{g_{{p}^{k+1}, 11}}\Big|_{\underline{ x}^{k+1} }$ to (\ref{diff pos iden}) to have 
		\begin{equation*} \label{diff pos iden ii}
		\begin{split}
		&  \frac{\p \X^{k+1}_{{p}^{k+1}, 1}}{\p \X^{k}_{{p}^{k}, 1}} \frac{\p\underline{ \eta}_{{p}^{k+1}}^{k+1} }{\p \X^{k+1}_{{p}^{k+1}, 1}}\Big\vert_{\underline{ x}^{k+1} } \cdot \frac{\p_{1} \underline{ \eta}_{{p}^{k+1}}^{k+1} }{g_{{p}^{k+1}, 11}}\Big|_{\underline{ x}^{k+1} } = \frac{\p \X^{k+1}_{{p}^{k+1}, 1}}{\p \X^{k}_{{p}^{k}, 1}} 
		\\
		&=  - (t^{k}-t^{k+1}) \frac{\p \underline{ v}^{k}  }{\p \X_{p^{k},1}^{k}} \cdot \frac{\p_{1} \underline{ \eta}_{{p}^{k+1}}^{k+1} }{g_{{p}^{k+1},11}}\Big|_{\underline{ x}^{k+1} }  -   \frac{\p (t^{k }-t^{k+1}) }{\p {\X^{k}_{{p}^{k}, 1}}} \underline{ v}^{k}    \cdot \frac{\p_{1} \underline{ \eta}_{{p}^{k+1}}^{k+1} }{g_{{p}^{k+1}, 11}}\Big|_{\underline{ x}^{k+1} }	
		+ 
		\p_{1} \underline{\eta}_{{p}^{k}}( \X^{k}_{{p}^{k},1}, 0 ) \cdot \frac{\p_{1} \underline{ \eta}_{{p}^{k+1}}^{k+1} }{g_{{p}^{k+1},11}}\Big|_{\underline{ x}^{k+1} } .
		\end{split}
		\end{equation*}
		Since
		$$  \underline{ v}^{k}    \cdot \frac{ \p_{1} \underline{ \eta}_{{p}^{k+1}}^{k+1} }{g_{{p}^{k+1}, 11}}\big|_{ \underline{ x}^{k+1} } = - {\V^{k+1}_{p^{k+1},1} \over \sqrt{g_{{p}^{k+1}, 11}} }, $$  
		from (\ref{orthonormal_eta}) and (\ref{Est-1}),
		\begin{equation*} \label{Estimate 2}
		\begin{split}
		\frac{\p \X_{{p}^{k+1},1}^{k+1}}{\p { \X_{{p}^{k}, 1}^{k}}}   
		&=    \frac{1}{   \V^{k+1}_{ {{p}^{k+1}} ,3 }   }
		\frac{  \p_{{3}} \underline{ \eta}_{{p}^{k+1}}^{k+1} (\underline{ x}^{k+1} )}{\sqrt{g_{{p}^{k+1},33}(\underline{ x}^{k+1} )}} \cdot 
		\bigg[  \p_{1} \underline{\eta}_{{p}^{k}} ( \underline{ x}^{k}  ) - (t^{k}-t^{k+1})
		\frac{\p \underline{ v}^{k}   }{\p \X^{k}_{{p}^{k}, 1}}
		\bigg] 
		\frac{ \V^{k+1}_{p^{k+1}, 1} }{\sqrt{g_{p^{k+1}, 11}}} \Big|_{\underline{ x}^{k+1} } \\
		&\quad	+ \frac{\p_{1} \underline{\eta}_{p^{k+1}}} { {g_{p^{k+1}, 11} }} \Big|_{\underline{ x}^{k+1} }  \cdot  \Big[\p_{1} \underline{\eta}_{p^{k}} (\underline{ x}^{k} )
		- (t^{k}-t^{k+1})  \frac{\p \underline{ v}^{k} }{\p \X^{k}_{{p}^{k}, 1}}  \Big].  \\
		\end{split}
		\end{equation*}
		This ends the proof of (\ref{Est-2}).

		\vspace{4pt}
		
		\noindent \textit{Proof of (\ref{Est-3}) and (\ref{Est-4})}. From (\ref{specular_cycles}) and (\ref{x^k}),			
		\begin{equation}  
		\begin{split}\label{v_123}
		{\V}^{k+1}_{{p}^{k+1}, 1}   
		&=  \frac{ \p_{1} \underline{ \eta}_{{p}^{k+1}}^{k+1}  }{\sqrt{g_{{p}^{k+1},11}}}  \Big|_{\underline{ x}^{k+1} } \cdot  \lim_{s \downarrow  t^{k+1}} V (s; t^{k}, \underline{ x}^{k} ,\underline{ v}^{k} ), \\
		{\V}^{k+1}_{{p}^{k+1},3}  & =  
		- \frac{ \p_{3} \underline{ \eta}_{{p}^{k+1}}^{k+1}  }{\sqrt{g_{{p}^{k+1},33}}}  \Big|_{\underline{ x}^{k+1} } \cdot  \lim_{s \downarrow  t^{k+1}} V (s; t^{k}, \underline{ x}^{k} ,\underline{ v}^{k} ).
		\end{split}
		\end{equation}
		From (\ref{v_123}), 
		\begin{equation*}\begin{split}
		\frac{\p \V^{k+1}_{{p}^{k+1}, 1}}{\p \X^{k}_{{p}^{k}, 1}}  =&
		\frac{\p_{1} \underline{\eta}_{{p}^{k+1}} (\underline{ x}^{k+1} ) }{\sqrt{g_{{p}^{k+1}, 11} (\underline{ x}^{k+1} )}} \cdot  \frac{\p \underline{ v}^{k}  }{\p \X^{k}_{{p}^{k}, 1}}  \\
		&\quad				+
		\frac{\p \X^{k+1}_{{p}^{k+1}, 1}}{\p \X^{k}_{{p}^{k}, 1}} \frac{\p}{\p \X^{k+1}_{{p}^{k+1}, 1}}\Big(  \frac{\p_{i} \underline{ \eta}_{{p}^{k+1}}^{k+1}   }{\sqrt{g_{{p}^{k+1}, 11}    }  }  \Big)\Big|_{\underline{ x}^{k+1} } \cdot   \lim_{s \downarrow  t^{k+1}} V (s; t^{k}, \underline{ x}^{k} ,\underline{ v}^{k} ), \end{split}\end{equation*}	
		\begin{equation*} \begin{split}
		\frac{\p \V^{k+1}_{{p}^{k+1}, 3}}{\p \X^{k}_{{p}^{k}, 1}}  =&
		-\frac{\p_{3} \underline{\eta}_{{p}^{k+1}} (\underline{ x}^{k+1} ) }{\sqrt{g_{{p}^{k+1}, 33} (\underline{ x}^{k+1} )}} \cdot 
		 \frac{\p \underline{ v}^{k} }{\p \X^{k}_{{p}^{k}, 1}}
		  \\
		&
		-
		\frac{\p \X^{k+1}_{{p}^{k+1}, 1}}{\p \X^{k}_{{p}^{k}, 1}} \frac{\p}{\p \X^{k+1}_{{p}^{k+1}, 1}}\Big(  \frac{\p_{3} \underline{ \eta}_{{p}^{k+1}}^{k+1}   }{\sqrt{g_{{p}^{k+1}, 33}    }  }  \Big)\Big|_{\underline{ x}^{k+1} } \cdot  \lim_{s \downarrow  t^{k+1}}V (s; t^{k}, \underline{ x}^{k} ,\underline{ v}^{k} ).\end{split}
		\end{equation*}
		From (\ref{down_V}), we prove (\ref{Est-3}) and (\ref{Est-4}).				

		\vspace{4pt}
		
		Now we consider (\ref{Est-5})-(\ref{Est-8}) for $v -$derivatives.
		
		\vspace{4pt}
		
		\noindent\textit{Proof of (\ref{Est-5})}. We take $\frac{\p}{\p \V^{k}_{{p}^{k},j}}$ to (\ref{position identity}) for $j=1,3$ to get
		\begin{equation} \label{v diff pos iden}
		\begin{split}
		&  \frac{\p \X^{k+1}_{{p}^{k+1},1}}{\p \V^{k}_{{p}^{k},j}} \frac{\p\underline{ \eta}_{{p}^{k+1}}^{k+1} }{\p \X^{k+1}_{{p}^{k+1},1}}\Big\vert_{\underline{ x}^{k+1} }   
		=  - (t^{k}-t^{k+1}) \frac{\p \underline{ v}^{k}  }{\p \V_{p^{k},j}^{k}}  -   \frac{\p (t^{k }-t^{k+1}) }{\p {\V^{k}_{{p}^{k},j}}}   \underline{ v}^{k} ,   
		\\
		\end{split}
		\end{equation}
		and then take an inner product with $\frac{\p_{{3}} \underline{ \eta}_{{p}^{k+1}}^{k+1} }{\sqrt{g_{{p}^{k+1},33}}}\Big|_{\underline{ x}^{k+1} }  $ to have

		\begin{equation}  \label{v diff pos iden dot 3}
		\begin{split}
		& \frac{\p \X^{k+1}_{{p}^{k+1},1}}{\p \V^{k}_{{p}^{k},j}} \frac{\p\underline{ \eta}_{{p}^{k+1}}^{k+1} }{\p \X^{k+1}_{{p}^{k+1},1}}\Big\vert_{\underline{ x}^{k+1} }  \cdot   \frac{\p_{3} \underline{\eta}_{{p}^{k+1}} }{\sqrt{g_{{p}^{k+1}, 33} }} \Big|_{\underline{ x}^{k+1} }  \\
		&=	\Big\{ - ( t^{k} - t^{k+1})  \frac{\p \underline{ v}^{k}  }{\p \V^{k}_{p^{k}, j}} - \frac{\p( t^{k} - t^{k+1})}{\p \V^{k}_{p^{k}, 1}} \lim_{s \downarrow t^{k+1}} V (s;t^{k}, \underline{ x}^{k} , \underline{ v}^{k} ) \Big\} \cdot  \frac{\p_{3} \underline{\eta}_{{p}^{k+1}} }{\sqrt{g_{{p}^{k+1}, 33} }} \Big|_{\underline{ x}^{k+1} } . \\
		\end{split}
		\end{equation}
		Due to (\ref{orthonormal_eta}), the LHS equals zero. Now we consider the RHS.
		From (\ref{v_v}),
		
		
		\begin{equation} \label{dv^0/dv^0}
		\frac{\p \underline{ v}^{k}  }{\p {\V}^{k}_{{p}^{k},j}} =  \frac{\p_{j} \underline{\eta}_{{p}^{k}}  (\X_{{p}^{k}, 1}^{k}, 0 ) }{\sqrt{g_{{p}^{k},jj} (\X_{{p}^{k}, 1}^{k},  0 )}} .
		\end{equation} 

		\noindent Using (\ref{v_123}), (\ref{v diff pos iden dot 3}), and (\ref{dv^0/dv^0}), we prove (\ref{Est-5}).
		
		\vspace{4pt}
		
		\noindent\textit{Proof of (\ref{Est-6})}. For $j=1,3$ , we take inner product with $\frac{\p_{{i}} \underline{ \eta}_{{p}^{k+1}}^{k+1} }{g_{{p}^{k+1},ii}}\Big|_{\underline{ x}^{k+1} }$ to (\ref{v diff pos iden}) to have 
		\begin{eqnarray*}
			\frac{\p \X^{k+1}_{p^{k+1}, 1}}{\p \V^{k}_{p^{k}, j}} 
			&=&
			\Big\{- \frac{\p (t^{k} - t^{k+1})}{\p \V^{k}_{p^{k}, j}} \lim_{s\downarrow t^{k+1}} \underline{V} (s; t^{k}, \underline{ x}^{k} , \underline{ v}^{k} )- (t^{k} - t^{k+1}) \frac{\p \underline{ v}^{k} }{\p \V^{k}_{p^{k}, j}}
			\Big\}
			\cdot \frac{\p_{{1}} \underline{ \eta}_{{p}^{k+1}}^{k+1} }{g_{{p}^{k+1},11}}\Big|_{\underline{ x}^{k+1} } .	\\
		\end{eqnarray*}
		From (\ref{dv^0/dv^0}) and (\ref{Est-5}), we prove (\ref{Est-6}).

		\vspace{4pt}
		
		\noindent\textit{Proof of (\ref{Est-7}) and (\ref{Est-8}) }. For $j=1,3$, from (\ref{v_123}), 
		\begin{eqnarray*}
			\frac{\p {\V}^{k+1}_{{p}^{k+1},{1} }}{\p {\V}^{k}_{{p}^{k},j}}  &=&
			\frac{\p \X^{k+1}_{{p}^{k+1}, 1}}{\p \V^{k}_{{p}^{k}, j}}
			\p_{1}\Big( \frac{\p_{1} \underline{ \eta}_{{p}^{k+1}}^{k+1}  }{\sqrt{g_{{p}^{k+1},11}}}\Big)\Big|_{\underline{ x}^{k+1} } \cdot  \lim_{s \downarrow t^{k+1}} V (s;t^{k}, \underline{ x}^{k} , \underline{ v}^{k} )   + \frac{\p_{1} \underline{\eta}_{{p}^{k+1}} (\underline{ x}^{k+1} ) }{\sqrt{g_{{p}^{k+1}, 11} (\underline{ x}^{k+1} )}} \cdot 
			\frac{\p \underline{ v}^{k}  }{\p \V^{k}_{{p}^{k},j}}  \\
			&=& \frac{\p \X^{k+1}_{{p}^{k+1}, 1}}{\p \V^{k}_{{p}^{k}, j}}
			\p_{1} \Big( \frac{\p_{1} \underline{ \eta}_{{p}^{k+1}}^{k+1}  }{\sqrt{g_{{p}^{k+1},11}}}\Big)\Big|_{\underline{ x}^{k+1} } \cdot \underline{ v}^{k}   
			+
			\frac{\p_{1} \underline{\eta}_{{p}^{k+1}} (\underline{x}^{k+1}) }{\sqrt{g_{{p}^{k+1}, 11} (\underline{ x}^{k+1} )}} \cdot \frac{\p_{j} \underline{ \eta}_{{p}^{k}}^{k}  (\underline{x}^{k})}{\sqrt{g_{{p}^{k}, jj} (\underline{x}^{k})}} .  \\
		\end{eqnarray*}
		From (\ref{Est-5}) and (\ref{Est-6}), we prove (\ref{Est-7}). The proof of (\ref{Est-8}) is also very similar as above from (\ref{v_123}).  \\
	\end{proof}
	
		\begin{lemma} \label{global to local} 
			Assume that $x\in \Omega$ (not necessarily convex) and $\xb(t,x,v)$ is in the neighborhood of $p^{1} \in \p\Omega$. When $|\V_{p^{1},3}^{1}| > 0$, locally, 
			\begin{equation} \label{Est--1}
			\begin{split}
			\frac{\p \tb }{\p x_{j} } &= - \frac{1}{\V^{1}_{{p^{1}},3}} 
			e_{j} \cdot  \frac{\p_{3} \underline{\eta}_{{p}^{1}}  (\underline{x}^{1} )}{\sqrt{g_{{p^{1}}, 33} (\underline{x}^{1})}}, \quad j=1,3, \\
			\end{split}
			\end{equation}	
			\begin{equation} \label{Est--2}
			\begin{split}	
			\frac{\p \tb}{\p v_{j}}
			&= \frac{\tb}{\V^{1}_{{p^{1}},3}} e_{j} 
			\cdot \frac{\p_{3} \underline{\eta}_{{p^{1}}} (\underline{x}^{1}) }{ \sqrt{g_{{p^{1}},33} ( \underline{x}^{1} ) }} ,  \quad j=1,3, \\
			\end{split}
			\end{equation}
			\begin{equation} \label{Est--3}
			\begin{split}	
			\frac{\p \X^{1}_{{p^{1}}, 1}}{\p x_{j} } 
			&=
			e_{j} \cdot \frac{1}{\sqrt{g_{{p^{1}},11} (\underline{x}^{1})  }}
			\Big[ \frac{\p_{1} \underline{\eta}_{{p^{1}}} (\underline{x}^{1})   }{ \sqrt{g_{{p^{1}}, 11}  (\underline{x}^{1})  }   }
			+ \frac{\V^{1}_{{p^{1}}, 1}   }{\V_{{p^{1}},3}   } \frac{\p_{3} \underline{\eta}_{{p}^{1}}  (\underline{x}^{1}) }{\sqrt{g_{{p^{1}}, 33} (\underline{x}^{1})  }}  \Big],   \quad j=1,3, \\
			\end{split}
			\end{equation}
			\begin{equation} \label{Est--4}
			\begin{split}		
			\frac{\p \X^{1}_{{p^{1}},1}}{\p v_{j}}
			&=
			- \tb e_{j} 
			\cdot \frac{1}{\sqrt{g_{{p^{1}},11}  (\underline{x}^{1})  }}
			\Big[ \frac{\p_{1} \underline{\eta}_{{p}^{1}} (\underline{x}^{1})  }{\sqrt{g_{{p^{1}},11}(\underline{x}^{1})}   } { +} \frac{\V^{1}_{{p^{1}},1}}{\V^{1}_{{p^{1}},3}} \frac{\p_{3} \underline{\eta}_{{p}^{1}}(\underline{x}^{1})   }{\sqrt{g_{{p^{1}},33}  (\underline{x}^{1}) }}\Big] \quad j=1,3, \\
			\end{split}
			\end{equation}
			\begin{equation} \label{Est--5}
			\begin{split}	
			\frac{\p \V^{1}_{{p^{1}}, i}}{\p x_{j}} 
			&=  
			\frac{\p \X^{1}_{{p^{1}}, 1}}{\p x_{j}}  
			\p_{1}\big(  \frac{\p_{i} \underline{\eta}_{{p}^{1}}}{\sqrt{g_{{p^{1}}, ii}}}  \big) 
			\Big|_{\underline{x}^{1}}\cdot V (t-\tb)   
			=
			 \frac{\p \X^{1}_{{p^{1}}, 1}}{\p x_{j}}  
			\p_{1}\big(  \frac{\p_{i} \underline{\eta}_{{p}^{1}}}{\sqrt{g_{{p^{1}}, ii}}}  \big) 
			\Big|_{\underline{x}^{1}}\cdot  \underline{v} , \quad i = 1,3, \ j = 1,3, \\
			\end{split}
			\end{equation}
			\begin{equation} \label{Est--6}
			\begin{split}	
			\frac{\p \V^{1}_{{p^{1}}, i}}{\p v_{j}} 
			&= 
			\frac{\p_{i} \underline{\eta}_{{p}^{1}}(\underline{x}^{1}) }{\sqrt{g_{p^{1},ii}(\underline{x}^{1})  }} \cdot e_{j}
			+  
			 \frac{\p \X^{1}_{{p^{1}}, 1}}{\p v_{j}} 
			\p_{1}
			\Big( \frac{\p_{i} \underline{\eta}_{{p}^{1}}}{\sqrt{g_{{p^{1}}, ii}}}  \Big)\Big|_{\underline{x}^{1}} \cdot  \underline{v} , \quad i = 1,3, \ j = 1,3. \\	\\
			\end{split}
			\end{equation}
			Here, $e_{j}$ is the $j^{th}$ directional standard unit vector in $\mathbb{R}^{3}$. \\ 
			\noindent	Moreover,
			\begin{equation} \label{Est--7}
			\begin{split}
			\frac{ \p | \underline{\mathbf{v}}^{1}_{p^{1}}  | }{\p x_{j}} &= 0,
			\end{split}
			\end{equation}
			\begin{equation} \label{Est--8}
			\begin{split}
			\frac{\p |\underline{\mathbf{v}}^{1}_{p^{1}}|}{\p v_{j}} &=  \lim_{s \downarrow t^{1}} \frac{V_{j}(s;t, \underline{x}, \underline{v})}{|\underline{V}(s;t, \underline{x}, \underline{v})|} .  \\
			\end{split}
			\end{equation}
		\end{lemma}	
		
		\begin{proof}
			We have
			\begin{equation}\label{global down_V}
			\lim_{s\downarrow t^{1}} \underline{V}(s;t, \underline{x}, \underline{v}) = \underline{v} ,\quad
			\underline{X}( \tau; t, \underline{x}, \underline{v}) = \underline{x} +   \underline{v} (\tau- t)  .
			\end{equation}
			Especially, when $\tau=t^{1}$, we get
			\begin{equation}\label{global position identity}
			\underline{X}( t^{1}; t, \underline{x}, \underline{v}) = \underline{x} +   \underline{v} (t^{1}- t) .
			\end{equation}
			From (\ref{global down_V}), we have
			\begin{equation*} \label{global V over x}
			\begin{split}
			\lim_{s \downarrow t^{1}} \frac{\p \underline{V} (s;t, \underline{x}, \underline{v})}{\p x_{j}} &= 0.
			\end{split}
			\end{equation*}
					
			To prove (\ref{Est--1}) - (\ref{Est--6}), these estimates are very similar with those of Lemma~\ref{Jac_billiard}. We are suffice to choose global euclidean coordinate instead of $\underline{\eta}_{p^{k}}$. Therefore we should replace
			\begin{equation*} \label{replace}
			\underline{\eta}_{p^{k+1}} \rightarrow \underline{\eta}_{p^{1}},\quad \underline{\eta}_{p^{k}} \rightarrow \underline{x},\quad t^{k} \rightarrow t,\quad t^{k+1} \rightarrow t-\tb=t^{1},\quad \p_{x_{j}} \underline{x} = e_{j}.
			\end{equation*}
			
			\textit{Proof of (\ref{Est--1})}. For $j=1,3$, we apply $\p x_{j}$ to (\ref{global position identity}) and take $\cdot \frac{\p_{3} \underline{\eta}_{p}^{1}}{\sqrt{g_{p^{1},33}}} \Big|_{\underline{x}^{1}}$. In this case,  we have ${\p \underline{v} \over \p x_{j} } = 0$. Then we get 
			\begin{equation*} \label{instead Estimate 1}
			\begin{split}
			\frac{\p \tb}{\p x_{j}} &=
			- 
			\frac{1}{ \V^{1}_{ {{p}^{1}} ,3 }  }
			\frac{  \p_{3} \underline{\eta}_{{p}^{1}}(\underline{x}^{1})}{\sqrt{g_{{p}^{1},33}(\underline{x}^{1})}} \cdot 
			e_{j} .  \\
			\end{split}	
			\end{equation*}
			
			\textit{Proof of (\ref{Est--2})}. For $j=1,2$, we apply $\p v_{j}$ to (\ref{global position identity}) and take $\cdot \frac{\p_{3} \underline{\eta}_{p}^{1}}{\sqrt{g_{p^{1},33}}} \Big|_{\underline{x}^{1}}$. Then we get
			\begin{equation*}  \label{instead v diff pos iden dot 3}
			\begin{split}
			0 &=  \frac{\p \X^{1}_{{p}^{1},1}}{\p v_{j}} \frac{\p \underline{\eta}_{{p}^{1}} }{\p \X^{k}_{{p}^{1},1}}\Big\vert_{\underline{x}^{1}}  \cdot   \frac{\p_{3} \underline{\eta}_{{p}^{1}} }{\sqrt{g_{{p}^{1}, 33} }} \Big|_{\underline{x}^{1}}  
			=	\Big\{ - ( t - t^{1}) e_{j} - \frac{\p( t - t^{1})}{\p v_{j} } \lim_{s \downarrow t^{1}} \underline{V} (s;t, \underline{x}, \underline{v}) \Big\} \cdot  \frac{\p_{3} \underline{\eta}_{{p}^{1}} }{\sqrt{g_{{p}^{1}, 33} }} \Big|_{\underline{x}^{1}}.  \\
			\end{split}
			\end{equation*}
			We use (\ref{global down_V}) to get (\ref{Est--2}). \\		
			
			\textit{Proof of (\ref{Est--3})}. For $j=1,3$, we apply $\p x_{j}$ to (\ref{global position identity}) and take $\cdot \frac{\p_{1} \underline{\eta}_{p^{1}}}{\sqrt{g_{p^{1},11}}} \Big|_{\underline{x}^{1}}$. And then,
			
			\begin{equation*} \label{instead Estimate 2}
			\begin{split}
			\frac{\p \X_{{p}^{1},1}^{1}}{\p  x_{j} }   
			&=    \frac{1}{   \V^{1}_{ {{p}^{1}} ,3 }   }
			\frac{  \p_{{3}} \underline{\eta}_{{p}^{1}}(\underline{x}^{1})}{\sqrt{g_{{p}^{1},33}(\underline{x}^{1})}} \cdot 
			e_{j} 
			\frac{ \V^{1}_{p^{1}, 1} }{\sqrt{g_{p^{1}, 11}}} \Big|_{\underline{x}^{1}}  + \frac{\p_{1} \underline{\eta}_{p^{1}} }{ {g_{p^{1}, 11} }} \Big|_{\underline{x}^{1}}  \cdot  e_{j} . \\
			\end{split}
			\end{equation*}
			This yields (\ref{Est--3}). \\
			
			\textit{Proof of (\ref{Est--4})}. For $j=1,3$, we apply $\p v_{j}$ to (\ref{global position identity}) and take $\cdot \frac{\p_{1} \underline{\eta}_{p^{1}}}{\sqrt{g_{p^{1},11}}} \Big|_{\underline{x}^{1}}$. 	
			\begin{eqnarray*}
				\frac{\p \X^{1}_{p^{1}, 1}}{\p v_{j} } 
				&=&
				\Big\{- \frac{\p (t - t^{1})}{\p v_{j} } \lim_{s\downarrow t^{1}} \underline{V} (s; t, \underline{x}, \underline{v})- (t - t^{1}) \frac{\p \underline{v} }{\p v_{j} }
				\Big\}
				\cdot \frac{\p_{{1}} \underline{\eta}_{{p}^{1}}}{g_{{p}^{1}, 11}}\Big|_{\underline{x}^{1}} .  \\
			\end{eqnarray*}
			And then we get (\ref{Est--4}). \\
			
			\textit{Proof of (\ref{Est--5})}. For $j=1,3$, we apply $\p x_{j}$ to \begin{equation}  
			\begin{split}\label{global v_123}
			{\V}^{1}_{{p}^{1},1}   
			&=  \frac{ \p_{1} \underline{\eta}_{{p}^{1}} }{\sqrt{g_{{p}^{1},11}}}  \Big|_{\underline{x}^{1}} \cdot  \lim_{s \downarrow  t^{1}} \underline{V} (s; t , \underline{x} ,\underline{v} ), \\
			{\V}^{1}_{{p}^{1},3}  & =  
			- \frac{ \p_{3} \underline{\eta}_{{p}^{1}} }{\sqrt{g_{{p}^{1},33}}}  \Big|_{\underline{x}^{1}} \cdot  \lim_{s \downarrow  t^{1}} \underline{V} (s; t , \underline{x} ,\underline{v} ).
			\end{split}
			\end{equation}
			From (\ref{v_123}), 
			\begin{equation*}
			\begin{split}
			\frac{\p \V^{1}_{{p}^{1}, 1}}{\p x_{j} }  =&
			\frac{\p \X^{1}_{{p}^{1}, 1}}{ \p x_{j} } \frac{\p}{\p \X^{1}_{{p}^{1}, 1}}\Big(  \frac{\p_{1} \underline{\eta}_{{p}^{1}}  }{\sqrt{g_{{p}^{1}, 11}    }  }  \Big)\Big|_{\underline{x}^{1}} \cdot   \lim_{s \downarrow  t^{1}} \underline{V} (s; t , \underline{x} ,\underline{v} ),  
			\end{split}
			\end{equation*}	
			\begin{equation*} \begin{split}
			\frac{\p \V^{1}_{{p}^{1}, 3}}{\p x_{j} }  =&
			-
			\frac{\p \X^{1}_{{p}^{1}, 1}}{\p x_{j} } \frac{\p}{\p \X^{1}_{{p}^{1}, 1}}\Big(  \frac{\p_{3} \underline{\eta}_{{p}^{1}}  }{\sqrt{g_{{p}^{1}, 33}    }  }  \Big)\Big|_{\underline{x}^{1}} \cdot  \lim_{s \downarrow  t^{1}} \underline{V} (s; t , \underline{x} ,\underline{v} ).\end{split}
			\end{equation*}
			From (\ref{global down_V}), (\ref{Est--3}), and (\ref{Est--1}), we prove (\ref{Est-5}).\\
			
			\textit{Proof of (\ref{Est--6})}. Similar as above, we apply $\p v_{j}$ to (\ref{global v_123}) and then use (\ref{global down_V}), (\ref{Est--4}), and (\ref{Est--2}). We skip detail. \\
			
			\textit{Proof of (\ref{Est--7})}. Since there is no external force speed is constant, so result is obvious.  \\
			
			\textit{Proof of (\ref{Est--8})}. Note that $|\underline{\V}_{p^{1}}^{1}| = \lim_{s \downarrow t^{1}} {|\underline{V}(s;t,\underline{x},\underline{v})|}$ and 
			
			\begin{equation*}
			\begin{split}
			2 | \underline{\V}_{p^{1}}^{1} | \frac{\p | \underline{\V}_{p^{1}}^{1} | }{ \p v_{j} } = 2 \lim_{s \downarrow t^{1}}  \underline{V} (s;t,\underline{x},\underline{v}) \cdot \lim_{s \downarrow t^{1}}  \p_{v_{j}} \underline{V} (s;t,\underline{x},\underline{v}) ,
			\end{split}
			\end{equation*}
			so we have
			\begin{equation} \label{norm over v}
			\begin{split}
			\frac{\p | \underline{\V}_{p^{1}}^{1} | }{ \p v_{j} } &= \lim_{s \downarrow t^{1}} { \underline{V} (s;t,\underline{x},\underline{v}) \over | \underline{V} (s;t,\underline{x},\underline{v})| } \cdot \lim_{s \downarrow t^{1}}  \p_{v_{j}} \underline{V} (s;t,\underline{x}, \underline{v}) .
			\end{split}
			\end{equation}
			Since
			\begin{equation} \label{global V over v}
			\begin{split}
			\lim_{s \downarrow t^{1}} \frac{\p \underline{V} (s;t,\underline{x},\underline{v})}{\p v_{j}} &= e_{j} ,
			\end{split}
			\end{equation}
			we combine (\ref{norm over v}), (\ref{global V over v}), and (\ref{Est--2}) to derive (\ref{Est--8}). \\		
	\end{proof}

	\begin{lemma}\label{det_billiard} 
		Assume $\O$ satisfies Definition~\ref{AND} and $\frac{1}{N} \leq |v  | \leq N$, for $1 \ll N$. Also we assume $|t^{k}-t^{k+1}| \leq 1$ and $|\V^{k}_{p^{k},3}|, \ | \V^{k+1}_{p^{k+1},3}| > 0$. Then 
		
		\begin{equation}\notag 
		\begin{split}
		& \bigg| \det \left[\begin{array}{cc} \p_{\X^{k}_{{p}^{k}, 1} } \X^{k+1}_{{p}^{k+1}, 1} & \nabla_{ \underline{\V}^{k}_{{p}^{k}} }  \X^{k+1}_{{p}^{k+1}, 1}\\
		\p_{\X^{k}_{{p}^{k}, 1}}  \underline{\V}^{k+1}_{{p}^{k+1}} 
		&  \nabla_{ \underline{\V}^{k}_{{p}^{k}}  }  \underline{\V}^{k+1}_{{p}^{k+1}}
			\end{array}\right]_{3\times 3} \bigg|
			= \frac{ \sqrt{g_{{p}^{k },11} (\underline{x}^{k}) } }{ \sqrt{g_{{p}^{k+1},11}  (\underline{x}^{k+1})  }  }	
			\frac{
				|\V^{k }_{p^{k },3}|
			}{|\V^{k+1}_{p^{k+1},3}|},
			\end{split}
			\end{equation}
			for the mapping $(\X^{k}_{p^{k},1}, \underline{\V}^{k}_{p^{k}}) \mapsto (\X^{k+1}_{p^{k+1},1}, \underline{\V}^{k}_{p^{k+1}})$.
	\end{lemma}
	\begin{proof}
		We note that Lemma~\ref{Jac_billiard} holds for nonconvex domain and result is exactly same as Lemma 26 in \cite{KimLee}, without external potential. Then simplified two-dimensional version directly yields above result. \\
	\end{proof}
	
	\begin{lemma} \label{5X5}
		We define, 
		\begin{equation*}\label{hat_v}
		\hat{\V}^{k}_{{p}^{k},1} := \frac{\V^{k}_{{p}^{k},1}}{|\underline{\V}^{k}_{{p}^{k}}|}, \quad |\underline{\V}^{k}_{p^{k}}| = \sqrt{ ({\V}^{k}_{{p}^{k},1})^{2} + ({\V}^{k}_{{p}^{k},3})^{2}}, 
		\end{equation*}	
		where $\underline{\V}^{k}_{p^{k}} = \underline{\V}^{k}_{p^{k} }(t,\underline{x},\underline{v})$ are defined in (\ref{x^k}).
		Assume $\frac{1}{N} \leq |v| \leq N$ and $|\V^{k}_{p^{k},3}|, \ | \V^{k+1}_{p^{k+1},3}| > \delta_{2} > 0$ for $1 \ll N$ and $k\lesssim_{\O, N, \delta_{2}} 1$. If $|t-t^{k}| \leq 1$, then  \begin{equation}\label{Jac_hat}
		\bigg|\det\left[\begin{array}{cc} 
		\frac{\p \X^{k}_{{p}^{k},1}}{\p \X^{1}_{{p}^{1},1}} & \frac{\p \X^{k}_{{p}^{k},1}}{\p \hat{\V}^{1}_{{p}^{1},1}} \\
		\frac{\p \hat{\V}^{k}_{{p}^{k},1}}{\p \X^{1}_{{p}^{1},1}} & \frac{\p   \hat{\V}^{k}_{{p}^{k},1}}{\p \hat{\V}^{1}_{{p}^{1},1}} 
		\end{array}\right] \bigg| \  >  \ \epsilon_{ \Omega,N, \delta_{2}} >0,
		\end{equation}	
		where $ t^{1}  =   t^{1} (t,x,v),$ $\X^{1}_{p^{1},i}  =  \X^{1}_{p^{1},i} (t,x,v),$ $\hat{\V}^{1}_{p^{1},i} = \hat{\V}^{1}_{p^{1},i} (t,x,v)$, and 
		\begin{eqnarray*}
			\X^{k}_{{p}^{k},i} &=&   \X^{k}_{{p}^{k},i} (t^{1}, \X^{1}_{p^{1},1},    \hat{\V}^{1}_{p^{1},1},  |\underline{\V}^{1}_{p^{1}}| ), \\ 
			\hat{\V}^{k}_{{p}^{k},i} &=&   \hat{\V}^{k}_{{p}^{k},i} (t^{1}, \X^{1}_{p^{1},1},   \hat{\V}^{1}_{p^{1},1},   |\underline{\V}^{1}_{p^{1}}| ).
		\end{eqnarray*}
		Here, the constant $\epsilon_{ \Omega,N,  \delta_{2}}>0$ does not depend on $t$ and $x$.\end{lemma}
	\begin{proof} 
		
		\textit{Step 1.} We compute 
		
		\begin{equation} \label{similar form}
		\begin{split}
		J^{i+1}_{i} &:= \frac{\p ( \X^{i+1}_{{p}^{i+1},1},   \hat{\V}^{i+1}_{{p}^{i+1},1} ,   |  \underline{\V}^{i+1}_{{p}^{i+1}} |)}{\p (
			\X^{i }_{{p}^{i },1},   \hat{\V}^{i }_{{p}^{i },1} ,    |\underline{\V}^{i }_{{p}^{i}} |
			)} 	\\
		&=
		\underbrace{ \frac{\p ( \X^{i}_{{p}^{i},1}, \underline{\V}^{i}_{{p}^{i}} )}{\p (
				\X^{i }_{{p}^{i },1},   \hat{\V}^{i }_{{p}^{i },1} ,  | \underline{\V}^{i }_{{p}^{i}} |
				)} 	
		}_{=Q_{i}}
		\underbrace{
			\frac{\p ( \X^{i+1}_{{p}^{i+1},1},  \underline{\V}^{i+1}_{{p}^{i+1}}  )}{\p (
				\X^{i }_{{p}^{i },1},  \underline{\V}^{i}_{{p}^{i}} 
				)}
		}_{=P_{i}}
		\underbrace{ \frac{\p ( \X^{i+1}_{{p}^{i+1},1},  \hat{\V}^{i+1}_{{p}^{i+1},1} ,  |  \underline{\V}^{i+1}_{{p}^{i+1}} |)}{\p (
				\X^{i+1 }_{{p}^{i+1 },1},  \underline{\V}^{i+1 }_{{p}^{i+1}}	)} 
		}_{=Q_{i+1}} .
		\\
		\end{split}
		\end{equation} 		
		For $Q_{i}$,
		\begin{equation*} \label{Qi}
		\begin{split}
		Q_{i} 
		&= 
		\left[\begin{array}{c|cc} 
		1 & 0 &0  \\ \hline
		0&  { \p\V_{p^{i},1}^{i} \over \p\hat{\V}_{p^{i},1}^{i} }  & { \p\V_{p^{i},1}^{i} \over \p |\underline{\V}_{p^{i}}^{i}| }  \\
		0&  { \p\V_{p^{i},3}^{i} \over \p\hat{\V}_{p^{i},1}^{i} }
		 & { \p\V_{p^{i},3}^{i} \over \p |\underline{\V}_{p^{i}}^{i}| }
		\end{array}\right] 
		=
		\left[\begin{array}{c|cc} 
		1 & 0 &0  \\ \hline
		0&  |\underline\V_{p^{i}}^{i}|  & { \p\V_{p^{i},1}^{i} \over \p |\underline{\V}_{p^{i}}^{i}| }  \\
		0&  { \p\V_{p^{i},3}^{i} \over \p\hat{\V}_{p^{i},1}^{i} }
		&  { \p\V_{p^{i},3}^{i} \over \p |\underline{\V}_{p^{i}}^{i}| }
		\end{array}\right]. 	\\
		\end{split}
		\end{equation*}
		For $Q_{i+1}$,
		\begin{equation} \label{Qi+1}
		\begin{split}
		Q_{i+1} 
		= 
		\left[\begin{array}{c|cc} 
		1 & 0 &0  \\ \hline
		0& { \p\hat{\V}_{p^{i+1},1}^{i+1} \over \p{\V}_{p^{i+1},1}^{i+1} } & { \p\hat{\V}_{p^{i+1},1}^{i+1} \over \p{\V}_{p^{i+1},3}^{i+1} }  \\
		0& { \p|\underline{\V}_{p^{i+1}}^{i+1}| \over \p{\V}_{p^{i+1},1}^{i+1} }
		 & { \p|\underline{\V}_{p^{i+1}}^{i+1}| \over \p {\V}_{p^{i+1},3}^{i+1} }
		\end{array}\right] 	
		&=
		\left[\begin{array}{c|cc} 
		1 & 0 & 0\\ \hline
		0& | \underline{\V}_{p^{i+1}}^{i+1} |^{-1}  & { \p\hat{\V}_{p^{i+1},1}^{i+1} \over \p{\V}_{p^{i+1},3}^{i+1} }  \\
		 0& { \p |\underline{\V}_{p^{i+1}}^{i+1}| \over \p{\V}_{p^{i+1},1}^{i+1} }
		  & { \p |\underline{\V}_{p^{i+1}}^{i+1}| \over \p {\V}_{p^{i+1},3}^{i+1} }
		\end{array}\right]. 	\\
		\end{split}
		\end{equation}
		Note that 
		\begin{equation} \label{1323}
		\begin{split}
		\frac{ \p \hat{\V}_{p^{i+1},{1}}^{i+1} }{ \p \V_{p^{i+1},{3}}^{i+1} } 
		&=
		\V_{p^{i+1},{1}}^{i+1} \frac{\p}{\p \V_{p^{i+1},{3}}^{i+1}} \big({1\over |\underline{\V}_{p^{i+1}}^{i+1}|} \big) = \frac{ \V_{p^{i+1},{1}}^{i+1} \V_{p^{i+1},{3}}^{i+1} }{ |\underline{\V}_{p^{i+1}}^{i+1}|^{3} } ,
		\end{split}
		\end{equation}
		and for $k=1,3$,
		\begin{equation} \label{313233}
		\begin{split}
		\frac{ \p |\underline{\V}_{p^{i+1}}^{i+1}| }{ \p \V_{p^{i+1},{k}}^{i+1} } 
		&=
		- \frac{ \V_{p^{i+1},{k}}^{i+1} }{ |\underline{\V}_{p^{i+1}}^{i+1}| }.
		\end{split}
		\end{equation}
		From (\ref{Qi+1}), (\ref{1323}), and (\ref{313233}),
		\begin{equation} \label{det Qi+1}
		\begin{split}
		\det Q_{i+1} &= \frac{1}{|\underline{\V}_{p^{i+1}}^{i+1}|} \Big( -\frac{\V_{p^{i+1},{3}}^{i+1}}{ |\underline{\V}_{p^{i+1}}^{i+1}| } \Big) + \frac{ \V_{p^{i+1},{1}}^{i+1} \V_{p^{i+1},{3}}^{i+1} }{ |\underline{\V}_{p^{i+1}}^{i+1}|^{3} } \frac{ \V_{p^{i+1},{1}}^{i+1} }{ |\underline{\V}_{p^{i+1}}^{i+1}| }  = - \frac{ (\V_{p^{i+1},{3}}^{i+1})^{3} }{ |\underline{\V}_{p^{i+1}}^{i+1}|^{4} } .
		\end{split}
		\end{equation}
		By taking inverse, we get
		\begin{equation} \label{det Qi}
		\begin{split}
		\det Q_{i} &= - \frac{ |\underline{\V}_{p^{i}}^{i}|^{4} } { (\V_{p^{i},{3}}^{i})^{3} }.
		\end{split}
		\end{equation}
		From (\ref{similar form}), (\ref{det Qi}), (\ref{det Qi+1}), and Lemma~\ref{det_billiard}, we get
		\begin{equation}\notag 
		\begin{split}
		& \bigg| \det \left[\begin{array}{cc} \p_{\X^{k}_{{p}^{k}, 1} } \X^{k+1}_{{p}^{k+1}, 1} & \nabla_{ \underline{\V}^{k}_{{p}^{k}} }  \X^{k+1}_{{p}^{k+1}, 1}\\
		\p_{\X^{k}_{{p}^{k}, 1}}  \underline{\V}^{k+1}_{{p}^{k+1}} 
		&  \nabla_{ \underline{\V}^{k}_{{p}^{k}}  }  \underline{\V}^{k+1}_{{p}^{k+1}}
		\end{array}\right]_{3\times 3} \bigg|
		= \frac{ \sqrt{g_{{p}^{k },11} (\underline{x}^{k}) } }{ \sqrt{g_{{p}^{k+1},11}  (\underline{x}^{k+1})  }  }	
		\frac{
			|\V^{k }_{p^{k },3}|
		}{|\V^{k+1}_{p^{k+1},3}|} .
		\end{split}
		\end{equation}
		Therefore,
		\begin{equation*} \label{det Jii+1}
		\begin{split}
		| \det J_{i}^{i+1} | &= | \det Q_{i} \det P_{i} \det Q_{i+1} |  \\
		&=  \frac{ |\underline{\V}_{p^{i}}^{i}|^{4} } { (\V_{p^{i},{3}}^{i})^{3} } 
		 \frac{ \sqrt{g_{{p}^{i },11} }  \big|_{\underline{x}^{i}} }{\sqrt{g_{{p}^{i+1 },11}   }   \big|_{\underline{x}^{i+1}}  }	
		\frac{
			|\V^{i }_{p^{i },3}|
		}{|\V^{i+1}_{p^{i+1},3}|}
		\frac{ (\V_{p^{i+1},{3}}^{i+1})^{3} }{ |\underline{\V}_{p^{i+1}}^{i+1}|^{4} }  
		= 
		 \frac{ \sqrt{g_{{p}^{i },11} }   \big|_{\underline{x}^{i}} }{\sqrt{g_{{p}^{i+1 },11}   }   \big|_{\underline{x}^{i+1}}  }	
		\frac{ |\V_{p^{i+1},{3}}^{i+1}|^{2} }{ |\V_{p^{i},{3}}^{i}|^{2} } ,
		\\
		\end{split}
		\end{equation*}
		and we get
		\begin{equation} \label{det J1k}
		\begin{split}
		| \det J_{1}^{k} | 
		&= 
		 \frac{ \sqrt{g_{{p}^{1 },11} } \big|_{ \underline{x}^{1}} }{\sqrt{g_{{p}^{k },11}   }   \big|_{\underline{x}^{k}}  }	
		\frac{ |\V_{p^{k},{3}}^{k}|^{2} }{ |\V_{p^{1},{3}}^{1}|^{2} } .
		\\
		\end{split}
		\end{equation}
		
		\noindent \textit{Step 2.} From (\ref{down_V}),
		\begin{equation*} \label{hatv part1}
		\begin{split}
		2| \underline{\V}^{i+1}_{p^{i+1}} |\frac{\p | \underline{\V}^{i+1}_{p^{i+1}} |}{\p  \V^{i}_{p^{i}, n}}		
		&=\frac{ \p | \underline{V}(t^{i+1};t^{i}, \underline{x}^{i}, \underline{v}^{i}) |^{2}}{\p \V^{i}_{p^{i}, n}}
		= 2 \frac{\p \underline{V}(t^{i+1};t^{i}, \underline{x}^{i}, \underline{v}^{i})  }{\p \V^{i}_{p^{i}, n}} \cdot  \underline{V}(t^{i+1};t^{i}, \underline{x}^{i}, \underline{v}^{i})
		\\
		&= 
		2 \frac{\p_{n} \underline{\eta}_{p^{i}}}{\sqrt{g_{p^{i}, nn}}}\Big|_{\underline{x}^{i}}
		\cdot 
		\underline{V}(t^{i+1};t^{i}, \underline{x}^{i}, \underline{v}^{i})
		= 2 \V^{i}_{p^{i},n} .
		\end{split}
		\end{equation*}
		Therefore, we get
		\begin{equation}\label{|v|_v}
		\frac{\p | \underline{\V}^{i+1}_{p^{i+1}}  |}{\p  \V^{i}_{p^{i}, n}} = \frac{\V^{i}_{p^{i},n}}{|\underline{\V}^{i+1}_{p^{i+1}}|}, \ \ \ \text{for} \ \ n=1,3.
		\end{equation}
		Since speed is conserved, for $n=1,3$,
		\begin{equation}\label{|v|_x}
		\frac{\p | \underline{\V}^{i+1}_{p^{i+1}}  |}{\p  \X^{i}_{p^{i}, n}}	 
		= 0   , \ \ \ \text{for} \ \ n=1,3.
		\end{equation}
		
		\noindent Also, by conservation,
		\begin{equation}\label{|v|_|v|}
		\begin{split}
		\frac{  \p |\underline{\V}^{i+1}_{p^{i+1}}|}{\p |\underline{\V}^{i}_{p^{i}}|  }  &= 1.
		\end{split}\end{equation}
		
		\noindent \textit{Step 3.} From (\ref{det J1k}), (\ref{|v|_v}), (\ref{|v|_x}), and (\ref{|v|_|v|}),
		\begin{equation*} 
		\begin{split}
		| \det J_{1}^{k} | 
		&= 
		 \frac{ \sqrt{g_{{p}^{1 },11} }   \big|_{\underline{x}^{1}} }{\sqrt{g_{{p}^{k },11}   }   |_{ \underline{x}^{k}}  }	
		\frac{ |\V_{p^{k},{3}}^{k}|^{2} }{ |\V_{p^{1},{3}}^{1}|^{2} } \\
		&=
		\bigg|\det\left[\begin{array}{c|c|c} 
		\frac{\p  \X^{k}_{{p}^{k},1} }{\p \X^{1}_{{p}^{1},1} } & \frac{\p \X^{k}_{{p}^{k},1}  }{\p \hat{\V}^{1}_{{p}^{1},1} } & 0 \\ \hline
		\frac{\p \hat{\V}^{k}_{{p}^{k},1} }{\p \X^{1}_{{p}^{1},1} } & \frac{\p \hat{\V}^{k}_{{p}^{k},1} }{\p \hat{\V}^{1}_{{p}^{1},1} } & 0  \\ \hline
		0  & 0 & 1 
		\end{array}\right]_{3\times 3} \bigg| 
		= 
		\bigg|\det\left[\frac{\p (\X^{k}_{{p}^{k},1},   \hat{\V}^{k}_{{p}^{k},1} ) }{\p (\X^{1}_{{p}^{1},1},  \hat{\V}^{1}_{{p}^{1},1} ) }\right]_{2\times 2}  \bigg| .
		\end{split}
		\end{equation*}
		Therefore, we conclude (\ref{Jac_hat}) by (\ref{det J1k}).  \\
	\end{proof}

	Now we study lower bounded of $\det\big(\frac{d \underline{X}}{d\underline{v}}\big)$. Instead of Euclidean variable $\underline{v} = (v_{1}, v_{3})$, we introduce new variables via geometric decomposition. In two-dimensional cross section, we split velocity $\underline{v}$ into speed and direction,
	\begin{equation*} \label{hat def v}
		|\underline{v}| \quad\text{and} \quad \hat{v}_{1} := \frac{v_{1}}{|\underline{v}|}.
	\end{equation*}
	Note that $\{\p_{|v|}, \p_{\hat{v}_{1}}\}$ are independent if $v_{3} \geq \frac{1}{N} > 0$. So under assumption of $v_{3} \geq \frac{1}{N} > 0$, we perform $\p_{|v|}, \p_{\hat{v}_{1}} $, instead of $\p_{x_{1}}, \p_{x_{2}}$. We assume $\frac{1}{N}\leq |\underline{v}| \leq N$, $t^{k+1}(t,\underline{x},\underline{v})<s<t^{k}(t,\underline{x},\underline{v})$, and $|\V^{i}_{p^{i},3}| > \delta_{2} > 0$ (nongrazing) for $1\leq \forall i \leq k$.
	When we differentiate $\underline{X}$ by speed $|\underline{v}|$,
	\begin{equation} \label{dX d|v|}
	\begin{split}
		\p_{|\underline{v}|} \underline{X} (s;t,\underline{x}, \underline{v}) 
		&= \p_{|\underline{v}|} \Big( \underline{\eta}_{{p}^{k}} (\X_{{p}^{k},1}^{k}, 0) - (t^{k} - s)||\underline{v}|| \hat{v}^{k} \Big) ,\quad \hat{v}^{k} :=  \frac{\underline{v}^{k}}{|\underline{v}|} \\
		&= { \p_{|\underline{v}|} \X^{k}_{{p}^{k}, 1} \p_{1} \underline{\eta}_{{p}^{k}} (\X_{{p}^{k},1}^{k}, 0) } 
		+ {\p_{|\underline{v}|} \big[ ( t-t^{k})  |\underline{v}^{k}|\big]}
		\hat{v}^{k} - { ( t - s) \big[ \p_{|\underline{v}|}  |\underline{v}^{k}| } \big] \hat{v}^{k}
		\\
		& - (t^{k}-s)|\underline{v}^{k} |  \p_{\underline{|v|}} \hat{v}^{k}  
		=  - (t - s) \hat{v}^{k},
	\end{split}
	\end{equation}	
	where we used $\p_{|\underline{v}|} \X^{k}_{{p}^{k}, 1} = 0$, $\p_{|\underline{v}|} \big[ ( t-t^{k})  |\underline{v}^{k}|\big] = 0$, $\p_{\underline{|v|}} \hat{v}^{k} = 0$, and $\p_{|\underline{v}|}  |\underline{v}^{k}| = 1$. Note that this is because, bouncing position $x^{k}$, travel length until $x^{k}$, direction of $v^{k}$ are independent to $|\underline{v}|$.    \\
	
	On the other hand, differentiating $\underline{X}$ by $\hat{v}_{1}$,
	\begin{equation} \label{dX dv1 pre}
	\begin{split}
		\p_{\hat{v}_{1}}   [\underline{X}(s;t,\underline{x},\underline{v}) ]  
		&= \p_{\hat{v}_{1}} \X^{k}_{{p}^{k}, 1} \p_{1} \underline{\eta}_{{p}^{k}} (\X^{k}_{{p}^{k},1 },  0) -  \p_{\hat{v}_{1}}  t^{k} |\underline{\V}^{k}_{{p}^{k}}| 
		\underline{\hat{v}}^{k}		 
		- (t^{k}-s)| \underline{\V^{k}_{{p}^{k}}} | {\p}_{\hat{v}_{1}} \underline{\hat{v}}^{k} .  \\		
	\end{split}	
	\end{equation}
	To compute the last term ${\p}_{\hat{v}_{1}} \underline{\hat{v}}^{k}$, we use $\hat{\V}^{k}_{p^{k},3}
	= \sqrt{1- |\hat{ \mathbf{v}}_{{p}^{k},1}^k
		|^2  } $ and $|\hat{\V}^{k}_{p^{k},3}| > 0$ to get
	\begin{equation} \label{p v hat}
	\begin{split}
	\p_{\hat{v}_{1}} \underline{\hat{v}}^{k}  
	=& \p_{\hat{v}_{1}} \bigg[  \frac{\p_{1} \underline{\eta}_{{p}^{k}}}{\sqrt{g_{{p}^{k},11}}} (\mathbf{x}^k_{{p}^{k},1}, 0) \hat{\mathbf{v}}^{k}_{{p}^{k}, 1} + \frac{\p_3 \underline{\eta}_{{p}^{k}}}{\sqrt{g_{{p}^{k},33}}} (\mathbf{x}^k_{{p}^{k},1} , 0) \sqrt{1  - |\hat{ \mathbf{v}}_{{p}^{k},1}^k
		|^2  }  \bigg]
	\\
	=& {\p_{\hat{v}_{1}} \mathbf{x}^k_{{p}^{k},1}} \sum_{\ell=1,3} \p_{1} \Big[  \frac{\p_{\ell} \underline{\eta}_{{p}^{k}}}{\sqrt{g_{{p}^{k}, \ell\ell}}} \Big](\mathbf{x}^k_{{p}^{k},1},  0) \hat{\mathbf{v}}^{k}_{{p}^{k}, \ell}
	+ \frac{\p_{1} \underline{\eta}_{{p}^{k}}}{\sqrt{g_{{p}^{k},11}}} (\mathbf{x}^k_{{p},1},  0)  \p_{\hat{v}_{1}} \hat{\mathbf{v}}^{k}_{{p}^{k}, 1} \\
	&- \frac{\p_3 \underline{\eta}_{{p}^{k}}}{\sqrt{g_{{p}^{k},33}}} 
	(\mathbf{x}^k_{{p}^{k},1}, 0) \frac{1}{  \sqrt{1- |\hat{\mathbf{v}}^k_{{p}^{k},1}|^2  }   } \Big[ \hat{\mathbf{v}}^k_{{p}^{k},1}  \p_{\hat{v}_{1}} [\hat{\mathbf{v}}^k_{{p}^{k},1} ]  \Big]\\
	=&
	\bigg( \sum_{\ell=1,3} \frac{\p}{\p {\mathbf{x}^{k}_{{p}^{k}, 1}}}
	\Big[  \frac{\p_{\ell} \underline{\eta}_{{p}^{k}}}{\sqrt{g_{{p}^{k}, \ell\ell}}} \Big](\mathbf{x}^k_{{p}^{k},1},  0) \hat{\mathbf{v}}^k_{{p}^{k}, \ell} \bigg)
	{\p_{\hat{v}_{1}} \mathbf{x}^k_{{p}^{k}, 1}}
	\\
	&
	+
	\Big[ \frac{\p_{1} \underline{\eta}_{{p}^{k}}}{\sqrt{g_{{p}^{k},11}}}(\mathbf{x}^k_{{p}^{k},1}, 0)  -  \frac{\p_3 \underline{\eta}_{{p}^{k}}}{\sqrt{g_{{p}^{k},33}}}(\mathbf{x}^k_{{p}^{k},1},  0)\frac{\hat{\mathbf{v}}_{{p}^{k},1}^k }{ \hat{\mathbf{v}}_{{p}^{k},3}^{k} } \Big] {\p_{\hat{v}_{1}} [\hat{\mathbf{v}}_{{p}^{k},1}^k ]} .
	\end{split}\end{equation}
	
	\noindent Combining (\ref{dX dv1 pre}) and (\ref{p v hat}), we get
	\begin{equation}\label{p_X}
	\begin{split} 
	\p_{\hat{v}_{1}} \big[ \underline{X}(s;t,\underline{x},\underline{v})  \big]
	&= - \big( \p_{\hat{v}_{1}} t^{k} \big) v^{k}
	+ \p_{\hat{v}_{1}} \mathbf{x}^{k}_{{p}^{k},1}  \p_{1} \underline{\eta}_{{p}^{k}} (\mathbf{x}^{k}_{{p}^{k},1 }, 0)  \\
	& 
	- (t^{k}-s)| \underline{\mathbf{v}}_{{p}^{k}}^{k} | 
	\bigg( \sum_{\ell=1,3} \frac{\p}{\p {\mathbf{x}^{k}_{{p}^{k}, 1}}}
	\Big[  \frac{\p_{\ell} \underline{\eta}_{{p}^{k}} }{\sqrt{g_{{p}^{k}, \ell\ell}}} \Big](\mathbf{x}^k_{{p}^{k},1},  0) \hat{\mathbf{v}}^k_{{p}^{k}, \ell} \bigg)
	{\p_{\hat{v}_{1}} \mathbf{x}^k_{{p}^{k}, 1}}  \\
	&
	- (t^{k}-s)| \underline{\mathbf{v}}_{{p}^{k}}^{k} | 
	\Big[ \frac{\p_1 \underline{\eta}_{{p}^{k}, i} }{\sqrt{g_{{p}^{k},11}}}(\mathbf{x}^k_{{p}^{k},1}, 0)  -  \frac{\p_3 \underline{\eta}_{{p}^{k}, i}}{\sqrt{g_{{p}^{k},33}}}(\mathbf{x}^k_{{p}^{k},1}, 0)\frac{\hat{\mathbf{v}}_{{p}^{k},1}^k }{ \hat{\mathbf{v}}_{{p}^{k},3}^{k} } \Big] {\p_{\hat{v}_{1}} [\hat{\mathbf{v}}_{{p}^{k},1}^k ]} .  \\
	\end{split}
	\end{equation}
	
	\begin{definition}[Specular Basis and Matrix]
		Recall the specular cycles $(t^{k},\underline{x}^{k}, \underline{v}^{k})$ in (\ref{specular_cycles}). Assume
		\begin{equation}\label{no_grazing_0}
		\mathbf{n}(\underline{x}^{k}) \cdot \underline{v}^{k} \neq 0.
		\end{equation}
		Recall $\underline{\eta}_{p^{k}}$ in (\ref{eta}). We define the \textit{specular basis}, which is an orthonormal basis of $\R^{2}$, as 
		\begin{equation}\label{orthonormal_basis}
		\begin{split}
		&\mathbf{ {e}}^{k}_{0}  := \frac{\underline{v}^{k}}{|\underline{v}^{k}|} = \frac{1}{|\underline{v}^{k}|} (v^{k}_{1}, v^{k}_{3}) 
		, \mathbf{  {e}}^{k}_{\perp,1}
		:= \frac{1}{ |\underline{v}^{k} | } ( v^{k}_{3},  -v^{k}_{1}, 0 ) . \\  
		\end{split}
		\end{equation}
		Also, for fixed $k \in \mathbb{N}$, assume (\ref{no_grazing_0})
		with 	$\underline{x}^{k} = \underline{x}^{k} (t, \underline{x}, |\underline{v}|, \hat{v}_{1} )$ and $\underline{v}^{k} = \underline{v}^{k} (t, \underline{x}, |\underline{v}|, \hat{v}_{1} )$. We define the $2\times2$ specular transition matrix  
		$\mathcal{S}^{k, p^{k}} = \mathcal{S}^{k, p^{k}}(t, \underline{x}, |\underline{v}|, \hat{v}_{1})$ as 
		\begin{equation}\label{specular_transition_matrix}
		\mathcal{S}^{k, p^{k} }
		:= \left[\begin{array} {cc}
		\mathcal{S}_{1}^{k, p^{k} }	 & 0  \\
		\mathcal{S}_{2}^{k, p^{k} } & \mathcal{S}_{3}^{k, p^{k} }
		\end{array} \right]_{2\times 2},
		\end{equation}
		where 
		\begin{equation}\begin{split}\notag
		\mathcal{S}_{1}^{k, p^{k} } 
		&: = 
		\p_{1} \underline{\eta}_{{p}^{k}} \cdot \mathbf{e}^{k}_{\perp,1} ,  \\
		\mathcal{S}_{2}^{k, p^{k} } &: = 	
		\Big(  \sum_{\ell=1}^{3} \p_{1} \big[ \frac{\p_{\ell} \underline{\eta}_{{p}^{k}}}{\sqrt{g_{ {{p}^{k}},\ell\ell}}}  \big] \hat{\mathbf{v}}^{k}_{ {{p}^{k}},\ell}  \Big) \cdot \mathbf{e}^{k}_{\perp,1}  ,  \\
		\mathcal{S}_{3} ^{k, p^{k} }& : = 
		\Big[ \frac{\p_{1} \underline{\eta}_{p^{k}}}{\sqrt{g_{p^{k},11}}} - \frac{\p_{3} \underline{\eta}_{p^{k}}}{\sqrt{g_{p^{k},33}}}  \frac{\hat{\mathbf{v}}^{k}_{p^{k},1}}{\hat{\mathbf{v}}^{k}_{p^{k},3}} \Big] \cdot \mathbf{e}^{k}_{\perp,1} ,
		\end{split} \end{equation}
		and where $\eta_{p^{k}}$ and $g_{p^{k}}$ are evaluated at $\underline{x}^{k}(t, \underline{x}, |\underline{v}|, \hat{v}_{1} )$. We also define 
		\begin{equation}\label{specular_matrix}
			\begin{bmatrix}
				\mathcal{R}^{k, p^{k}}_{1} \\ \mathcal{R}^{k, p^{k}}_{2}
			\end{bmatrix}
			: = 	\mathcal{S}^{k, p^{k} } 
			\begin{bmatrix}
				\frac{\p \mathbf{x}^{k}_{ {p}^{k},1}}{\p \hat{v}_{1} }
				\\
				\frac{\p \hat{\mathbf{v}}^{k}_{{p}^{k},1}}{\p \hat{v}_{1} }
			\end{bmatrix} ,
		\end{equation}	
		where $\underline{\mathbf{x}}^{k}_{{p}^{k}} = \underline{\mathbf{x}}^{k}_{{p}^{k}} (t, \underline{x}, |\underline{v}|, \hat{v}_{1} )$ and $\underline{\mathbf{v}}^{k}_{{p}^{k}} = \underline{\mathbf{v}}^{k}_{{p}^{k}}(t, \underline{x}, |\underline{v}|, \hat{v}_{1} )$. 
	\end{definition}

	\begin{lemma}\label{nonzero_sub}  
		Fix $k \in \mathbb{N}$ with $|t-t^{k}| \leq 1$. Assume $\frac{1}{N} \leq |\underline{v}| \leq N$ and $\frac{1}{N} \leq |v_{3}|$, for $N \gg 1$. We also assume non-grazing condition
		\begin{equation}\label{no_grazing_lemma}
			|\V^{i}_{p^{i},3}| = |\underline{v}^{i}(t, \underline{x} , \underline{v}) \cdot \mathbf{n}(\underline{x}^{i}(t, \underline{x} ,\underline{v}))|> \delta_{2} > 0, \quad \forall 1\leq i \leq k,
		\end{equation} 
		and
		\begin{equation} \label{geo e1}
			\Big| \frac{ \p_{1} \underline{\eta}_{ {p}^{1}}  }{ \sqrt{g_{{p}^{1}, 11} } } \cdot e_{1} \Big| > \frac{1}{N} > 0,
		\end{equation}
		for some uniform $\delta_{2} > 0$.
		Then there exist at least one $i \in \{1,2\}$ such that 
		\begin{equation} \label{nonzero_sub1}
		|\mathcal{R}_{i}^{k,{p}^{k} }   (t, x , v)| > \varrho_{\Omega, N, \delta_{2}} ,
		\end{equation}
		for some constant $\varrho_{\Omega, N , \delta_{2}}>0$. 
	\end{lemma}	 
	
	\begin{proof}

		\noindent First we claim that 
		\begin{equation}
		| \det (\mathcal{S}^{k,p^{k} }) | > \varrho_{\Omega, N, \delta_{2}}. \notag
		\end{equation}
		We suffice to compute diagonal entries. From (\ref{orthonormal_basis}),  
		\begin{eqnarray*}
			| \mathcal{S}_{1}^{k,p^{k} } |
			&=&
			\big| \sqrt{g_{{p}^{k}, 11} } \frac{ \p_{1} \underline{\eta}_{ {p}^{k}}  }{ \sqrt{g_{{p}^{k}, 11} } } \cdot \frac{\underline{v}^{k}_{\perp}}{|\underline{v}^{k}|} \big| \\
			&=& 
			\big| \sqrt{g_{{p}^{k}, 11} } \frac{1}{|\underline{v}^{k}|} \big( \frac{ \p_{1} \underline{\eta}_{ {p}^{k} } }{ \sqrt{g_{{p}^{k}, 11} } } \cdot \underline{v}^{k}_{\perp} \big) \big| = \sqrt{g_{{p}^{k}, 11} } \frac{\big|\V^{k}_{{p}^{k},3} \big|}{|\underline{v}^{k}|}  ,
		\end{eqnarray*}
		and	
		\begin{eqnarray*}
			| \mathcal{S}_{3}^{k, p^{k} } | 
			&=&
			\big| \Big[ \frac{\p_{1} \underline{\eta}_{p^{k}}}{\sqrt{g_{p^{k},11}}} - \frac{\p_{3} \underline{\eta}_{p^{k}}}{\sqrt{g_{p^{k},33}}}  \frac{\hat{\mathbf{v}}^{k}_{p^{k},1}}{\hat{\mathbf{v}}^{k}_{p^{k},3}} \Big] \cdot \mathbf{e}^{k}_{\perp,1} \big|  \\
			&=& \big| \frac{1}{ \hat{\mathbf{v}}^{k}_{p^{k},3} } \Big[ \hat{\mathbf{v}}^{k}_{p^{k},3} \frac{\p_{1} \underline{\eta}_{p^{k}}}{\sqrt{g_{p^{k},11}}} - {\hat{\mathbf{v}}^{k}_{p^{k},1}} \frac{\p_{3} \underline{\eta}_{p^{k}}}{\sqrt{g_{p^{k},33}}}  \Big] \cdot \mathbf{e}^{k}_{\perp,1} \big|  
			= \frac{ |\underline{\mathbf{v}}^{k}_{p^{k}}| }{ |\mathbf{v}^{k}_{p^{k},3}| } ,
		\end{eqnarray*}
		which implies uniform invertibility of $2\times 2$ matrix $\mathcal{S}^{k,p^{k}}$. To consider $2\times 1$ vector on the RHS of (\ref{specular_matrix}), we compute
		\begin{equation} \label{189}
		\begin{split}
			& \begin{bmatrix}
			\frac{\p \X_{{p}^{k},1}^{k}}{\p x_{1}} & \frac{\p \X_{{p}^{k},1}^{k}}{\p \hat{v}_{1}} \\
			\frac{\p \V_{{p}^{k},1}^{k}}{\p x_{1}} & \frac{\p \V_{{p}^{k},1}^{k}}{\p \hat{v}_{1}} \\
			\end{bmatrix}  
			=
			\begin{bmatrix}
			\frac{\p \X_{{p}^{k},1}^{k}}{\p \X_{{p}^{1},1}^{1}} & \frac{\p \X_{{p}^{k},1}^{k}}{\p \V_{{p}^{1},1}^{1}} \\
			\frac{\p \V_{{p}^{k},1}^{k}}{\p \X_{{p}^{1},1}^{1}} & \frac{\p \V_{{p}^{k},1}^{k}}{\p \V_{{p}^{1},1}^{1}} \\
			\end{bmatrix}
			\begin{bmatrix}
			\frac{\p \X_{{p}^{1},1}^{1}}{\p x_{1}} & \frac{\p \X_{{p}^{1},1}^{1}}{\p \hat{v}_{1}} \\
			\frac{\p \V_{{p}^{1},1}^{1}}{\p x_{1}} & \frac{\p \V_{{p}^{1},1}^{1}}{\p \hat{v}_{1}}
			\end{bmatrix}  \\
			&= 
			\underbrace{ 
			\begin{bmatrix}
			\frac{\p \X_{{p}^{k},1}^{k}}{\p \X_{{p}^{1},1}^{1}} & \frac{\p \X_{{p}^{k},1}^{k}}{\p \V_{{p}^{1},1}^{1}} \\
			\frac{\p \V_{{p}^{k},1}^{k}}{\p \X_{{p}^{1},1}^{1}} & \frac{\p \V_{{p}^{k},1}^{k}}{\p \V_{{p}^{1},1}^{1}} \\
			\end{bmatrix}
			}_{A}
			\underbrace{	
			\begin{bmatrix}
			\frac{\p \X_{{p}^{1},1}^{1}}{\p x_{1}} & \frac{\p \X_{{p}^{1},1}^{1}}{\p \hat{v}_{1}}  \\
			\frac{\p \X_{{p}^{1},1}^{1}}{\p x_{1}} \p_{1}\big( \frac{\p_{1} \underline{\eta}_{{p}^{1}} }{ \sqrt{g_{{p}^{1},11}} } \big) \cdot \underline{v}
			&
			|v|\big( \frac{\p_{1} \underline{\eta}_{{p}^{1}} }{ \sqrt{g_{{p}^{1},11}} } \big)\cdot e_{1} + \frac{\p \X_{{p}^{1},1}^{1}}{\p \hat{v}_{1}} \p_{1}\big( \frac{\p_{1} \underline{\eta}_{{p}^{1}} }{ \sqrt{g_{{p}^{1},11}} } \big) \cdot \underline{v}
			\end{bmatrix} 
			}_{B} ,  \\
		\end{split}
		\end{equation}
		where we used (\ref{Est--5}) and (\ref{Est--6}). Determinant of $A$ is uniformly nonzero from (\ref{Jac_hat}) in Lemma~\ref{5X5}. From elementary row operation for $B$,
		\begin{equation*} \label{det B}
		\begin{split}
			\det B 
			&= 
			\det \begin{bmatrix}
			\frac{\p \X_{{p}^{1},1}^{1}}{\p x_{1}} & \frac{\p \X_{{p}^{1},1}^{1}}{\p \hat{v}_{1}}  \\
			0
			&
			|\underline{v}|\big( \frac{\p_{1} \underline{\eta}_{{p}^{1}} }{ \sqrt{g_{{p}^{1},11}} } \big)\cdot e_{1} 
			\end{bmatrix} .
		\end{split}
		\end{equation*} 
		From (\ref{Est--3}), $(1,1)$ entry of matrix $B$ is computed by 
		\begin{equation*}
		\begin{split}
		\big| \frac{\p \X_{{p}^{1},1}^{1}}{\p x_{1}} \big|
		&= \Big| 
		\frac{e_{1}}{\sqrt{g_{{p}^{1}, 11}  }}  
		\cdot \Big[ \frac{\p_{1} \underline{\eta}_{{p}^{1}} }{ \sqrt{g_{{p}^{1} ,11}  }}
		+  \frac{\V^{1}_{{p}^{1}, 1}  }{\V^{1}_{{p}^{1}, 3}  } 
		\frac{\p_{3} \underline{\eta}_{{p}^{1}}  }{ \sqrt{g_{{p}^{1}, 33} }}
		\Big] \Big|   \\
		&= \Big| \frac{1}{\sqrt{g_{{p}^{1},11} }} \frac{1}{ |\V^{1}_{{p}^{1}, 3}| } \frac{e_{1}}{\sqrt{g_{p^{1},11}}} \cdot \Big( \V^{1}_{{p}^{1}, 3} \frac{\p_{1} \eta_{{p}^{1}}}{\sqrt{g_{p^{1},11}}} + \V^{1}_{{p}^{1}, 1} \frac{\p_{3} \eta_{{p}^{1}}}{\sqrt{g_{p^{1},33}}} \Big) \Big|  \\
		&= \Big| \frac{1}{\sqrt{g_{{p}^{1},11} }} \frac{1}{ |\V^{1}_{{p}^{1}, 3}| } \frac{e_{1}}{\sqrt{g_{p^{1},11}}} \cdot \big( v_{1}e_{3} + v_{3}e_{1} \big) \Big|  
		= \frac{1}{g_{{p}^{1},11} (\underline{x}^{1})   }
		\frac{|v_{3}|}{| \V^{1}_{{p}^{1},3}|}. \label{UL 22}
		\end{split}
		\end{equation*} 
		Therefore, from (\ref{no_grazing_lemma}) and (\ref{geo e1}), determinant of $B$ is uniformly nonzero and thus LHS of (\ref{189}) has also uniformly nonzero determinant. This yields uniform nonzeroness of second column, i.e. $ \begin{bmatrix}
		\frac{\p \mathbf{x}^{k}_{ {p}^{k},1}}{\p \hat{v}_{1} }
		\\
		\frac{\p \hat{\mathbf{v}}^{k}_{{p}^{k},1}}{\p \hat{v}_{1} }
		\end{bmatrix}$. From uniform invertibility of matrix $\mathcal{S}^{k,p^{k}}$ and (\ref{specular_matrix}), we finish the proof.   \\
	\end{proof}
	
	\begin{lemma}\label{zero_poly}
		Assume that $b(z),c(z)$ are continuous-functions of $z \in \R^{n}$ locally. We consider $G(z,s) := b(z)s + c(z)$.  \\
		
		\noindent (i) Assume $\min |b| >0$. Define 
		\begin{equation}\label{psi4}
		\varphi_{1} ( z ) : = \frac{- c(z)}{ b(z)}.
		\end{equation}
		Then $\varphi_{1}(z) \in C^{1}_{t,x,v}$ with $\| \varphi_{1}\|_{C^{1}_{t,x,v}} \leq C(\min |b|  , \| b \|_{C^{1}_{t,x,v}}, \| c \|_{C^{1}_{t,x,v}} )$. Moreover, if $|s| \leq 1$ and $|s- \varphi_{1}(z)|> \delta$, then $|G(z,s)| \gtrsim \min |b| \times \delta$.

		\noindent (ii) Assume $\min |c| >0$. Define 
		\begin{equation}\label{psi5}
		\varphi_{2}(z) : =  \mathbf{1}_{|b(z) | > \frac{\min |c|}{4}}  \frac{- c(z)}{ b(z)}.
		\end{equation}
		Then $\varphi_{2} (z) \in C^{1}_{t,x,v} $ with $\| \varphi_{2}\|_{C^{1}_{t,x,v}} \leq C(\min |b|  , \| b \|_{C^{1}_{t,x,v}}, \| c \|_{C^{1}_{t,x,v}})$. Moreover, if $|s| \leq1$ and $|s- \varphi_{5} (z)|>\delta$, then $|G(z,s)|\geq \min \big\{ \frac{\min |c|}{2},  \frac{\min |c|}{4} \times  \delta \big\}
		$.
		

	\end{lemma}

	\begin{proof}
		Now we consider $(i)$. Clearly $\varphi_{1}$ is $C $ for this case. And
		\[
		|G(z,s)| \geq \min \{| b(z)( \frac{-c(z)}{b(z)} + \delta) + c(z)|,| b(z)( \frac{-c(z)}{b(z)}  \delta) + c(z)|\} \geq \min |b| \times \delta. 
		\]	
		Now we consider $(ii)$. First, if $|b| < \frac{\min |c|}{2}$ then $|\varphi_{2}(z)|\geq \frac{ |c(z)|}{ \min|c|/2 } \geq 2$. Therefore,
		\[
		|G(z,s)| \geq \min \{ |G(z,1)|, |G(z,-1)|\} \geq |c(z)| - |b(z)|
		\geq   \frac{ \min|c|}{2}.
		\]
		Consider the case of $|b| > \frac{\min |c|}{4}$. If $|s- \varphi_{2}(s)|> \delta$ then 
		\begin{eqnarray*}
			|G(z,s)| &\geq& \min \big\{ |  b(z)( \frac{-c(z)}{b(z)} + \delta)+ c(z)  | ,  |  b(z)( \frac{-c(z)}{b(z)} - \delta)+ c(z)  |\big\}\\
			&=& \min |b| \times \delta \geq \frac{\min|c|}{2} \times \delta.
		\end{eqnarray*} 
	\end{proof}
	
	\begin{lemma} \label{lemma rank 2}
		Fix $k\in \mathbb{N}$ with $ t^{k}\geq t-1$. Assume $\Omega$ is $C^{2}$ and (\ref{eta}). Let $t^{0}  \geq 0$, $\underline{x}^{0} \in \bar{\O}$, $\underline{v}^{0} \in\R^{2}$, and assume
		\begin{equation}\begin{split}\label{con_v0}
		\frac{1}{N} \leq |\underline{v}^{0}| \leq N,  \ 
		\frac{1}{N} \leq |v^{0}_{3}|, \ |\V^{i}_{p^{i},3}| > \delta_{2} > 0, \ \ \forall 1\leq i \leq k
		,\end{split}
		\end{equation} 
		and (\ref{geo e1}) in Lemma~\ref{nonzero_sub}, where $(\underline{x}^{1}, \underline{v}^{1}) =\big(\underline{x}^{1}(t^{0}, \underline{x}^{0}, \underline{v}^{0}), \underline{v}^{1}(t^{0}, \underline{x}^{0}, \underline{v}^{0})\big)$. Then there exists $\e>0$ and $C^{1}_{t, \underline{x}, \underline{v}}$-functions $\psi^{k}_{1}, \ \psi^{k}_{2}  : B_{\varepsilon}(t,\underline{x},\underline{v})\rightarrow \R$ with $ \max_{i=1,2} \|\psi^{k}_{i} \|_{C^{1}_{t,\underline{x},\underline{v}}} \lesssim_{\delta_{2}, \O, N }
		1$ and there exists a constant $\epsilon_{\delta_{2}, \O, N }
		>0$, such that 
		\begin{equation*}\label{det_nonzero}
		\begin{split}
		& \text{if } \  \min_{i=1,2}|s- \psi_{i}^{k}(t,\underline{x},\underline{v})|>\delta_{*}\\
		&
		\ \ \ \ \  \text{and} \ (s;t,\underline{x},\underline{v}) \in [ \max\{t-1, t^{k+1}\}, \min \{t- \frac{1}{N}, t^{k}\}] \times  B_{\e}(t^{0},\underline{x}^{0}, \underline{v}^{0} ) , \\
		& \text{then} \  \big| \p_{ |\underline{v}|} \underline{X} (s;t, \underline{x}, \underline{v} ) \times \p_{\hat{v}_{1}} \underline{X}(s;t, \underline{x},\underline{v} ) \big| >   \epsilon_{\delta_{2}, \O, N, \delta_{*} }
		.\end{split}\end{equation*}
	\end{lemma}
	It is important that this lower bound $\epsilon_{\delta_{2}, \O, N }$ does not depend on time $ t $.  \\	
	
	\begin{proof}
		\noindent\textit{Step 1. }Fix $k$ with $|t^{k}(t,\underline{x},\underline{v}) -t| \leq 1$. Then we fix the orthonormal basis 	$\big\{ \mathbf{e}^{k}_{0}  ,  \mathbf{e}_{\perp,1}^{k} \big\}$ of (\ref{orthonormal_basis}) with $\underline{x}^{k} = \underline{x}^{k}(t,\underline{x},\underline{v})$, $\underline{v}^{k} = \underline{v}^{k}(t,\underline{x},\underline{v})$. Note that this orthonormal basis $\big\{ \mathbf{e}^{k}_{0}  ,  \mathbf{e}_{\perp,1}^{k} \big\}$ depends on $(t, \underline{x}, \underline{v})$.


		For $t^{k+1} < s < t^{k}$, recall the forms of $\frac{\p \underline{X}(s)}{ \p |\underline{v}|}$ and $\frac{\p \underline{X}(s)}{\p \hat{v}_{j}}$ in (\ref{dX d|v|}) and (\ref{p_X}), where 
		$\underline{X}(s) = \underline{X}(s; t^{k}, \underline{x}^{k}, \underline{v}^{k}).
		$
		Using the \textit{specular basis} (\ref{orthonormal_basis}), we rewrite (\ref{dX d|v|}) and (\ref{p_X}) as 
		\begin{equation*} \label{sub_R}
		\begin{split}
		& \left[\begin{array}{cc}
		\frac{\p \underline{X} (s)}{\p |\underline{v} |}\cdot\mathbf{e}_0^{k} &  \frac{\p \underline{X}(s)}{\p \hat{v}_{1}}\cdot\mathbf{e}_0^{k}    \\
		\frac{\p \underline{X}(s)}{\p |\underline{v} |}\cdot\mathbf{e}_{\perp,1}^{k} &  \frac{\p \underline{X}(s)}{\p \hat{v}_{1}}\cdot\mathbf{e}_{\perp,1}^{k}   \\ 
		\end{array}\right] 
		= 
		\left[\begin{array}{cc}
			-(t-s) &  \frac{\p \underline{X}(s)}{\p \hat{v}_{1}}\cdot\mathbf{e}_0^{k}  \\
			0 &  \frac{\p \underline{X}(s)}{\p \hat{v}_{1}}\cdot\mathbf{e}_{\perp,1}^{k}   \\ 
		\end{array}\right] . \\
		\end{split}
		\end{equation*} 
		
		Note that $(2,2)$ component is written by 
		\begin{equation*} \label{22 compo}
		\begin{split}
			\frac{\p \underline{X}(s)}{\p \hat{v}_{1}}\cdot\mathbf{e}_{\perp,1}^{k} &= \mathcal{R}^{k, p^{k} }_{1 } - (t^{k} - s)\mathcal{R}^{k, p^{k} }_{2 },
		\end{split}
		\end{equation*}
		by (\ref{p_X}) and (\ref{specular_transition_matrix}), 	where $\mathcal{R}^{k, p^{k}}_{i}$ are defined in (\ref{specular_matrix}).	By the direct computation, determinant becomes,
		\begin{equation}\begin{split}\label{comp_X_times_X}
		  \p_{|\underline{v}| } \underline{X}(s) \times \p_{\hat{v}_{1}} X  (s)   
		=   
		- (t-s)    
		\big\{ \mathcal{R}^{k,{p}^{k}}_{1} - (t^{k}-s) |\underline{\mathbf{v}}^{k}_{{p}^{k}}| \mathcal{R}^{k,{p}^{k}}_{2} 
		\big\} .
		\end{split}
		\end{equation}
		Here $\mathcal{R}^{k,p^{k}}_{i}, t^{k}, \underline{\mathbf{v}}^{k}_{p^{k}}$, and $\mathbf{e}^{k}_{\perp, i}$ depend on $(t,\underline{x},\underline{v})$, but not $s$.  \\

		\noindent\textit{Step 2.} Recall Lemma~\ref{nonzero_sub}. From (\ref{con_v0}), we can choose non-zero contants $\delta_{2}$ for a large $N \gg 1$. Applying Lemma~\ref{nonzero_sub} and (\ref{nonzero_sub1}), we conclude that, for some $i \in \{1,2\}$, 
		\begin{equation}\label{lower_R0}
		|\mathcal{R}^{k,p^{k}}_{i}(t, \underline{x},\underline{v})| > \varrho_{\Omega, N, \delta_{2}} > 0.  \\
		\end{equation}
	
		Also, we can claim that $\mathcal{R}^{k,p^{k} }_{i } (t,\underline{x},\underline{v}) \in C^{1}_{t,\underline{x},\underline{v}}$. From (\ref{con_v0}), all bouncings are non-grazing. We use Lemma~\ref{local conti}, (\ref{Est--4}) and (\ref{Est--6}) in Lemma~\ref{global to local}, and (\ref{specular_matrix}) with regularity of $\O$ to derive $\mathcal{R}^{k,p^{k} }_{i } (t,\underline{x},\underline{v}) \in C^{1}_{t,\underline{x},\underline{v}}$. Finally we choose a small constant $\e>0$ such that, for some $i\in \{1,2\}$ satisfying (\ref{lower_R0}),
		\begin{equation}\label{lower_R}
		|\mathcal{R}^{k,p^{k}}_{i}(t ,\underline{x},\underline{v} )|> \frac{\varrho_{\Omega, N, \delta_{2}  }}{2} \ \  \ \text{for } |(t,\underline{x},\underline{v}) - (t^{0}, \underline{x}^{0},\underline{v}^{0})| < \e.
		\end{equation}

		\noindent\textit{Step 3. }  With $N\gg1$, from (\ref{lower_R}), we divide the cases into the follows
		\begin{equation} \label{case1_R}
		|\mathcal{R}^{k,p^{k}}_{1}|> \frac{\varrho_{\Omega, N,   \delta_{2} } }{2}
		 \ \   \text{and} \ \
		|\mathcal{R}^{k,{p}^{k}}_{2}| \geq  \frac{\varrho_{\Omega, N,   \delta_{2}   }  }{2} .
		\end{equation}
		
		\noindent We split the first case (\ref{case1_R}) further into two cases as  
		\begin{equation}\label{case1_R-1}
		 |\mathcal{R}^{k,p^{k}}_{1}|> \frac{\varrho_{\Omega, N, \delta_{2} } }{2}   \ \   \text{and} \ \ |\mathcal{R}^{k,p^{k}}_{2}| < \frac{\varrho_{\Omega, N,  \delta_{2} } }{4N},
		\end{equation}
		and 	
		\begin{equation}\notag\label{case1_R-2}
		|\mathcal{R}^{k,p^{k}}_{1}|> \frac{\varrho_{\Omega, N,  \delta_{2}} }{2}   \ \   \text{and} \ \ |\mathcal{R}^{k,p^{k}}_{2}| \geq \frac{\varrho_{\Omega, N,  \delta_{2} } }{4N}.
		\end{equation}
		
		\noindent Set the other case\begin{equation}\label{case2_R}
		|\mathcal{R}^{k,{p}^{k}}_{2}| \geq  \frac{\varrho_{\Omega, N,  \delta_{2}    }  }{2} .
		\end{equation}
		Then clearly (\ref{case1_R-1}) and (\ref{case2_R}) cover all the cases. \\
		
		\noindent\textit{Step 4. } We consider the case of (\ref{case1_R-1}). 
		Then, from (\ref{comp_X_times_X}), 
		\begin{equation}\label{det_expansion}
		\begin{split}
		|  \p_{|\underline{v}| } \underline{X} (s)\times \p_{\hat{v}_{1}} \underline{X}(s) |
		&\geq  \big| |\underline{v}^{k}| \mathcal{R}^{k,p^{k} }_{2 } (t^{k}-s)   - \mathcal{R}^{k,p^{k} }_{1 }    \big|(t-s)  \\
		&= \big|
		\underbrace{| \underline{v}^{k} | \mathcal{R}_{2,2}^{k,p^{k} } (t-s)  + \big[ - \mathcal{R} ^{k,p^{k} }_{1 } + (t^{k}-t) |\underline{v}^{k}| \mathcal{R} ^{k,p^{k} }_{2} \big] } \big|(t-s).
		\end{split}
		\end{equation}
		
		We define 
		\begin{equation} \label{tilde_s}
		\tilde{s} = t-s,
		\end{equation}
		and set 
		\begin{equation*}\label{abc_R}
		b:= |\underline{v}^{k} | \mathcal{R} ^{k,p^{k}}_{2} \quad\text{and}\quad  c:=- 
		\mathcal{R} ^{k,p^{k}}_{1} + (t^{k}-t) | \underline{v}^{k} | \mathcal{R}_{2}^{k,p^{k}}.
		\end{equation*}
		Note that $\mathcal{R}^{k,p^{k}}_{1},$ $\mathcal{R}^{k,p^{k}}_{2}$, $|\underline{v}^{k}|$, and $t^{k}$ only depend on $(t, \underline{x}, \underline{v})$.
		
		Hence we regard the underbraced term of (\ref{det_expansion}) as an affine function of $\tilde{s}$
		\begin{equation*} \label{affine_tilde_s}
		b(t, \underline{x}, \underline{v}) \tilde{s}+ c(t, \underline{x}, \underline{v}).
		\end{equation*}
		Note that from (\ref{case1_R-1})
		\begin{equation}\notag
		|c(t, \underline{x}, \underline{v})| \geq  \frac{\varrho_{\Omega, N, \delta_{2}}}{2}  - N \frac{ \varrho_{\Omega, N, \delta_{2}
			}
		}{4N}
		\geq \frac{ \varrho_{\Omega, N, \delta_{2}
			}}{4}.
			\end{equation}
			Now we apply $(ii)$ of Lemma~\ref{zero_poly}. With $\varphi_{2} (t,\underline{x},\underline{v})$ in (\ref{psi5}), if $|\tilde{s} - \varphi_{2} (t,\underline{x},\underline{v})|> \delta_{*}$, then $|b(t,\underline{x},\underline{v}) \tilde{s}+ c(t,\underline{x},\underline{v})| \geq \frac{\varrho_{\Omega, N, \delta 
				}}{4} \times \delta_{*}$. We set
				\begin{equation*} \label{psi5_phi5}
				\psi_{2} (t,\underline{x},\underline{v}) = t- \varphi_{2} (t,\underline{x},\underline{v}) .
				\end{equation*}
				From (\ref{tilde_s}), 
				\begin{equation}\label{lower_bound_1}
				\text{if}  \ |s - \psi_{2} (t,\underline{x},\underline{v})|> \delta_{*}, \ \text{then} \  |b(t,\underline{x},\underline{v})  (t-s)+ c(t,\underline{x},\underline{v})| \geq \frac{\varrho_{\Omega, N,  \delta_{2}
					}}{4} \times \delta_{*}.
					\end{equation}
					
					%
					%
					%
					%
					

					\vspace{4pt}
					
					Now we consider the case of (\ref{case2_R}). From (\ref{comp_X_times_X}),
					\begin{equation}\label{det_expansion2}
					\begin{split}
					|\p_{|\underline{v}|} \underline{X}(s) \times \p_{\hat{v}_{1}}X(s)|
					&\geq \big| 
					| \underline{v}^{k} | \mathcal{R}^{k,{p}^{k}}_{2} (t-s) + \big[ - \mathcal{R}^{k,{p}^{k}}_{1}
					+ (t^{k}-t)| \underline{v}^{k} | \mathcal{R}^{k,{p}^{k}}_{2} \big]
					\big| (t-s) .
					\end{split}
					\end{equation} 
					We set $\tilde{s}$ as (\ref{tilde_s}) and
					\begin{equation}\label{abc_R2} 
					 b:= |\underline{v}^{k} | \mathcal{R}^{k,p^{k}}_{2} \quad\text{and}\quad  c:=- 
					\mathcal{R}^{k,p^{k}}_{1} + (t^{k}-t) | \underline{v}^{k} | \mathcal{R}_{2}^{k,p^{k}}.
					\end{equation}

					From (\ref{case2_R}) and (\ref{abc_R2})
					\begin{equation}\notag
					|b(t,\underline{x},\underline{v})| \geq   \frac{\varrho_{\Omega, N, \delta_{2} }}{8N^{2}}.
					\end{equation}
					We apply $(i)$ of Lemma~\ref{zero_poly} to this case: With $\varphi_{1}(t,\underline{x},\underline{v})$ in (\ref{psi4}), we set
					\begin{equation*}\label{psi4_phi4}
					\psi_{4} (t,\underline{x},\underline{v}) = t - \varphi_{1} (t,\underline{x},\underline{v}),
					\end{equation*}	
					and 
					\begin{equation}\label{lower_bound_2}
					\text{if}  \ |s - \psi_{1} (t,\underline{x},\underline{v})|> \delta_{*}, \ \text{then} \  |b(t,\underline{x},\underline{v})  (t-s)+ c(t,\underline{x},\underline{v})| \gtrsim  \frac{\varrho_{\Omega, N,  \delta_{2} }}{8N^{2}}
					\times
					\delta_{*}
					.
					\end{equation}
					Finally, from (\ref{lower_bound_1}), (\ref{det_expansion}), (\ref{lower_bound_2}), and (\ref{det_expansion2}), we conclude the proof of Lemma~\ref{lemma rank 2}.  \\ 
				\end{proof}

	Now we return to three-dimensional cylindrical domain $U := \O \times (0, H) \subset \mathbb{R}^{3}$. We state a theorem about uniform positivity of determinant of $\frac{dX}{dv}$.  \\
	 
	\begin{proposition}\label{prop_full_rank} 
		Let $t \in [T,T+1]$,  
		\[
		(x,v) = (x,\underline{v}, v_{2}) \in U\times\VN \times \{v_{2}\in \mathbb{R} : \frac{1
		}{N} \leq v_{2} \leq N \}.  
		\]
	 	Recall $\varepsilon, \delta$ in Lemma~\ref{G_C unif}.  For each $i=1,2, \cdots, l_{G}$, there exists $\delta_{2}>0$ and $C^{1}_{t, \underline{x}, \underline{v} }$-function $\psi^{\ell_{0}, \vec{\ell}, i, k}$ for uniform bound $k \leq C_{\varepsilon,N}$, where $\psi^{\ell_{0}, \vec{\ell}, i, k}$ is defined locally around $(T+ \delta_{2} \ell_{0}, X(T+ \delta_{2} \ell_{0};t,x,v), (\delta_{2} \vec{\ell}, u_{2}))$ with $(\ell_{0},\vec{\ell}) = (\ell_{0}, \ell_{1}, \ell_{3}) \in \{ 0,1, \cdots, \lfloor\frac{1}{\delta_{2}}\rfloor+1\} \times \{- \lfloor\frac{N}{\delta_{2}}\rfloor-1, \cdots, 0 , \cdots, \lfloor\frac{N}{\delta_{2}}\rfloor+1 \}^{2}$ and $ \| \psi^{\ell_{0}, \vec{\ell}, i, k}  \|_{C^{1}_{t, \underline{x}, \underline{v}}} \leq C_{N,\Omega,\delta,\delta_{2}} < \infty$.  \\
		
		For $( \underline{X}(s;t, \underline{x}, \underline{v}), \underline{u} )\in \{cl(\O)\times \VN \} \backslash \mathfrak{IB}$, if
		\begin{equation} \label{non degene}
			u_{3} \geq \frac{1}{N},
		\end{equation}
		\begin{equation} \label{geo e1 u}
			\Big| \frac{ \p_{1} \underline{\eta}_{ {p}^{1}}  }{ \sqrt{g_{{p}^{1}, 11} } } \big\vert_{x^{1}(\underline{X}(s;t, \underline{x}, \underline{v}), \underline{u})} \cdot e_{1} \Big| > \frac{1}{N} > 0,
		\end{equation} 
		\begin{equation} \label{SG}
			\underline{X}(s;t, \underline{x}, \underline{v}) \notin \bigcup_{j=1}^{l_{sg}} B(y_{j}^{C}, \varepsilon),\quad\text{sticky grazing set defined in Lemma~\ref{G_C unif}},
		\end{equation}
		\begin{equation}\label{Xs_OVi}
		( \underline{X}(s;t, \underline{x}, \underline{v}), \underline{u} ) \in B(x^{C}_{i}, r^{C}_{i}) \times \VN \backslash \mathcal{O}^{C}_{i}
		\ \ \text{for some} \ \ i=1,2,\cdots, l_{G},
		\end{equation}
		
		\begin{equation}\label{su_ell0}
		(s, \underline{u}) \in [T+ (\ell_{0} -1) \delta_{2},T+ (\ell_{0} +1) \delta_{2} ] \times B(\delta_{2} \vec{\ell}, 2\delta_{2}),
		\end{equation}
		
		\begin{equation}\label{s prime}
		|s - s^{\prime}| \geq \delta_{2},
		\end{equation}
		
		\begin{equation}\label{sprime_k}
		\begin{split}
		s^{\prime} &\in \big[ t^{k+1} (T+ \delta_{2} \ell_{0}; \underline{X}(T+ \delta_{2} \ell_{0} ;t, \underline{x}, \underline{v}), \delta_{2} \vec{\ell} )+\frac{1}{N} \\
		&\quad\quad\quad\quad , t^{k } (T+ \delta_{2} \ell_{0}; \underline{X}(T+ \delta_{2} \ell_{0} ;t, \underline{x}, \underline{v}), \delta_{2} \vec{\ell} )-\frac{1}{N}\big], 
		\end{split}
		\end{equation}
		and 
		\begin{equation}\label{sprime_psi}
		|s^{\prime} -   \psi^{\ell_{0}, \vec{\ell}, i, k}  (T+ \delta_{2} \ell_{0} , \underline{X}( T+ \delta_{2}\ell_{0} ; t, \underline{x}, \underline{v}), \delta_{2} \vec{\ell} ) | > N^{2} (1 + \| \psi^{\ell_{0}, \vec{\ell}, i, k} \|_{C^{1}_{t, \underline{x}, \underline{v}}} ) \delta_{2},
		\end{equation}
		then 
		\begin{equation} \label{lower_1}
			\det\Big( \frac{\p X(s^{\prime}; s, X(s;t,x,v), u) }{ \p u} \Big) > \epsilon^{\prime}_{\O,N,\delta,\delta_{2}} > 0,
		\end{equation}
		where $B(x^{C}_{i}, r^{C}_{i}) \times \VN\backslash \mathcal{O}^{C}_{i}$ was constructed in Lemma~\ref{G_C unif}. Also note that $\epsilon^{\prime}_{\O,N,\delta,\delta_{2}}$ does not depend on $T, t,x,v$. 
	\end{proposition}
	\begin{proof}
		\noindent \textit{Step 1.} First we extend two-dimensional analysis into three dimension case. For $v_{2}$ direction, dynamics is very simple, i.e.
		\[
			X_{2}(s;t,x,v) = x_{2} - (t-s) v_{2},
		\] 
		so we have
		\[
			\frac{d X_{2}}{d v_{2}} = - (t-s).\quad  \\
		\]
		Note that it is obvious that $v_{2}$ directional dynamics is independent to two-dimensional trajectory which is projected on cross section $\O$, because of cylindrical domain with the specular boundary condition. \\
	
		\noindent \textit{Step 2.} Fix $t \in [T,T+1], \ (x,v) \in \O\times\VN$ and assume $( \underline{X}(s;t, \underline{x}, \underline{v}), \underline{u} )\in \{cl(\O)\times \VN \} \backslash \mathfrak{IB}$. Assume that $s\in [T,t]$,
		\begin{equation*} 
			\underline{X}(s;t, \underline{x}, \underline{v}) \notin \bigcup_{j=1}^{l_{sg}} B(y_{j}^{C}, \varepsilon)
		\quad \text{and} \quad 
		( \underline{X}(s;t, \underline{x}, \underline{v}), \underline{u} ) \in B(x^{C}_{i}, r^{C}_{i}) \times \VN \backslash \mathcal{O}^{C}_{i},
		\end{equation*}
		for some $i = 1,\cdots, l_{G}$. Due to Lemma~\ref{G_C unif}, $( \underline{X}(s^{\prime}; s, \underline{X}(s;t, \underline{x}, \underline{v}), \underline{u}), \underline{V}(s^{\prime}; s, \underline{X}(s;t, \underline{x}, \underline{v}), \underline{u}) )$ is well-defined for all $s^{\prime} \in [T,s]$ and \[
			|\mathbf{n}(x^{k}(s, \underline{X}(s;t, \underline{x}, \underline{v}), \underline{u})) \cdot v^{k }(s, \underline{X}(s;t, \underline{x}, \underline{v}), \underline{u})| > \delta,
		\] 
		for all $k$ with $|t-t^{k}(s, \underline{X}(s;t, \underline{x}, \underline{v}), \underline{u})| \leq 1$.  \\

		From $\underline{X}(s;t, \underline{x}, \underline{v}) = \underline{X}(\bar{s};t, \underline{x}, \underline{v}) + \int^{s}_{\bar{s}} \underline{V}(\tau; t, \underline{x}, \underline{v}) \dd \tau$,
		\begin{equation}\begin{split}\label{diff_psi_k}
		&|\psi^{k} (s,\underline{X}(s;t, \underline{x}, \underline{v}), \underline{u}) - \psi^{k} (\bar{s}, \underline{X}(\bar{s};t, \underline{x}, \underline{v}), \underline{\bar{u}} )| \\
		&\leq   \   \|  \psi^{k} \|_{C^{1}_{t,\underline{x}, \underline{v}}}  \{ |s- \bar{s}| + |\underline{X}(s;t, \underline{x}, \underline{v})- \underline{X}(\bar{s};t, \underline{x}, \underline{v})| + |\underline{u}- \underline{\bar{u}} | \}  \\
		&\leq \ \| \psi^{k} \|_{C^{1}_{t,\underline{x}, \underline{v}}}  \{ |s- \bar{s}| + 
		(1+ N)
		| \underline{u} - \underline{\bar{u}} | \}.\end{split}
		\end{equation} 
		For $0<\delta_{2}\ll1$ we split 
		\begin{equation}\notag
		\begin{split}
		[T,T+1] &= \bigcup_{\ell_{0}=0}^{ [\delta_{2}^{-1}]+1 } \big[T+ (\ell_{0}-1) \delta_{2}, T+  (\ell_{0}+1) \delta_{2}  \big] , \\  
		\VN \backslash \mathcal{O}^{C}_{i} &=  \bigcup_{|\ell_{i}|=0}^{ [N /\delta_{2}^{-2}]+1 }
		B\big( (\ell_{1} \delta_{2}, \ell_{3} \delta_{2} ), 2\delta_{2}\big)
		\cap  \  \VN \backslash \mathcal{O}^{C}_{i}. 
		\end{split}
		\end{equation}
		From (\ref{diff_psi_k}), if 
		$$
		(s, \underline{u}) \in \big[T+ (\ell_{0}-1) \delta_{2}, T+  (\ell_{0}+1) \delta_{2}  \big] \times \{ B\big( (\ell_{1} \delta_{2}, \ell_{3} \delta_{2} ), 2\delta_{2}\big)
		\cap  \  \VN \backslash \mathcal{O}^{C}_{i} \} ,
		$$ 
		then 
		\begin{eqnarray*}
			&&	|\psi^{k} ( T+ \ell_{0} \delta, \underline{X}(T+ \ell_{0}  \delta;t, \underline{x}, \underline{v}), (\ell_{1} \delta,  \ell_{3} \delta)  ) - \psi^{k} (s, \underline{X}(s;t, \underline{x}, \underline{v}), \underline{u}) | \\
			&&\leq
			\| \psi^{k} \|_{C^{1}_{t, \underline{x}, \underline{v}}} (2+ N) \delta_{2}.
		\end{eqnarray*}
		Therefore, if (\ref{sprime_psi}) holds,
		\begin{eqnarray*}
		&&	|s^{\prime} - \psi^{k} (s, \underline{X}(s;t, \underline{x}, \underline{v}), \underline{u})| \nonumber\\
		&\geq& |s^{\prime} - \psi^{k} ( T+ \ell_{0} \delta, \underline{X}(T+ \ell_{0}  \delta;t, \underline{x}, \underline{v}), (\ell_{1} \delta, \ell_{3} \delta)  )|\notag\\
		&&
		- |\psi^{k} ( T+ \ell_{0} \delta, \underline{X}(T+ \ell_{0}  \delta;t, \underline{x}, \underline{v}), (\ell_{1} \delta, \ell_{3} \delta)  ) - \psi^{k} (s, \underline{X}(s;t,\underline{x}, \underline{v}), \underline{u}) |
		\label{sprime-psi_lower}
		\\ 
		&\gtrsim&(N^{2}- N)\| \psi^{k} \|_{C^{1}_{t, \underline{x}, \underline{v}}} \delta_{2} \gtrsim_{N}  \| \psi^{k} \|_{C^{1}_{t, \underline{x}, \underline{v}}} \delta_{2}. \notag
		\end{eqnarray*}

		\noindent \textit{Step 3.} Consider the three-dimensional mapping $u \mapsto X(s^{\prime}; s,X(s;t,x,v),u)$. Note that from Lemma~\ref{G_C unif} we verify the condition of Lemma~\ref{lemma rank 2}. From Lemma~\ref{lemma rank 2} and 
		\ref{uniformbound}, we construct $C^{1}_{t,\underline{x}, \underline{v}}$-function $\psi^{k}: B_{\varepsilon} (s, \underline{X}(s;t, \underline{x}, \underline{v}), \underline{u}) \rightarrow \R$ for uniform bound $k \leq C_{\varepsilon,N}$ such that if 	$|s^{\prime} - \psi^{k}(s, \underline{X}(s;t, \underline{x}, \underline{v}), \underline{u})| \gtrsim_{N,\O,\delta} \delta_{2}$, then 
		\begin{equation*}
		\begin{split}
			& \big| \det\Big( \frac{\p X(s^{\prime}; s, X(s;t,x,v), u) }{ \p u} \Big) \big| \\
			&=  \big| \frac{d X_{2}}{d v_{2}} \big| \Big| |\p_{|u|} \underline{X}(s^{\prime}; s, \underline{X}(s;t, \underline{x}, \underline{v}), \underline{u}) \times \p_{\hat{u}_{1}} \underline{X}(s^{\prime}; s, \underline{X}(s;t, \underline{x}, \underline{v}), \underline{u}) \Big|  \\
			&> |s - s^{\prime}| \ \epsilon_{\O, N, \delta, \delta_{2} } >  \epsilon^{\prime}_{\O,N,\delta,\delta_{2}} > 0.
		\end{split}
		\end{equation*}
	\end{proof}

	Now we study $L^{\infty}$ estimate via trajectory and Duhamel's principle.
	\begin{lemma}
		Let $f$ solves linear boltzmann equation (\ref{lin eq}). For $h :=  wf$ with $w=(1+|v|)^{\b}, \ \b > 5/2$, we have the following estimate.
		\begin{equation*} 
		\begin{split}
		\|h(t)\|_\infty & \lesssim \ e^{-\nu_{0} t}\|h(0)\|_\infty +  \int_{0}^{t} \|f (s)\|_2 \dd s  .
		\end{split}
		\end{equation*}	
	\end{lemma}
	\begin{proof}
		\noindent Since $L=\nu(v) - K$,
		\begin{equation*} \label{iteration scheme}
		\begin{split}
		&\p_t f  + v\cdot\nabla  f + \nu f   =
		Kf .
		\end{split}
		\end{equation*}
		For $h := w f$
		\begin{equation*} \label{equation_h}
		\begin{split}
		&  \p_t h + v\cdot\nabla_x h + \nu h =  K_w h ,\quad  K_{w}h := wK(\frac{h}{w}).
		\end{split}
		\end{equation*}
		
		\noindent We define,
		\begin{equation*}\label{E_G}\begin{split}
		E(v,t,x) &:= \exp \Big\{ -\int_{s}^{t} \nu(V(\tau)) \Big\}.  \\
		\end{split}\end{equation*}
		Along the trajectory,
		\begin{equation}\begin{split}\notag
		&  \frac{\dd}{\dd s} \Big( E(v,t,s) h (s,X(s;t,x,v),V(s;t,x,v)) \Big)\\
		&=  {E(v,t,s) } \big[ K_w h  \big](s,X(s;t,x,v),V(s;t,x,v)).\end{split}
		\end{equation}
		By integrating from $0$ to $t$, we obtain
		\begin{equation} \begin{split}\label{Duhamel_once}
		h (t,x,v) =& E(v,t,0) h (0, X(0), V(0)) \\ 
		&  + \int^{t}_{0} E(v,t,s) \int_{\R^{3}} k_{w}(u,V(s)) h(s,  X(s;t,x,v),u)   \dd u  \dd s .\end{split}
		\end{equation}
		Recall the standard estimates (see Lemma 4 and Lemma 5 in \cite{GKTT1})
		\begin{equation*}\label{est_kw}
		\int_{\R^{3}} |k_{w} (v,u)| \dd u \leq C_{K}\langle v\rangle^{-1}. 
		\end{equation*}
		
		We apply Duhamel's formula (\ref{Duhamel_once}) \textit{two times}, for sufficiently small $0< \bar{\delta}\ll 1$, and cut a part of domain where change of variable does not work. Especially, we use Lemma~\ref{G_C unif} and split sticky grazing set.  
		
		\begin{equation} \label{full expan}
		\begin{split}
			h (t,x,v) 
			&= E(v,t,0) h (0) + \int^{t}_{0} E(v,t,s) \int_{u} k_{w}(u,v) h(s,  X(s),u)   \dd u  \dd s  \\ 
			&\leq E(v,t,0) h(0) + \int^{t}_{0} E(v,t,s) \int_{u} k_{w}(u,v) E(u,s,0) h(0)   \\
			&\quad + |(\mathcal{E}_{1})| + |(\mathcal{E}_{2})| + |(\mathcal{E}_{3})| + |(\mathcal{E}_{4})| + |(\mathcal{E}_{5})| , \\
		\end{split}	 
		\end{equation}
		where
		\begin{equation} \label{def Ek}
		\begin{split}	
			(\mathcal{E}_{k}) &:= \int^{t}_{0}  E(v,t,s) \int_{u} k_{w}(u,v) 
			\int^{s}_{0} E(u,s,s^{\prime}) \int_{u^{\prime}}
			k_{w}(u^\prime,u) h(s ^{\prime} , 
			X(s^{\prime}) ,u^{\prime} ) \ \mathbf{1}_{E_{k}}(X(s), u),\quad k=1,2,3,4,5.  \\	
		\end{split}	 
		\end{equation}
		Note that we abbreviated notations
		\begin{equation*}
		\begin{split}
		X(s) &:= X(s;t,x,v) ,\quad X(s^{\prime}) := X^{\prime}(s^{\prime};s,X(s;t,x,v),u),  \\
		\end{split}
		\end{equation*}
		and $E_{k}$ in characteristic functions in (\ref{def Ek}) are defined as
		\begin{equation*}
		\begin{split}
			E_{1} &:= \big\{ (X(s),u) \in\mathbb{R}^{3}\times\mathbb{R}^{3} : \underline{u}\in\mathbb{R}^{2}\backslash \VN \ \ \text{or} \ \ |u_{2}|\in\mathbb{R}\backslash [\frac{1}{N},N] \ \big\},  \\
			E_{2} &:= \big\{ (X(s),u) \in\mathbb{R}^{3}\times\mathbb{R}^{3} : 	
			(\underline{u}, u_{2}) \in \VN\times[\frac{1}{N}, N], \ (\underline{X}(s),\underline{u}) \in \mathfrak{IB} \ \big\},  \\
			E_{3} &:= \big\{ (X(s),u) \in\mathbb{R}^{3}\times\mathbb{R}^{3} : 	
			(\underline{u}, u_{2}) \in \VN\times[\frac{1}{N}, N], \ (\underline{X}(s),\underline{u}) \in \{cl(\O)\times\VN \}\backslash\mathfrak{IB}, \\
			&\quad\quad\quad\quad \underline{X}(s)\in \bigcup_{j=1}^{l_{sg}} B( y_{j}^{C} , \varepsilon ) \ \big\},  \\
			E_{4} &:= \big\{ (X(s),u) \in\mathbb{R}^{3}\times\mathbb{R}^{3} : 	
			(\underline{u}, u_{2}) \in \VN\times[\frac{1}{N}, N], \ (\underline{X}(s),\underline{u}) \in \{cl(\O)\times\VN \}\backslash\mathfrak{IB}, \\ &\quad\quad\quad\quad (\underline{X}(s), \underline{u}) \in \Big\{ \bigcup_{i=1}^{l_G} \ B(x^{C}_{i}, r^{C}_{i})  \times  \mathcal{O}^{C}_{i} \Big\} \backslash \Big\{ \bigcup_{j=1}^{l_{sg}} B( y_{j}^{C} , \varepsilon ) \times \VN\Big\} \ \big\},  \\
		\end{split}
		\end{equation*}
		and
		\begin{equation} \label{E5}
		\begin{split}	
			E_{5} &:= \big\{ (X(s),u) \in\mathbb{R}^{3}\times\mathbb{R}^{3} : 	
			(\underline{u}, u_{2}) \in \VN\times[\frac{1}{N}, N], \ (\underline{X}(s),\underline{u}) \in \{cl(\O)\times\VN \}\backslash\mathfrak{IB}, \\ &\quad\quad\quad\quad (\underline{X}(s), \underline{u}) \in \Big\{ B(x^{C}_{i}, r^{C}_{i})  \times  \{\VN\backslash\mathcal{O}^{C}_{i} \} \Big\} \backslash \Big\{ \bigcup_{j=1}^{l_{sg}} B( y_{j}^{C} , \varepsilon ) \times \VN\Big\} \ \text{for some} \ i=1,\cdots, l_{sg} \ \big\}.  \\
		\end{split}
		\end{equation}

		\noindent Also note that 
		\[
			(\mathfrak{G})_{\varepsilon} := \Big\{ \bigcup_{i=1}^{l_G} \ B(x^{C}_{i},r^{C}_{i})  \times  \mathcal{O}^{C}_{i} \Big\} \ \bigcup \ \{ \bigcup_{j=1}^{l_{sg}} B( y_{j}^{C} , \varepsilon ) \times \VN \}
		\]
		was defined in Lemma~\ref{G_C unif} and we have
		
		\begin{equation*} \label{exponent bound time dept}
		E(v,t,s) \leq e^{ - \nu(v) (t-s) }.
		\end{equation*}
		On the RHS of (\ref{full expan}), every terms except $(E_{1})$, $(E_{2})$, $(E_{3})$, $(E_{4})$, and $(E_{5})$, are controlled by 
		\begin{equation} \label{non main}
		C e^{-\nu_{0} t}\|h(0)\|_\infty 
		\end{equation}
		
		

		We claim smallness of $(E_{1}) \sim (E_{4})$. From $\int_{u} \mathbf{1}_{ \{\underline{u}\in\mathbb{R}^{2}\backslash \VN \ \text{or} \ |u_{2}|\in\mathbb{R}\backslash [\frac{1}{N},N] \} } (u) \sqrt{\mu} du = O(\frac{1}{N})$,
		\begin{equation} \label{E1}
		\begin{split}
			(E_{1}) &\leq O(\frac{1}{N}) \sup_{0\leq s\leq t} \|h(s)\|_{\infty}. 
		\end{split}
		\end{equation}
		From Lemma~\ref{infinite_bounces_set}, $\mathfrak{m}_{2}(\mathcal{O}_{i}^{IB}) \lesssim \varepsilon$ for $1\leq i \leq l_{IB}$. Therefore,
		\begin{equation} \label{E2}
		\begin{split}
			(E_{2}) &\leq O(\varepsilon) \sup_{0\leq s\leq t} \|h(s)\|_{\infty}. 
		\end{split}
		\end{equation}
		For $(E_{3})$, we also have similar estimate because
		\begin{equation} \label{E3}
		\begin{split}
			(E_{3}) &\leq \int_{0}^{t} \dd s \mathbf{1}_{ \underline{X}(s)\in \bigcup_{j=1}^{l_{sg}} B( y_{j}^{C} , \varepsilon ) } (s) \|h(s)\|_{\infty} \\
			&\leq C\frac{\varepsilon}{1/N} \sup_{0\leq s\leq t} \|h(s)\|_{\infty},\quad \text{since}\quad \underline{v} \geq \frac{2}{N},  \\
			&\leq C\varepsilon N \sup_{0\leq s\leq t} \|h(s)\|_{\infty}  \leq O(\frac{1}{N}) \sup_{0\leq s\leq t} \|h(s)\|_{\infty}.
		\end{split}
		\end{equation}
		For estimate for $(E_{4})$, since $\mathfrak{m}_{2}(\mathcal{O}^{C}_{i}) < \varepsilon$ from Lemma~\ref{G_C unif},
		\begin{equation} \label{E4}
		\begin{split}
		(E_{4}) &\leq O(\varepsilon) \sup_{0\leq s\leq t} \|h(s)\|_{\infty}. 
		\end{split}
		\end{equation}
		For $(E_{5})$, we choose $m({N})$ so that 
		\begin{equation*} \label{opeator k split}
		k_{w,m}(u,v) := \mathbf{1}_{\{ |u-v|\geq\frac{1}{m}, \ |u|\leq m \}} k_{w}(u,v) ,
		\end{equation*}
		satisfies $\int_{\mathbb{R}^3} |k_{w,m}(u,v)-k_{w}(u,v)| \ \dd u \leq \frac{1}{{N}}$ for sufficiently large ${N}\geq 1$. Then, by splitting $k_w$,
		
		\begin{equation} \label{* decomp} \begin{split}
		(E_{5}) &\leq \underline{ \int_0^t \int_0^s e^{- \nu(v) (t-s^{\prime})} \int_u k_{w,m}(u,v)  \int_{u^{\prime}} k_{w,m}(u^{\prime},u) h (s^{\prime},X^{\prime}(s^{\prime}),u^{\prime}) \ \mathbf{1}_{E_{5}}(X(s), u) \ \dd u^{\prime} \dd u \dd s^{\prime} \dd s  }_{(**)}  \\
		& \quad\quad + O_{\O}(\frac{1}{N}) \sup_{0\leq s \leq t} \|h (s)\|_\infty    \\
		\end{split} \end{equation}
		
		\noindent 
		We define following sets for fixed $n,\vec{n},i,k,$, where Proposition~\ref{prop_full_rank} does not work.
		\begin{equation} \label{R16}
		\begin{split}
		R_1 &:= \{ u \ \vert \ \underline{u} \notin B( \vec{n}\delta , 2\delta ) \cap \{\R^2\backslash \mathcal{O}^{C}_{i_{s}} \}  \} ,    \\
		R_2 &:= \{ s^{\prime} \ \vert \ |s-s^{\prime}| \leq \delta \} , \\
		R_3 &:= \{ s^{\prime} \ \vert \ \max_{i=1,2}| s^{\prime} -  \psi_{1}^{n,\vec{n},i,k} ( n \delta, \underline{X}(n \delta;t, \underline{x}, \underline{v}), (\vec{n}\delta, u_{2}) ) |\lesssim_{N} \delta \|\psi_{1} \|_{C^{1}_{t, \underline{x}, \underline{v}}} \} ,  \\
		R_4 &:= \{ s^{\prime} \ \vert \ | s^{\prime} -  t^{k} ( n \delta, \underline{X}(n \delta;t, \underline{x}, \underline{v}), (\vec{n}\delta, u_{2}) ) |\lesssim_{N} \delta \|\psi_{1} \|_{C^{1}_{t, \underline{x}, \underline{v}}} \} ,  \\
		R_5 &:= \{ u \ \vert \ |u_{3}| \leq \frac{1}{N} \} ,  \\
		R_6 &:= \{ \underline{u}\in\mathbb{R}^{2} \ \vert \ \Big| \frac{ \p_{1} \underline{\eta}_{ {p}^{1}}  }{ \sqrt{g_{{p}^{1}, 11} } } \big\vert_{x^{1}(\underline{X}(s;t, \underline{x}, \underline{v}), \underline{u})} \cdot e_{1} \Big| \leq \frac{1}{N} \} .  \\
		\end{split}
		\end{equation}
		Using (\ref{R16}), we write $(**)$ as
		\begin{equation} \label{**}
		\begin{split}	
		(**) &= \underline{ \sum_{n =0}^{[t/\delta]+1} \sum_{|\vec{n}| \leq N}  \sum_{k}^{ C_{\varepsilon,N} }  \int^{(n+1) \delta}_{ (n-1) \delta} \ \int^{t^{k}}_{t^{k+1}}  e^{- \nu(v)(t-s^{\prime})} } \\ 
		&\quad\quad\quad\quad \times \underline{ \int_{ |u|\leq N, |u^{\prime}|\leq N } k_{w,m}(u,v) k_{w,m}(u^{\prime},u) \ | h( s^{\prime},X(s^{\prime}), u^{\prime} ) | } \\ &\quad\quad\quad\quad \times \underline{ \mathbf{1}_{R_{1}^{c} \cap R_{2}^{c} \cap R_{3}^{c} \cap R_{4}^{c} \cap R_{5}^{c} \cap R_{6}^{c} } \mathbf{1}_{E_{5}}(X(s), u)  }_{\text{(MAIN)}}  +  R,  \\
		\end{split}	
		\end{equation}	
		where $R$ corresponds to where $(u, s^{\prime})$ is in one of $R_1\sim R_6$. We replace $\mathbf{1}_{R_{1}^{c} \cap R_{2}^{c} \cap R_{3}^{c} \cap R_{4}^{c} \cap R_{5}^{c} \cap R_{6}^{c}}$ into $\mathbf{1}_{R_{1} \cup R_{2} \cup R_{3} \cup R_{4} \cup R_{5} \cup R_{6}}$ in (MAIN).  For $R$, we have the following smallness estimate:
		
		\begin{equation} \label{BR}
		\begin{split}
		R &\leq \int_0^t \int_0^s e^{- \frac{1}{2} \nu(v)(t-s^{\prime})} \int_{|u|\leq N} k_{w,m}(u,v) \int_{|u^{\prime}|\leq N} k_{w,m}(u^{\prime},u) h (s^{ \prime},X^{ \prime}(s^{ \prime}),u^{ \prime})  \times \mathbf{1}_{ R_1 \cup R_2 \cup R_3 \cup R_4 \cup R_5 \cup R_6 }  \\
		&\leq C_{N} \big(\delta + \varepsilon + O(\frac{1}{N}) \big) \sup_{0\leq s \leq t} \|h(s)\|_\infty,  \\
		\end{split}
		\end{equation}
		by choosing sufficiently small $\delta \ll \frac{1}{N}$. Note that smalless from $R_{1}$ to $R_{5}$ are trivial. For $R_{6}$, we note that by analyticity and boundness of $\O$, there are only finite points $\underline{x}$ such that $\frac{ \p_{1} \underline{\eta}_{ {p}^{1}}  }{ \sqrt{g_{{p}^{1}, 11} } } \bigg\vert_{\underline{x}\in\p\O} \cdot e_{1} = 0$, so $R_{6}$ gives smallness $O(\frac{1}{N})$.  \\
		
		Let us focus on $\text{(MAIN)}$ in (\ref{**}). From (\ref{E5}) and (\ref{R16}), all conditions (\ref{non degene})--(\ref{sprime_psi}) in Proposition~\ref{prop_full_rank} are satisfied and 
		\begin{equation*}
		\begin{split}
		&\exists i_{s} \in \{ 1,2,\cdots, l_{G}\} \quad\text{such that}\quad \underline{X}(s) \in B(x^{C}_{i_{s}}, r^{C}_{i_{s}}).  \\
		\end{split}
		\end{equation*}
		Under the condition of $(u, s^{\prime}) \in R_{1}^{c} \ \cap \ R_{2}^{c} \ \cap R_{3}^{c} \ \cap R_{4}^{c} \ \cap R_{5}^{c} \ \cap R_{6}^{c}  $, indices $n, \vec{n}, i_{s}, k$ are determined so that
		\begin{eqnarray*}
			t \ &\in& \ [  (n-1)\delta,  (n+1)\delta ],   \\
			\underline{X}(s;t, \underline{x}, \underline{v}) \ &\in& \ B(x^{C}_{i_{s}}, r^{C}_{i_{s}}) ,  \\
			\underline{u} \ &\in& \ B(\vec{n}\delta, 2\delta) \cap \{ \VN \backslash\mathcal{O}^{C}_{i_{s}} \} ,   \\
		\end{eqnarray*}
		and (\ref{lower_1}) in Proposition~\ref{prop_full_rank} gives local time-independent lower bound
		\begin{equation*} 
			\Big| \det\Big( \frac{\p X(s^{\prime}) }{ \p u} \Big) \Big| > \epsilon^{\prime}_{\delta} > 0.
		\end{equation*}
		If we choose sufficiently small $\delta$, there exist small $r_{\delta,n,\vec{n},i,k}$ such that there exist one-to-one map $\mathcal{M}$, 
		\begin{eqnarray*}
			\mathcal{M} &:& B(\vec{n}\delta, 2\delta) \cap \{ \VN\backslash\mathcal{O}^{C}_{i_{s}} \}
			\mapsto
			B( \underline{X}(s^{\prime}; s, \underline{X}(s;t, \underline{x}, \underline{v}), \underline{u}), r_{\delta,n,\vec{n},i,k} ).
		\end{eqnarray*} 
		So we perform change of variable for $\text{(MAIN)}$ in (\ref{**}) to obtain
		
		\begin{equation} \label{MAIN}
		\begin{split}	
		&\text{(MAIN)} \\
		&\leq  \sum_{n=0}^{[t/\delta]+1} \sum_{|\vec{n}| \leq N}   \sum_{k}^{ C_{\varepsilon,N} }  \int^{(n+1) \delta}_{ (n-1) \delta} \ \int^{t^{k} }_{t^{k+1} }  e^{-   \nu(v) (t-s^{\prime})}  \\ 
		& \quad\quad \times \int_{u} k_{w,m}(u,v) \int_{u^{\prime}} k_{w,m}(u^{\prime},u) \ \mathbf{1}_{ |u|\leq N, |u^{\prime}|\leq N} \ | h( s^{ \prime},X( s^{ \prime}), u^{ \prime} ) |  \ \mathbf{1}_{E_{5}}(X(s), u) \dd u \dd s^{\prime} \dd s   \\
		&\leq  \sum_{n=0}^{[t/\delta]+1} \sum_{|\vec{n}| \leq N}   \sum_{k}^{ C_{\varepsilon,N} }  \int^{(n+1) \delta}_{ (n-1) \delta} \ \int^{t^{k} }_{t^{k+1} }  e^{-   \nu(v) (t-s^{\prime})}  \\ 
		& \quad\quad \times \int_{u} k_{w,m}(u,v) \mathbf{1}_{ |u|\leq N } \ \dd u  \  \| f( s^{ \prime},X( s^{ \prime}), u ) \|_{ L^{2}_{|u^{\prime}| \leq N} }  \mathbf{1}_{E_{5}}(X(s), u) \dd u \dd s^{\prime} \dd s   \\
		&\leq  \sum_{n=0}^{[t/\delta]+1} \sum_{|\vec{n}| \leq N}   \sum_{k}^{ C_{\varepsilon,N} }  \int^{(n+1) \delta}_{ (n-1) \delta} \ \int^{t^{k} }_{t^{k+1} }  e^{-   \nu(v) (t-s^{\prime})} \Big\{ \int_{|u|\leq N} \| f( s^{ \prime},X( s^{ \prime}), u^{ \prime} ) \|^{2}_{ L^{2}_{|u^{\prime}| \leq N} } \dd u \Big\}^{1/2} 	\\
		&\quad\quad \times \mathbf{1}_{E_{5}}(X(s), u) \dd u \dd s^{\prime} \dd s   \\
		&\leq \sum_{n=0}^{[t/\delta]+1} \sum_{|\vec{n}| \leq N}   \sum_{k}^{ C_{\varepsilon,N} }  \int^{(n+1) \delta}_{ (n-1) \delta} \ \int^{t^{k} }_{t^{k+1} }  e^{-   \nu(v) (t-s^{\prime})} \Big\{ \int_{_{ B( X(s^{\prime}), r_{\delta,n,\vec{n},i,k} ) }} \| f( s^{ \prime},X( s^{ \prime}), u^{ \prime} ) \|^{2}_{ L^{2}_{|u^{\prime}| \leq N} } \frac{1}{ \epsilon^{\prime}_{\delta } } \dd x \Big\}^{1/2} 	\\
		&\leq C \int_{0}^{t} \|f\|_{L^{2}_{x,v}} \dd s.
		\end{split}	
		\end{equation}
		\\
		We collect (\ref{full expan}), (\ref{non main}), (\ref{E1})--(\ref{E4}), (\ref{* decomp}), (\ref{**}), (\ref{BR}), and (\ref{MAIN}) with sufficiently large $N \gg 1$ and small $\varepsilon, \delta \ll \frac{1}{N^{2}}$ to conclude
		\begin{equation} \label{h infty est}
		\begin{split}
			\|h(t)\|_\infty & \lesssim \  e^{-\nu_{0} t}\|h(0)\|_\infty  +  \int_{0}^{t} \|f (s)\|_2 \dd s .
		\end{split}
		\end{equation}
		
	\end{proof}

	\section{$L^{2}$-Coercivity via contradiction method} 
	
	We start with a lemma which was proved in Lemma 5.1 in \cite{KimLee}. 
	\begin{lemma}\label{boundary_interior}Let $g$ be a (distributional) solution to 
		\begin{equation*}\label{eqtn_g}
				\p_{t} g + v\cdot\nabla_{x} g = G.
		\end{equation*}
		Then, for a sufficiently small $\e>0$,
		\begin{equation*} \label{int_ext}
		\begin{split}
		&
		\int^{1-\e}_{\e}\|
		\mathbf{1}_{\dist(x, \p U)<\e^{4} } 
		\mathbf{1}_{|\mathbf{n}(x) \cdot v| > \e} g(t)
		\|_{2}^{2} \dd t
		\lesssim 
		\int^{1}_{0}\| \mathbf{1}_{\dist(x,\p U)> \e^{3}/2 } g(t)\|_{2}^{2} \dd t
		+ \int^{1}_{0} 
		\iint_{ U\times\R^{3}}|gG|.
		\end{split}
		\end{equation*}
	\end{lemma}
	
	\begin{proposition}\label{prop_coercivity}
		Assume that $f$ solves linear Boltzmann equation
		\begin{equation} \label{linearized_eqtn}
		\p_{t}f + v\cdot\nabla f + Lf = 0,
		\end{equation}
		and satisfies the specular reflection BC and (\ref{conserv_F_mass}) for $F= \mu + \sqrt{\mu}f$. Furthermore, for an axis-symmetric domain, we assume (\ref{conserv_F_angular}). Then there exists $C
		> 0$ such that, for all $N \in \mathbb{N}
		$, 
		\begin{equation}\label{coercive}
		\int^{N+1}_{N} \| \mathbf{P} f (t) \|_{2}^{2} \dd t
		\leq C
		\int^{N+1}_{N} \| (\mathbf{I} - \mathbf{P}) f (t) \|_{\nu}^{2} \dd t. 
		\end{equation}
	\end{proposition} 
	\begin{proof}
		We will use contradiction method which is used in \cite{Guo10} and also in \cite{KimLee} with some modification. Instead of full detail, we describe scheme of proof following \cite{KimLee}. \\
		
		\noindent\textit{Step 1.} First, (\ref{linearized_eqtn}) is translation invariant in time, so it suffices to prove coercivity for finite time interval $t \in [0,1]$ and so we claim (\ref{coercive}) for $N=0$. Now assume that Proposition~\ref{prop_coercivity} is wrong. Then, for any $m \gg1$, there exists a solution $f^{m}$ to (\ref{linearized_eqtn}) with specular reflection BC, which solves
		
		\begin{equation}\label{eqtn_fm}
		\p_{t} {f}^{m} + v\cdot \nabla_{x }  {f} ^{m} + L {f}^{m} =0, \ \ \text{for} \ t \in [0,1]
		\end{equation}
		and satisfies
		\begin{equation*}\label{contra_coercivity_1}
		\int^{ 1}_{ 0} \| \mathbf{P} f^{m } (t) \|_{2}^{2} \dd t
		\geq m \int^{ 1}_{0} \| (\mathbf{I} - \mathbf{P}) f^{m } (t) \|_{\nu}^{2} \dd t.
		\end{equation*}
		\noindent Defining normalized form of $f^{m}$ by
		\begin{equation} \label{normal Zm}
		Z^{m} (t,x,v) : = \frac{ f^{m}(t,x,v)}{ \sqrt{\int^{1}_{0} \| \mathbf{P} f^{m} (t) \| _{2}^{2} \dd t  }} .
		\end{equation}
		Then $Z^{m}$ also solves (\ref{eqtn_fm}) with specular BC and 
		\begin{equation} \label{contra_coercivity_Z}
		\frac{1}{m}
		\geq  \int^{ 1}_{0} \| (\mathbf{I} - \mathbf{P}) Z ^{m}(t) \|_{\nu}^{2} \dd t.  \\
		\end{equation}
		
		\noindent\textit{Step 2. }We claim that 
		\begin{equation}\label{claim_Zm}
		\sup_{m}\sup_{0 \leq t \leq 1}
		\| Z^{m} (t) \|_{2}^{2} < \infty.
		\end{equation}
		Since $Z_{m}$ solves (\ref{eqtn_fm}) with specular BC, for $0 \leq t \leq 1$,
		\begin{equation*} \label{sup_Zm}
		\sup_{0 \leq t \leq 1}\| Z^{m}(t) \|_{2}^{2} 
		\leq  \| Z^{m}(0) \|_{2}^{2}
		,
		\end{equation*}
		from the non-negativity of $L$.  Moreover, by integration $\int^{1}_{0}$ and using (\ref{contra_coercivity_Z}) and (\ref{normal Zm}),
		\begin{equation*} \label{bound_Zm}
		\| Z^{m}(0) \|_{2}^{2} 
		\lesssim
		\int^{1}_{0} \| Z^{m} \|_{2}^{2} 
		+ \int^{1}_{0}  \| (\mathbf{I} - \mathbf{P}) Z^{m} \|_{\nu}^{2}
		\lesssim 1+ \frac{1}{m}.
		\end{equation*} 
		Therefore, we proved the claim (\ref{claim_Zm}).  \\
		
		\noindent\textit{Step 3.} Therefore, the sequence $\{Z^{m}\}_{m\gg 1 }$ is uniformly bounded in $\sup_{0 \leq t \leq 1} \| g(t) \|_{\nu}^{2}$. By the weak compactness of $L^{2}$-space, there exists weak limit $Z$ such that 
		\begin{equation*}\label{limit Z}
		Z^{m}\rightharpoonup Z  \ \ \text{in} \ \ 
		L^{\infty} ([0,1]; L^{2}_{\nu}( U\times\R^{3})) \cap 
		L^{2} ([0,1] ; L^{2}_{\nu}(  U\times\R^{3})).
		\end{equation*}
		
		\noindent Therefore, in the sense of distributions, $Z$ solves (\ref{linearized_eqtn}) with the specular BC. See the proof of Proposition of 1.4 in \cite{KimLee} to see that $Z$ also satisfies the specular BC. Moreover, it is easy to check that weak limit $Z$ satisfies conservation laws: 
		\begin{equation}\label{coserv_Z} 
		\iint_{ U\times\R^{3}} Z(t) \sqrt{\mu } = 0, \ \
		\iint_{ U\times\R^{3}} Z(t) \frac{|v|^{2}}{2} \sqrt{\mu }=0, \quad 0\leq t \leq 1.  
		\end{equation}
		In the case of axis-symmetry (\ref{axis-symmetric}),  
		\begin{equation}\label{angular_Z} 
		\iint_{ U \times \R^{3}} \{ (x-x^{0} ) \times \varpi \} \cdot v Z  (t)\sqrt{\mu}=0.
		\end{equation}
		
		On the other hand, since
		\[
		\mathbf{P}Z ^{m} \rightharpoonup \mathbf{P}Z \quad\text{and}\quad (\mathbf{I-P})Z ^{m} \rightarrow 0 \quad\text{in}\quad \int_0^1 \|\cdot\|_{\nu}^2 \dd t, 
		\]
		we know that weak limit $Z$ has only hydrodynamic part, i.e.
		\begin{equation} \label{limit_Z}
		Z(t,x,v) = \{a(t,x) + v\cdot b(x,v) +  \frac{|v|^{2}-3}{2}c(t,x)\}\sqrt{\mu },
		\end{equation}
		and 
		\begin{equation*}\begin{split}\label{bound_Z}
		\int^{1}_{0} \| Z\|_{\nu}^{2}\dd t \leq  \liminf_{m\rightarrow \infty}  \int^{1}_{0} \| Z^{m}\|_{\nu}^{2}\dd t \leq 1+ \frac{1}{m}\rightarrow 1.  \\
		\end{split}
		\end{equation*}

		\noindent\textit{Step 4. Compactness. } For interior compactness, let $\chi_{\e}: cl(U) \rightarrow [0,1]$ be a smooth function such that $\chi_{\e}(x) =1$ if $\dist(x,\p U)> 2\e^{4} $ and $\chi_{\e}(x) =0$ if $\dist(x,\p U)< \e^{4}$. From (\ref{linearized_eqtn}) with $Z^{m}$ ,
		\begin{eqnarray*}
			[\p_{t} + v\cdot \nabla_{x}  ](\chi_{\e}  Z^{m})
			= 
			v\cdot \nabla_{x} \chi_{\e} Z^{m}
			- L (\chi_{\e}Z^{m}).
		\end{eqnarray*} 
		From the standard Average lemma, $\chi_{\e} Z^{m}$ is compact i.e. 
		\begin{equation}\label{compact_int}
		\chi_{\e} Z^{m} \rightarrow \chi_{\e} Z \ \ \text{strongly in } 
		L^{2}([0,1]; L^{2}_{\nu}( U\times\R^{3})).  \\
		\end{equation}
		
		\noindent For near boundary compactness for non-grazing part, we claim that 
		\begin{equation}\label{bdry_m}
		\begin{split}
		&\int^{1-\e}_{\e}	\| \big(Z^{m} (t,x,v)- Z(t,x,v) \big) 
		\mathbf{1}_{
			\dist(x,\p U) < {\e^{4}} 
		}
		\mathbf{1}_{|\mathbf{n}(x) \cdot v|> \e} \|_{2}^{2}\\
		&
		\lesssim \int^{1}_{0}
		\| \big(Z^{m} (t,x,v)- Z(t,x,v) \big)\mathbf{1}_{\dist(x,\p U) > \frac{\e^{3}}{2} } \|_{2}^{2}+ O(\frac{1}{\sqrt{m}}).
		\end{split}
		\end{equation}
		
		\noindent We are looking up the equation of $Z^{m}-Z$. From (\ref{limit_Z}),
		\begin{equation}\label{eqtn_Zm_Z}\begin{split}
		[\p_{t}  + v\cdot \nabla_{x } ] (Z^{m}
		-Z) + L Z^{m} = 0. 
		\end{split}\end{equation} 
		
		\noindent We apply Lemma~\ref{boundary_interior} to (\ref{eqtn_Zm_Z}) by equating $g$ and $G$ with $Z^{m}
		-Z$ and the RHS of (\ref{eqtn_Zm_Z}) respectively. Then
		\begin{eqnarray*}
			&& \int^{1-\e}_{\e}\|
			\mathbf{1}_{\dist(x, \p U)<\e^{4} } 
			\mathbf{1}_{|\mathbf{n}(x) \cdot v| > \e} 
			(Z^{m}-Z)
			(t)
			\|_{2}^{2} \dd t
			\\
			&& \lesssim  
			\int^{1}_{0}\| \mathbf{1}_{\dist(x,\p U)> \e^{3}/2 } (Z^{m}-Z)(t)\|_{2}^{2} \dd t
			+
			\int^{1}_{0} 
			\iint_{ U\times\R^{3}} |Z^{m}-Z|\langle v\rangle
			|(\mathbf{I}- \mathbf{P})Z^{m}|  \\
			&& \lesssim 
			\sqrt{m}\int^{1}_{0} \| (\mathbf{I} - \mathbf{P}) Z^{m} \|_{\nu}^{2} 
			+
			\frac{1}{ \sqrt{m}}
			\int^{1}_{0} \| Z^{m} \|_{\nu}^{2} + \| Z\|_{\nu}^{2}
			.
		\end{eqnarray*}
		By (\ref{claim_Zm}) and (\ref{contra_coercivity_Z}), we conclude (\ref{bdry_m}).

		On the other hand, from (\ref{contra_coercivity_Z}), (\ref{limit_Z}), and (\ref{claim_Zm}),
		\begin{equation}
		\label{grazing_small}
		\begin{split}
		\int^{1-\e}_{\e} \| ( Z^{m}- Z) \mathbf{1}_{|\mathbf{n}(x) \cdot v| \leq \e } \|_{2}^{2} 
		&\leq \   \int_{\e}^{1-\e} \| (\mathbf{I} - \mathbf{P}) Z^{m} \|_{\nu}^{2} + 
		O(\e) \int^{1-\e}_{\e} \| \mathbf{P} Z^{m} \|_{2}^{2} + \| \mathbf{P} Z  \|_{2}^{2}\\
		&\leq \   \frac{1}{m} + 	O(\e).
		\end{split}\end{equation}
		
		\vspace{4pt}
		
		\noindent\textit{Step 6. Strong convergence. }
		For given $\e>0$, we can choose $m\gg_{\e}1$ such that 
		\begin{eqnarray*}
			\int^{1}_{0} \iint_{ U \times \R^{3}}  | 
			Z^{m} - Z |^{2} 
			&\leq&
			\int^{1}_{1-\e} \iint_{ U \times\R^{3}}+ \int^{\e}_{0} \iint_{ U \times\R^{3}} + 
			\int^{1-\e}_{\e} \iint_{ U_{\e} \times \R^{3}}\\
			&&+ \int^{1-\e}_{\e} \iint_{
				\substack{ U \backslash  U_{\e} \times \R^{3}\\
					\cap \ 
					\{ |\mathbf{n}(x) \cdot v|<\e  \ \text{or} \ |v|\geq \e^{-1}\}
				}}
				+ \int^{1-\e}_{\e} \iint_{
					\substack{ U \backslash  U_{\e} \times \R^{3}\\
						\cap \ 
						\{ |\mathbf{n}(x) \cdot v|\geq\e  \ \text{and} \ |v|\leq \e^{-1}\}
					}}\\
					&<&C \e ,
				\end{eqnarray*} 
				where we have used (\ref{claim_Zm}), (\ref{compact_int}), (\ref{bdry_m}), and (\ref{grazing_small}). Therefore, we conclude that $Z^{m} \rightarrow Z$ strongly in $L^{2}([0,1] \times U \times\R^{3})$ and hence 
				\begin{equation}\label{Z=1}
				\int^{1}_{0} \| Z\|_{2}^{2}=1.  \\
				\end{equation}

				\noindent\textit{Step 8.} We claim $Z=0$. Plugging (\ref{limit_Z}) into linearized Boltzmann equation, we get 
				\begin{equation} \label{macro eq}
				\begin{split}
				\p_i c &= 0 ,    \\
				\p_t {c} + \p_i  {b}_i  &= 0 ,  \\
				\p_i {b}_{j} + \p_j {b}_{i} &= 0 ,\quad i\neq j ,   \\
				\p_t {b}_{i} + \p_i {a}  &= 0,  \\
				\p_t {a}  &= 0 .  \\
				\end{split}
				\end{equation}
				
				Using the first equationsand direct computation of Lemma 12 in \cite{Guo10}, 
				\begin{equation*}\label{b_form}
				b(t,x) = - \p_{t} c(t) x + \varpi(t) \times x + m(t).
				\end{equation*}	
				From the second equation in (\ref{macro eq}) and the specular BC, 
				\[
				c(t,x) = c_0,\quad b = \varpi(t) \times x + m(t).
				\]
				We split into two cases $\varpi=0$ and $\varpi \neq 0$.  \\
				\textit{Case of $\varpi=0$.} $b(t)=m(t)$ and from sepcular BC, we deduce that
				\begin{equation*}
				b(t)\equiv m(t) \equiv 0.\label{b_vanishing}
				\end{equation*}
				And, from the fourth and the last equations of (\ref{macro eq}), we can derive
				\begin{equation*}\label{a_Phi}
				a(t,x)= a_{0}.
				\end{equation*}
				Since $a(t,x)$ and $c(t,x)$ are constant, from (\ref{coserv_Z}), we derive $a_{0} = c_{0} = 0$, and hence $Z=0$.  \\
				\textit{Case of $\varpi\neq0$.}
				From the specular BC,
				\begin{eqnarray*}
					b(t,x)\cdot {n}(x)  =   \big( \varpi(t)\times x + m(t) \big)\cdot {n}(x)=0.
				\end{eqnarray*}
				Since $m(t)$ is fixed vector for given $t$, we decompose $m(t)$ into the parallel and orthogonal components to $\varpi(t)$ as
				\[
				m(t) = \alpha(t) \varpi(t) - \varpi(t)\times x_0(t).
				\] 
				Then 
				\begin{eqnarray}
				b(t,x)\cdot {n}(x) &=& \big( \varpi(t)\times x + m(t) \big)\cdot {n}(x)\notag  \\
				&=& \big( \varpi(t)\times (x-x_0(t)) \big) \cdot {n}(x) + \alpha(t)\varpi(t)\cdot {n}(x) = 0,\quad\forall x\in\p U. \label{o_a=0}  
				\end{eqnarray}
				
				\noindent Choose $t$ with $\varpi(t) \neq 0$. We can pick $x^{\prime}\in\p U$ such that $\varpi(t)\parallel  {n}(x^{\prime})$. Then the first term of the RHS in (\ref{o_a=0}) is zero. Hence we deduce
				\begin{equation}\label{alpha=0}
				\alpha(t)=0 \ \ \text{and} \ \ b(t,x)= \varpi(t) \times \big(x-x^{0}(t)\big).
				\end{equation}
				This yields
				\begin{equation}\label{axis_0}
				\big( \varpi(t)\times (x-x_0(t)) \big) \cdot {n}(x)=0,\quad\forall x\in\p U .
				\end{equation}
				The equality (\ref{axis_0}) implies that $ U$ is axis-symmetric with the origin $x_0(t)$ and the axis $\varpi(t)$. From (\ref{angular_Z}) and (\ref{alpha=0}),
				\begin{eqnarray*}
					0
					&=& \iint_{ U}|\varpi\times(x-x_0(t))\cdot v|^2 \mu \dd x \dd v  .
				\end{eqnarray*}
				Therefore, we conclude that $b(t,x) \equiv 0$. Then using conservation laws (mass and energy) again, we deduce $Z=0$.  \\
				
				\noindent\textit{Step 9. } Finally we deduce a contradiction from (\ref{Z=1}) and $Z=0$ of Step8. This finishes the proof.  \\		
			\end{proof}	
		
	\section{Linear and Nonlinear decay}
		
		\subsection{Linear $L^{2}$ decay}
		We use coercivity estimate Proposition~\ref{prop_coercivity} to derive exponential linear $L^{2}$ decay of linear boltzmann equation (\ref{lin eq}) with the specular boundary condition. \\
		
		\begin{corollary}\label{decay_U} Assume $f$ solves linear boltzmann equation with the specular BC so that $f$ satisfies Proposition~\ref{prop_coercivity}. Then there exists $\lambda>0$ such that a solution of (\ref{lin eq}) satisfies
			\begin{equation}\label{U_decay}
			\sup_{0 \leq t}e^{\lambda t} 	\|f(t)\|_{2}^{2} \lesssim \| f_{0} \|_{2}^{2}.
			\end{equation}
			
		\end{corollary}
		\begin{proof} Assume that $0 \leq t \leq 1$. From the energy estimate of (\ref{linearized_eqtn}) in a time interval $[0,N]$, 
			\begin{equation*}\label{L2_f} 
			\|f(N)\|_{2}^{2} + \int^{N}_{0} \iint_{ U\times\R^{3}}  f Lf 
			\leq \|f(0)\|_{2}^{2}  .
			\end{equation*}
			
			\noindent From (\ref{linearized_eqtn}), for any $\lambda>0$
			\begin{equation*}\label{eqtn_lamda}
			\big[\p_{t}  + v\cdot \nabla_{x} \big] (e^{\lambda t} f)  + L (e^{ \lambda t}f )
			= \lambda e^{\lambda t} f.
			\end{equation*}
			By the energy estimate,
			\begin{equation}\label{energy_f}  
			\| e^{\lambda t}f(N)\|_{2}^{2} + \underbrace{\int^{N}_{0} \iint_{ U\times\R^{3}}  e^{2\lambda s} f Lf } _{(I)}
			-  
			{
				\lambda \int^{N}_{0} \iint_{ U\times\R^{3}}
				|e^{\lambda s} f(s)|^{2}}
			\leq \|f(0)\|_{2}^{2} 
			. 
			\end{equation}
			
			\noindent Firstly we consider $(I)$ in (\ref{energy_f}). From semi-positiveness of operator $L$, the term $(I)$ in (\ref{energy_f}) is bounded below by 
			\begin{eqnarray*}
				%
				(I)\geq  
				\delta_{L} 
				\int^{N}_{0} \iint_{ U\times\R^{3}} \langle v\rangle | e^{ \lambda s}(\mathbf{I} - \mathbf{P})f|^{2} \geq  \delta_{L} \int^{N}_{0} \| e^{ \lambda s} (\mathbf{I} - \mathbf{P}) f  \|_{\nu}^{2}.
			\end{eqnarray*}
			By time translation, we apply (\ref{coercive}) to obtain
			\begin{equation*} \label{I}
			\begin{split} 
				(I)
				&\geq   \frac{\delta_{L} }{2} \int^{N}_{0} \| e^{ \lambda s}(\mathbf{I} - \mathbf{P}) f \|_{\nu}^{2}  + \frac{\delta_{L} }{2 C} \int^{N}_{0} \|  e^{ \lambda s} \mathbf{P} f \|_{2}^{2}  \geq  \frac{\delta_{L} }{2 C} \int^{N}_{0} \|  e^{ \lambda s}  f \|_{2}^{2}.
			\end{split}
			\end{equation*}
		
			\noindent Therefore, we derive 
			\begin{equation}\label{energy_N0} 
			e^{2\lambda N}\| f(N)\|_{2}^{2} 
			+ \Big(
			\frac{\delta_{L} }{2C} - \lambda
			\Big)
			\int^{N}_{0} \| e^{ \lambda s} f \|_{2}^{2} 
			\leq \| f (0)\|_{2}^{2}. 
			\end{equation}
			
		\noindent On the other hand, from the energy estimate of (\ref{linearized_eqtn}) in a time interval $[N,t]$, using semi-positiveness of $L$, we have 
			\begin{equation}\label{energy_tN}
			\| f (t) \|_{2}^{2}  \leq \| f(N) \|_{2}^{2}.
			\end{equation}
			
		\noindent Finally choosing $\lambda \ll 1$, from (\ref{energy_N0}) and (\ref{energy_tN}), we conclude that 
			\begin{equation*}
			e^{\lambda t } \| f(t) \|_{2}^{2} = e^{\lambda (t-N)} e^{\lambda N }\| f(N) \|_{2}^{2}
			\leq 2 \| f(0) \|_{2}^{2},
			\end{equation*}
			and obtain (\ref{U_decay}). \\  
		\end{proof}

		\subsection{Nonlinear $L^{\infty}$ decay}
		We use $L^{\infty}-L^{2}$ bootstrap form (\ref{h infty est}), Duhamel's principle, and Corollary~\ref{decay_U} to derive nonlinear $L^{\infty}$ decay.  \\		
		\begin{proof} [\textbf{Proof of Theorem~\ref{theorem_decay}}] 
		From (\ref{h infty est}),
		\begin{equation*} 
		\begin{split}
		\sup_{s\in[T,t]} \|h(s)\|_\infty & \lesssim \ e^{-\nu_{0} (t-T) } \|h (T)\|_{\infty} +  \int_{T}^{t} \|f (s)\|_2 \dd s  .
		\end{split}
		\end{equation*}
		
		We assume that $m \leq t < m+1$ and define $\lambda^{*}:=\min\{ \nu_{0}, \lambda\}$, where $\lambda$ is some constant from Corollary~\ref{decay_U}. We use (\ref{h infty est}) repeatedly for each time step, $[k,k+1), \ k\in\mathbb{N}$ and Corollary~\ref{decay_U} to perform $L^{2}-L^{\infty}$ bootstrap,
		\begin{equation*} \label{h infty decay}
		\begin{split}
		\|h(t)\|_{\infty} 
		&\lesssim e^{-m \nu_{0} } \|h(0)\|_{\infty} + \sum_{k=0}^{m-1} e^{-k \nu_{0} } \int_{m-1-k}^{m-k} \|f(s)\| \dd s  \\
		&\lesssim e^{-m \nu_{0} } \|h(0)\|_{\infty} + \sum_{k=0}^{m-1} e^{-k  \nu_{0} } \int_{m-1-k}^{m-k} e^{-\lambda(m-1-k)} \|f(0)\| \dd s  \lesssim  e^{- \lambda^{*}  t } \|h(0)\|_{\infty} . \\
		\end{split}
		\end{equation*} 
		
		For nonlinear problem from Duhamel principle,
		\begin{equation} \label{Duhamel}
		\begin{split}
		h  &:= U(t)h_0 + \int_{0}^{t} U(t-s) w  \Gamma(\frac{h }{w},\frac{h }{w}) (s)  \dd s ,  \\
		\|h (t) \|_{\infty} &\lesssim e^{- \lambda^{*} t } \|h(0)\|_{\infty} + \Big\| \int_{0}^{t}  U(t-s) w \Gamma\big( \frac{h }{w}, \frac{h }{w} \big)(s) \dd s \Big\|_{\infty} , 
		\end{split}
		\end{equation}
		where $U(t)$ is linear solver for linearized Boltzmann equation. Inspired by \cite{Guo10}, we use Duhamel's principle again:
		\begin{equation*} \label{double duhamel}
		U(t-s) = G(t-s) + \int_s^t G(t-s_1) K_{w}U(s_1-s) \dd s_1,
		\end{equation*}
		where $G(t)$ is linear solver for the system
		\begin{equation*} \label{step 2 eq}
		\begin{split}
		&\p_t h + v\cdot\nabla_x h  +  \nu h = 0,\quad \text{and}\quad |G(t)h_0| \leq e^{-\nu_{0}t} |h_0|.
		\end{split}
		\end{equation*}
		
		\noindent For the last term in (\ref{Duhamel}),
		\begin{equation*} \label{nu increase}
		\begin{split}
		& \Big\| \int_{0}^{t}  U(t-s) w \Gamma\big( \frac{h }{w}, \frac{h }{w} \big)(s) \dd s \Big\|_{\infty}   \\
		&\leq  \Big\| \int_{0}^{t}  G(t-s) w  \Gamma\big( \frac{h }{w}, \frac{h }{w} \big)(s) \dd s \Big\|_{\infty}   +   \Big\| \int_{0}^{t}  \int_{s}^{t} G(t-s_1) K_{w}  U(s_1-s) w \Gamma\big( \frac{h }{w}, \frac{h }{w} \big)(s) \dd s_1 \dd s  \Big\|_{\infty}   \\
		&\leq 
		C e^{- \lambda^{*} t} \Big(\sup_{0\leq s\leq \infty} e^{ \lambda^{*} s} \|h (s)\|_{\infty} \Big)^{2}.
		\end{split}
		\end{equation*}	
		
		\noindent Therefore, for sufficiently small $\|h_0\|_{\infty} \ll 1$, we have uniform bound
		\begin{equation*} \label{uniform bound} 
		\sup_{0\leq t\leq \infty} e^{\lambda^{*} t } \|h(t)\|_\infty \ll 1,
		\end{equation*} 
		hence we get global decay and uniqueness. Also note that positivity of $F$ is standard by linear solvability and solution sequence $F^{\ell}$:
		\begin{equation}\notag
		\begin{split}
		&\p_{t } F^{\ell+1} + v\cdot \nabla F^{\ell+1} = Q_{+} (F^{\ell},F^{\ell}) - \nu(F^{\ell}) F^{\ell+1},  \ \ F|_{t=0} = F_{0}, \\
		&F^{\ell+1} (t,x,v) = F^{\ell+1} (t,x,R_{x}v) \ \ \ \text{on}  \ \  \p U.
		\end{split}
		\end{equation} 
		From $F_{0} \geq 0$ and $ F^{\ell} \geq 0$, we have $F^{\ell+1} \geq 0$.
		
		\end{proof}

		\section{Appendix: Example of sticky grazing point}
		
		Let us consider backward in times trajectories which start from $(1,1)$ with velocity $v$ between $(1,1)$ and $(1,1+\delta)$, $0 < \delta \ll 1$. Then all the trajectories are part of set of rays
		\[
			\{ (x,y) : y = (1 + \delta)(x-1) + 1, \ 0 < \delta \leq \varepsilon \ll 1  \}.
		\]  	
		We consider the trajectories bounce on the curve $f(x) = \frac{1}{2}x^{2}$. When $\delta = 0$, trajectory bounce on $(0,0)$ with collision angle $\frac{\pi}{4}$. When $0 < \delta \ll 1$, bouncing point on $f(x) = \frac{1}{2}x^{2}$ is 
		\[
			(\delta_{*}, \frac{1}{2}\delta_{*}^{2}),\quad\text{where}\quad \delta_{*} = (1 + \delta) - \sqrt{(1+\delta)^{2} - 2\delta}.
		\]
		Using the specular BC, bounced trajectory with $v^{1}$ direction is part of set of rays
		\[
		\{ (x,y) : y = L(\delta)(x - \delta_{*}) + \frac{1}{2}\delta_{*}^{2} \},\quad L(\delta) = \frac{ (1+\delta)(1+\delta_{*}^{2}) - 2\sqrt{1+\delta^{2}} }{ 1 + \delta_{*}^{2} + 2\delta_{*}\sqrt{1+\delta^{2}} }.
		\] 
		We parametrize convex grazing boundary with parameter $\delta$, 
		\[
			\big( X(\delta), Y(\delta) \big),\quad X(0) = -Y(0) < 0.
		\]
		Considering tangential line on $X(\delta), Y(\delta))$, it is easy to derive two conditions from concave grazing.
		\begin{equation} \label{two eq}
		\begin{split}
			\frac{Y^{\prime}(\delta)}{X^{\prime}(\delta)} &= L(\delta),  \\
			-L(\delta)\delta_{*} + \frac{1}{2}\delta_{*}^{2} &= -\frac{Y^{\prime}(\delta)}{X^{\prime}(\delta)} X(\delta) + Y(\delta).
		\end{split}
		\end{equation}		
		We differentiate second equation and combine with first equation to get 
		\begin{equation} \label{X form}
		\begin{split}
			\frac{d}{d\delta}\big( -L(\delta)\delta_{*} + \frac{1}{2}\delta_{*}^{2} \big) &= -L^{\prime}(\delta)X(\delta) - L(\delta) X^{\prime}(\delta) + Y^{\prime}(\delta)  \\
			&= -L^{\prime}(\delta)X(\delta).
		\end{split}
		\end{equation}
		It is easy to check $L^{\prime} > 0$ locally $0< \delta \ll 1$. (\ref{X form}) gives $X(\delta)$ and this is analytic from analyticity of $L(\delta)$ and $\delta_{*}^{2}$, $X(\delta)$ is analytic function of $\delta$ for local $0<\delta \ll 1 $. Using the first equation of (\ref{two eq}), we obtain ODE for $Y(\delta)$ with $Y(0)=-X(0)$. Since $X(\delta)$ is analytic, $Y(\delta)$ is also analytic. Moreover, we can check concavity of $\big( X(\delta), Y(\delta) \big)$ by 
		\[
			\frac{d}{d\delta} \Big( \frac{Y^{\prime}(\delta)}{X^{\prime}(\delta)} \Big) = L^{\prime}(\delta) > 0.
		\]
		
		\noindent\textbf{Acknowledgements.} The authors thank Yan Guo for stimulating discussions. Their research is supported in part by $\text{NSF-DMS }1501031$, WARF, and the Herchel Smith fund. They thank KAIST Center for Mathematical Challenges and ICERM for the kind hospitality. \\

		\bibliographystyle{plain}

\begin{thebibliography}{99}
			
			\bibitem{CIP} Cercignani, C., Illner, R., and Pulvirenti, M. \textit{The
				mathematical theory of dilute gases}. Applied Mathematical Sciences, 106.
			Springer-Verlag, New York, 1994.
			
			
			\bibitem{CM} Chernov, N. and Markarian, R. \textit{Chaotic Billiards. Mathematical Surveys and Monographs}, 127. American Mathematical Society, Providence, RI, 2006.
			
			\bibitem{DV} Devillettes, L. and Villani, C. On the trend to global equilibrium for spatially inhomogeneous kinetic systems: the
			Boltzmann equation. \textit{Invent. Math.} \textbf{159} (2005), no. 2, 245--316. 
			
			\bibitem{dC} Do Carmo, M. \textit{Differential Geometry of Curves and Surfaces}. Prentice Hall, 1976.
			
			\bibitem{EGKM} Esposito, R., Guo, Y., Kim, C., and Marra, R. Non-Isothermal
			Boundary in the Boltzmann Theory and Fourier Law.  \textit{Comm. Math. Phys.} \textbf{323} (2003), no. 1, 177--239. 
			
			\bibitem{EGKM2} Esposito, R., Guo, Y., Kim, C., and Marra, R. Stationary solutions to the Boltzmann equation in the Hydrodynamic limit, submitted.
			
			
			
			
			
			\bibitem{GKTT1} Guo, Y., Kim, C., Tonon, D., and Trescases, A. Regularity of the Boltzmann Equation in Convex Domains. \textit{Invent. Math.} \textbf{207} (2017), no. 1, 115--290.
			
			
			\bibitem{GKTT2} Guo, Y., Kim, C., Tonon, D., and Trescases, A. BV-Regularity of the Boltzmann Equation in Non-convex Domains.  \textit{Arch. Rational Mech. Anal.} \textbf{220} (2016), no. 3, 1045--1093. 
			
			
			
			\bibitem{Guo10} Guo, Y. Decay and Continuity of Boltzmann Equation in
			Bounded Domains. \textit{Arch. Rational Mech. Anal.} \textbf{197} (2010), no. 3, 713--809.
			
			
			
			
			
			
			
			
			
			
			
			
			\bibitem{Strange} Halpern, B.: Strange billiard tables. \textit{Tran. Amer. Math. Soc.} \textbf{232} (1977), 297--305.
			
			
			\bibitem{Kim11} Kim, C. Formation and propagation of discontinuity for Boltzmann equation in non-convex domains. \textit{Comm. Math. Phys.} \textbf{308} (2011), no. 3, 641--701.
			
			\bibitem{K2} Kim, C. Boltzmann equation with a large external field. \textit{Comm. PDE.} \textbf{39} (2014), no. 8, 1393--1423. 
			
			
			\bibitem{KY} Kim, C. and Yun, S. The boltzmann equation near a rotational local maxwellian. \textit{SIAM J. Math. Anal.} \textbf{44} (2012), no. 4, 2560--2598.
			
			
			
			\bibitem{KimLee} Kim, C. and Lee, D. The Boltzmann equation with specular boundary condition in convex domains. \textit{Comm. Pure Appl. Math.} accepted. 
			
			\bibitem{SA} Shizuta, Y. and Asano, K. Global solutions of the Boltzmann equation in a bounded convex domain. \textit{Proc. Japan Acad.} \textbf{53A} (1977), 3--5.

			
			
			
			
			
			
			
		\end{thebibliography}

\end{document}